\documentclass[11pt,letterpaper]{amsart}

\usepackage[utf8]{inputenc}
\usepackage{enumerate}%
\usepackage{amsmath,amssymb}
\usepackage{cite}
\usepackage{comment}
\usepackage{amsfonts}
\usepackage{mathrsfs}
\usepackage{galois}
\usepackage[toc,page]{appendix}
\usepackage{xcolor}
\usepackage[colorlinks, linkcolor = blue, anchorcolor = blue, citecolor = blue]{hyperref}

\theoremstyle{definition}
\newtheorem{Def}{Definition}[section]
\newtheorem{Rem}[Def]{Remark}
\newtheorem*{Que}{Question}

\theoremstyle{plain}
\newtheorem{Thm}[Def]{Theorem}
\newtheorem*{Conj}{Conjecture}

\newtheorem{Lem}[Def]{Lemma}
\newtheorem{Prop}[Def]{Proposition}
\newtheorem{Cor}[Def]{Corollary}
\newtheorem{claim}{Claim}[section]

\numberwithin{equation}{section}

\newcommand{\AAa}{\mathbb{A}}
\newcommand{\BB}{\mathbb{B}}

\newcommand{\NN}{\mathbb{N}}

\newcommand{\RR}{\mathbb{R}}
\newcommand{\SSp}{\mathbb{S}}

\newcommand{\ZZ}{\mathbb{Z}}

\newcommand{\cA}{\mathcal{A}}
\newcommand{\cB}{\mathcal{B}}
\newcommand{\cC}{\mathcal{C}}

\newcommand{\cE}{\mathcal{E}}
\newcommand{\cF}{\mathcal{F}}
\newcommand{\cG}{\mathcal{G}}

\newcommand{\cI}{\mathcal{I}}

\newcommand{\cK}{\mathcal{K}}
\newcommand{\cL}{\mathcal{L}}
\newcommand{\cM}{\mathcal{M}}
\newcommand{\cN}{\mathcal{N}}

\newcommand{\cR}{\mathcal{R}}
\newcommand{\cS}{\mathcal{S}}

\newcommand{\cV}{\mathcal{V}}

\newcommand{\cZ}{\mathcal{Z}}

\newcommand{\scB}{\mathscr{B}}
\newcommand{\scC}{\mathscr{C}}

\newcommand{\scF}{\mathscr{F}}

\newcommand{\scH}{\mathscr{H}}

\newcommand{\scM}{\mathscr{M}}

\newcommand{\bfF}{\textbf{F}}

\newcommand{\mbfC}{\mathbf{C}}
\newcommand{\mbfD}{\mathbf{D}}

\newcommand{\mbfF}{\mathbf{F}}

\newcommand{\mbfI}{\mathbf{I}}

\newcommand{\mbfO}{\mathbf{O}}

\newcommand{\SCap}{\mathbf{SCap}}
\newcommand{\SCAP}{\mathbf{SCap}^*}

\newcommand{\IV}{\mathcal{IV}}

\newcommand{\set}[1]{\left\{#1\right\}}

\newcommand{\Ker}{\operatorname{Ker}}
\newcommand{\spt}{\operatorname{spt}}
\newcommand{\Clos}{\operatorname{Clos}}
\newcommand{\Int}{\operatorname{Int}}
\newcommand{\ind}{\operatorname{index}}
\newcommand{\vol}{\operatorname{vol}}
\newcommand{\diverg}{\operatorname{div}}
\newcommand{\graph}{\operatorname{graph}}
\newcommand{\Sym}{\operatorname{Sym}}
\newcommand{\injrad}{\operatorname{injrad}}
\newcommand{\dist}{\operatorname{dist}}
\newcommand{\Reg}{\operatorname{Reg}}
\newcommand{\Sing}{\operatorname{Sing}}
\newcommand{\Ric}{\operatorname{Ric}}
\newcommand{\Rm}{\operatorname{Rm}}

\newcommand{\Crit}{\mathrm{Crit}}

\newcommand{\gap}{\mathrm{gap}}
\newcommand{\eucl}{\mathrm{eucl}}

\title{Minimal hypersurfaces for generic metrics in~dimension~$8$}
\author{Yangyang Li}
\address{Department of Mathematics, Princeton University, Fine Hall, 304 Washington Road, Princeton, NJ 08540, USA}
\email{yl15@math.princeton.edu}
\author{Zhihan Wang}
\address{Department of Mathematics, Princeton University, Fine Hall, 304 Washington Road, Princeton, NJ 08540, USA}
\email{zhihanw@math.princeton.edu}
\date{}
\thanks{Y. Li was partially supported by NSF-DMS-1811840}

\begin{document}
\bibliographystyle{abbrvalpha}

\begin{abstract}
    We show that in an $8$-dimensional closed Riemmanian manifold with $C^\infty$-generic metrics, every minimal hypersurface is smooth and nondegenerate. This confirms a full generic regularity conjecture of minimal hypersurfaces in dimension eight. This also enables us to generalize many generic geometric properties of (Almgren-Pitts) min-max minimal hypersurfaces, previously only known in low dimensions, to dimension eight.
\end{abstract}
\maketitle
\tableofcontents
\section{Introduction}
    In Riemannian geometry, the theory of minimal surfaces has been the center since the 1760s, the era of L. Euler and J.-L. Lagrange. In particular, the well-known Plateau's problem on the existence of minimal surfaces with a given boundary had been open for over a century and was solved independently by J. Douglas \cite{douglas_solutionproblemplateau_1931} and T. Rad\'o \cite{rado_plateauproblem_1930} around 1930. Later, to solve the extended problem, especially higher-dimensional analogue, mathematicians have invented and developed geometric measure theory (GMT), which studies the geometric properties of sets via measure theory. The theory is successful in the sense that the Plateau's problem can be solved within a nontrivial homology class, if one accepts that solutions are not necessarily smooth submanifold, but ``area-minimizing rectifiable currents'' or ``stationary integral varifolds''. Moreover, the interior regularity theory also determines that the supports of these objects are indeed smooth if the ambient manifold has dimension between $3$ and $7$, and have singular sets of codimension at least $7$ in higher dimensions \cite{federerNormalIntegralCurrents1960, simonsMinimalVarietiesRiemannian1968, federerSingularSetsArea1970, bombieriMinimalConesBernstein1969, schoen_regularity_1981, wickramasekeraGeneralRegularityTheory2014, cheegerQuantitativeStratificationRegularity2013, naberSingularStructureRegularity2020}. For simplicity, we will abuse the notion ``minimal hypersurface'' to denote stationary integral varifolds with optimal regularity as above and their corresponding currents or supports.
    
    Built on these interior regularity results, the Almgren-Pitts min-max construction orignated from the work of F.J. Almgren Jr. \cite{almgren_homotopy_1962, almgren_theory_1965} and J. Pitts \cite{pitts_existence_1981} has resolved various problems related to minimal hypersurfaces in a closed Riemannian manifold $(M^{n+1}, g)$, especially in the recent decade due to the new ideas and techniques since the work of F.C. Marques and A. Neves \cite{marques_min-max_2014,marques_morse_2016,marques_existence_2017,marques_morse_2021}. Analogous to finding eigenvalues by the Reighley quotients, the theory identifies some critical values of the area functional via a variational method on multi-parameter sweepouts of (possibly singular) hypersurfaces,
    \[
        \set{\omega_p(M, g)}_{p \in \mathbb{N}^+}\,,
    \]
    which is called the volume spectra of $(M, g)$. Each $\omega_p(M, g)$ ($p$-width) can be realized by some stationary integral varifolds, and as in the previous regularity result, their supports are smooth in low dimensions and possibly singular in higher dimensions. Moreover, in the case where the closed manifold $M^{n+1} (3 \leq n+1 \leq 7)$ has a $C^\infty$ generic metric, more fascinating geometric properties of these min-max minimal hypersurfaces have been discovered, such as abundance, multiplicity one, density, Morse index estimates and so on \cite{irie_density_2018, marques_equidistribution_2019, marques_morse_2021, zhou_multiplicity_2019}. See Section \ref{Sec_Appl} for more details.
    
    These generic results heavily relies on the following structure theorem of minimal hypersurfaces by B. White \cite{white_space_1991,white_bumpy_2017}. 
    
    \begin{Thm}(White's Structure theorem)
        Given a closed Riemannian manifold $M$, an integer $k\geq 4$, and $\alpha\in (0, 1)$, let
        \begin{itemize}
            \item $\cG^{k, \alpha}(M)$ be the space of $C^{k,\alpha}$ Riemannian metrics on $M$;
            \item $\cM^{k, \alpha}(M)$ be the space of pairs $(g, \Sigma)$, where $g\in \cG^{k, \alpha}(M)$ and $\Sigma$ is a smooth minimal hypersurface in $(M, g)$.
        \end{itemize}
        Then the projection map $\Pi: \cM^{k, \alpha}(M) \rightarrow \cG^{k, \alpha}(M)$ to the first factor is a Fredholm map with Fredholm index zero.
        
        In particular, the regular values of $\Pi$, called \textbf{bumpy metrics}, form a generic subset $\cG^{k, \alpha}_{\mathrm{bumpy}}(M)$ of $\cG^{k, \alpha}(M)$. For any $g \in \cG^{k, \alpha}_{\mathrm{bumpy}}(M)$, if $\Sigma$ is a smooth $g$-minimal hypersurface, then $\Sigma$ is non-degenerate.
    \end{Thm}
    
    Intuitively, for bumpy metrics, the area functional looks like a Morse function, and the critical points of the area functional are ``discrete'', if there are ``only'' smooth minimal hypersurfaces. Therefore, in higher dimensions ($n+1 \geq 8$), the lack of generalized bumpy metrics for singular minimal hypersurfaces obstructs direct extension of the arguments for the geometric properties of min-max minimal hypersurfaces.
    
    Although previous works of the first author \cite{li_existence_2019, li_improvedMorseindex_2020} have shown that some of these properties can be obtained without the usage of bumpy metrics, in the current project, we are more interested in the other aspect of this generalization, i.e., the question of generic regularity. In \cite{10.2307/j.ctt1bd6kkq.37}, Yau has conjectured that a perturbation of the metric can eliminate the singularities of a given area minimizing hypersurface in its homology class. A more optimistic conjecture is the generic regularity in full as following.
    
    \begin{Conj}[Full generic regularity I]
        For a closed smooth manifold $M^{n+1}$ endowed with a (Baire) $C^\infty$-generic Riemannian metric $g$, every $g$-area homology minimizer is smooth.
    \end{Conj}
    
    Since the min-max construction doesn't produce area minimizers generally, it is natural to extend this conjecture to min-max minimal hypersurfaces, i.e., those generated from the min-max theory.
    
    \begin{Conj}[Full generic regularity II]
        For a closed smooth manifold $M^{n+1}$ endowed with a (Baire) $C^\infty$-generic Riemannian metric $g$, every min-max $g$-minimal hypersurface is smooth.
    \end{Conj}
    
    In the following, we always assume that $M^8$ is a closed smooth manifold of dimension $8$. Given a Riemannian metric $g$, an embedded smooth hypersurface $\Sigma\subset (M^{n+1}, g)$ is called a \textbf{locally stable minimal hypersurface (LSMH)}, if 
    \[ 
        \scH^n(\Sigma)<+\infty\,, \quad\scH^{n-2}(\bar{\Sigma}\setminus \Sigma) = 0\,,
    \]
    and $\forall p\in M$, there's a neighborhood $U_p$ in which $\Sigma$ is stable.
    
    The goal of this paper is to prove the following theorem.
    
    \begin{Thm} \label{Thm_Main}
        Let $M$ be a closed smooth manifold of dimension $8$, $k\geq 4$ or $k = \infty$, $\alpha \in (0, 1)$. Then for $C^{k,\alpha}$ generic metric $g$ on $M$, every embedded locally stable minimal hypersurface is regular and nondegenerate. 
    \end{Thm}
    \begin{Rem}
        This can be viewed as a weak generalization of B. White's structure theorem. In fact, we only need to verify the case where $k$ is finite. As in \cite{white_bumpy_2017}, we can use similar arguments therein or \cite[Section~7.1]{ambrozio_compactness_2017} to extend our result to the case $k = \infty$.
    \end{Rem}
    
    Note that every minimal hypersurfaces obtained by Almgren-Pitts min-max theory (or Allen-Cahn min-max theory) is locally stable. Therefore, we confirm the full generic regularity conjecture II in dimension eight.
    
    \begin{Cor}
        Let $M$ be a closed smooth manifold of dimension $8$, $k\geq 4$ or $k = \infty$. Then for $C^{k,\alpha}$ generic metric $g$ on $M$, every minimal hypersurface generated from Almgren-Pitts min-max theory (or Allen-Cahn min-max theory) is regular and nondegenerate.
    \end{Cor}
    
    Historically, in dimension $8$, Yau's generic smoothing conjecture of area minimizing hypersurface was affirmed by N. Smale  \cite{smaleGenericRegularityHomologically1993} in 1993, using the Hardt-Simon analysis of area-minimizing hypercones in $\mathbb{R}^8$ \cite{hardtAreaMinimizingHypersurfaces1985} (Mazzeo-Smale \cite{mazzeo_perturbing_1994} also extended this to higher dimensions with technical assumptions on the singularities in 1994). Recently, Chodosh-Liokumovich-Spolaor \cite{chodoshSingularBehaviorGeneric2020} showed that in the case of positive Ricci curvature, the first-width min-max minimal hypersurfaces can be smoothened via a metric perturbation as well. These have implied that in a generic closed $8$-dimensional Riemannain manifold with suitable topological assumption ($H_7(M^8) \neq 0$) or geometric assumption ($\mathrm{Ric}(M^8) > 0$) as above, there exists at least one smooth minimal hypersurface. Inspired by their work, we \cite{liwang2020generic} removed those assumptions and proved the generic existence of a smooth minimal hypersurface in $M^8$.
    
    The main theorem herein genuinely improves our previous work on the generic existence. Unlike our previous proof, the approach herein is more systematic, self-contained, and independent of the min-max theory in the course.
    
    \subsection{Outline of the proof}

        The main idea is to control and even reduce the \textit{singular capacity} (See Section \ref{Sec_Prelim}) of minimal hypersurfaces, an integer-value geometric quantity which counts singular points with weights. 

        Given $(M^8, g)$, for any fixed $\Lambda > 0$ and $I > 0$, we would consider all the minimal hypersurfaces with area bounded above by $\Lambda$ and Morse index bounded above by $\Lambda$. One can verify that the singular capcity over them has a uniform upper bound, and the maximum value can only be realized by finitely many minimal hypersurfaces.

        Then, we appeal to a generic perturbation of the metric $g$, so that the singular capacity of the finite set of minimal hypersurfaces mentioned above would decrease at least by $1$. Once this is established, by a simple induction argument, one will be able to find a metric where all the minimal hypersurfaces with bounded area and bounded index have $0$ singular capacity, i.e., are smooth.

        To prove the existence of such a generic perturbation, we will investigate the relation between the singular behavior of a sequence of minimal hypersurfaces and the induced Jacobi field on their limiting minimal hypersurface. Roughly speaking, we shall show that if the singular capacity of the limiting one isn't strictly greater than that of the sequence, then the induced Jacobi field lives in a subspace, i.e., the space of functions of \textit{slower growth}. In view of White's Structure Theorem, this means the limiting minimal hypersurface is \textit{degenerate} in certain sense and is expected to not happen if the metric is chosen generic a priori (see Corollary \ref{Cor_Canon Neighb_Global generic semi-nondeg} for the precise statement of \textit{generic semi-nondegeneracy}). A further generic perturbation of metrics is then indeed possible.

        Finally, by taking $\Lambda \to \infty$ and $I \to \infty$, one can see that there exist Riemannian metrics where every minimal hypersurfaces is smooth, and furthermore, they are dense in the space of all Riemannian metrics.

    \subsection{Organiztion of the paper}
        In Section \ref{Sec_Prelim}, we give a definition of locally stable minimal hypersurfaces, and recall the properties of their singularities using the analysis of \cite{WangZH20_Deform}. We also review the notions of cone decomposition discussed in Edelen's work \cite{edelen_degeneration_2021} and singular capacity introduced in our previous work \cite{liwang2020generic}. Note that the notion of cone decomposition will only be used in Part \ref{Part_Tech}, so the first time reader may ignore it until Section \ref{Sec_Analysis on SMC}.
        
        We prove the full generic regularity conjecture and its applications in min-max theory in Part \ref{Part_Generic Reg} (Section 3 - 5).
        
        In Section \ref{Sec_Canonical Neighb}, we introduce the notion ``canonical (pseudo-)neighborhood'' and the related generic semi-nondegeneracy results (See the defintion in Section \ref{Sec_Prelim}). This can be viewed as a first step generalization of White's Structure Theorem to singular minimal hypersurfaces. The proofs of these are very involved, so are postponed to Part \ref{Part_Tech}.
        
        In Section \ref{Sec_Pf of Main Thm}, with the genericity results in the previous section and the aid of singular capacity, we confirm the full regularity conjecture. 

        In Section \ref{Sec_Appl}, we show some applications of the main result in Almgren-Pitts theory.

        All the technical work will be included in Part \ref{Part_Tech} (Section 6 - 9).

        In Section \ref{Sec_Analysis on SMC}, we focus on stable minimal hypercones. We investigate the asymptotic rates of a Jacobi function and the graphical property of nearby minimal hypersurfaces, which would be useful in the later sections.

        In Section \ref{Sec_Growth est}, by estimating growth rate of graphical functions, we show that a sequence of pairs of minimal hypersufaces close enough (in the $\cL^{k,\alpha}$ sense) would induce a Jacobi field of slower growth.

        In Section \ref{Sec:Gen_Semi-nondeg}, we unwrap the Sard-Smale type arguments in White's structure theorem to show that semi-nondegenracy is a generic property for canonical (pseudo-)neighborhood, so as to prove Theorem \ref{Thm_Generic Semi-nondeg_Loc}.

        In Section \ref{Sec:Count_Decomp}, inspired by Edelen's cone decomposition, we introduce a two-step decomposition scheme to show that countably many ``canonical (pseudo-)neighborhoods'' could cover the whole space.

        In Appendix \ref{Sec_App_Proof Growth Rate Mon}, we finish the proof of Lemma \ref{Lem_Ana on SMC_Growth Rate Monoton}.

        In Appendix \ref{Sec_App_Geom of Minimal Graph}, we explicitly write down various estimates for minimal graphs.

        In Appendix \ref{Sec_App_Sing Cap}, we affirm a question on singular sheeting property asked by Ilmanen \cite{ilmanen_strong_1996} in dimension $8$. As an application, we extend our definition of singular capacity in \cite{liwang2020generic} to one-sided minimal hypersurfaces as well.

        In Appendix \ref{Sec_App_Reg Deform Thm}, we review some important results from \cite{WangZH20_Deform}.
        
\section*{Acknowledgements}
    We are grateful to our advisor Fernando Marques for his support. We would also like to thank Xin Zhou's interest in this work and inspiring discussions.
    
\section{Preliminaries} \label{Sec_Prelim}
    
    Let $(M^{n+1}, g)$ be a closed Riemannian manifold where $n +1 = 8$, and an integer $k\geq 4$ and a real number $\alpha\in (0, 1)$. We first list some notations in geometric measure theory. Interested readers may refer to \cite{simonLecturesGeometricMeasure1984}.
    
    \begin{itemize}
        \item $\cV_n(M)$: the space of $n$-dimensional varifolds in $M$;
        \item $\mbfF$: the metric on $\cV_n(M)$ defined by 
              \[
                  \mbfF(V, W) = \sup\set{\|V\|(f) - \|W\|(f): f \in C^1(M), |f|\leq 1, |Df|\leq 1}\,.
              \]
              Moreover, $\mbfF_U(V, W) := \mbfF(V_1\llcorner U, V_2\llcorner U)$ for any open subset $U$;        \item $\mbfI_n(M;\ZZ_2)$: the space of $n$-dimensional mod $2$ flat chains in $M$ endowed with the $\cF$ metric;
        \item $\IV_n(M)$: the space of $n$-dimensional integral varifolds in $M$;
        \item $\cR(M)$: the space of $n$-dimensional integral varifold in $M$ whose support is regular away from a closed singular set of Hausdorff dimension at most $n - 7$;
        \item $|\Sigma|$: the associated integral varifold for a hypersurface $\Sigma$;
        \item $\|V\|$: the associated Radon measure for $V \in \IV_n(M)$;
        \item $\theta_{V,g}(a, r)$: the density ratio of a varifold $V$ in the geodesic ball $B^g(a, r) \subset (M, g)$. If $\lim_{r \rightarrow 0+} \theta_{V,g}(a, r)$ exists, we define the density of $V$ at $a$ as
              \[
                  \theta_{V,g}(a) := \lim_{r \rightarrow 0+} \theta_{V, g}(a, r)\,.
              \]
              Here, the subscript $g$ will be omitted if there is no ambiguity;
        \item $\mathcal{C}$: the collection of regular stable minimal hypercones $\mathbf{C}^7$ in $\mathbb{R}^8$;
        \item $\mathcal{C}_{\Lambda}$: the subset of $\mathcal{C}$ with density at the origin bounded above by $\Lambda$;
        \item $\cG^{k, \alpha}(M)$ be the space of $C^{k, \alpha}$-Riemannian metrics on $M$, which is naturally a Banach manifold;
        \item $\cM^{k, \alpha}(M)$ be the space of pairs $(g, \Sigma)$, where $g\in \cG^{k, \alpha}(M)$ and $\Sigma$ is a connected LSMH, two-sided or one-sided, in $(M, g)$; This space is endowed with the topology induced by $C^{k, \alpha}$-convergence in the $g$ factor and multiplicity $1$ $\mbfF$-convergence in the $\Sigma$ factor;
        \item $\Pi: \cM^{k, \alpha}(M) \to \cG^{k, \alpha}(M)$ be the projection onto the first factor;
        \item $\dist_H$: Hausdorff distance function;
        \item $\eta_{p, r}: M \to \mathbb{R}^{n+1}$ defined by $\eta_{p, r}(x):= \exp_{p}^{-1}(x)/r$ where we identify $T_p M$ with $\mathbb{R}^{n+1}$.
    \end{itemize}
    
    Let $\Sigma \subset (M^{n+1}, g)$ be a $C^2$ (one-sided or two-sided) hypersurface.  Throughout this article, every such hypersurface $\Sigma$ is assumed to have 
    \[
        \scH^n(\Sigma)<+\infty\,,  \quad \scH^{n-2}(\overline{\Sigma}\setminus \Sigma) = 0\,.
    \]
    Let 
    \begin{itemize}
        \item $\Reg(\Sigma):= \{x\in \overline\Sigma: \overline\Sigma \text{ is an embedded }C^2 \text{ hypersurface near }x\}$;
        \item $\Sing(\Sigma):= \overline\Sigma\setminus \Reg(\Sigma)$.
    \end{itemize}
    By modifying $\Sigma$ up to a measure zero set, we may assume without loss of generality that $\Sigma = \Reg(\Sigma)$. In particular, $\Sigma$ is said to be \textbf{regular} if $\Sing(\Sigma) = \emptyset$.

    For an open subset $U\subset M$, recall $\Sigma$ is called \textbf{two-sided} in $U$ if there exists a globally defined unit normal field $\nu$ on $\Sigma\cap U$; if not, then it's called \textbf{one-sided} in $U$.  Note that since $\scH^{n-2}(\Sing(\Sigma))= 0$, when $U$ is simply connected, $\Sigma$ is always two-sided in $U$.
    
    When $\Sigma$ is connected and one-sided in $M$, there's a unique double cover $\pi: \hat{M} \to M$ such that $\hat{\Sigma} := \pi^{-1}(\Sigma) \subset (\hat{M}, \hat{g} := \pi^* g)$ is a connected two-sided $\ZZ_2$-invariant LSMH, where $\ZZ_2$ acts on $\hat{M}$ by deck transformation. This also induces a natural $\ZZ_2$-action on $C^0_{loc}(\hat{\Sigma})$ by $(\mathbf{-1}, f)\mapsto \check{f}$, where $\check{f}(x):= -f(-x)$.  For any subset $\scF \subset C_{loc}^0(\hat{\Sigma})$, denote by 
    \begin{align}
     \scF^{\ZZ_2} := \{f\in \scF: \check{f} = f\}.  \label{Equ_Pre_Z2 inv functions}
    \end{align}
    
    The Riemannian metric on $\Sigma$ is induced from $(M, g)$. Since we may deal with families of minimal hypersurfaces and metrics, we may use notations $\nabla_{\Sigma, g}$, $\Delta_{\Sigma, g}$, etc. to emphasis that they are operators with respect to $\Sigma, g$, but may omit the subscript if there's no confusion. Locally, by choosing an orthonormal frame $\{e_i\}$ of $\Sigma$, we can let
    \begin{itemize}
        \item $(\vec{A}_{\Sigma, g})_{ij}:= (e_i\cdot \nabla^g e_j)^\perp$ be the second fundamental form of $\Sigma$;
        \item $\vec{H}_{\Sigma, g} := (\sum_i e_i\cdot \nabla^g e_i)^\perp $ be the mean curvature vector of $\Sigma$.
    \end{itemize}
    Recall that $\Sigma$ is minimal in $(M, g)$ if and only if $\vec{H}_{\Sigma, g} = 0$.
    
    Following \cite{schoen_regularity_1981}, a minimal hypersurface $\Sigma\subset (M, g)$ is called \textbf{stable} in an open subset $U\subset M$, if for every $C^2$ family of diffeomorphism $\{\phi_t:U \to U\}_{t\in (-1,1)}$ such that $\phi_t = id_U$ near $\partial U$, we have
        \[
            \frac{d^2}{dt^2}\Big|_{t=0}\scH^n(\phi_t(\Sigma)) \geq 0.   
        \]
        By \cite{schoen_regularity_1981}, since $n+1 = 8$, a stable minimal hypersurfaces in $U$ has isolated singular set. If $\Sigma$ admits a global unit normal field $\nu =\nu_{\Sigma, g}$ in $U$, then $\Sigma$ being stable in $U$ is equivalent to that
        \begin{align}
            Q_{\Sigma, g}(\varphi, \varphi):= \int_\Sigma |\nabla_{\Sigma, g} \varphi|^2 - (|A_{\Sigma, g}|^2 + \Ric_g(\nu, \nu))\varphi^2 \geq 0,\ \ \ \forall \varphi\in C^2_c(\Sigma\cap U).   \label{Equ_Pre_Stablity Ineq for MH}
        \end{align}
        In this case, let 
        \begin{align}
            L_{\Sigma, g}:= \Delta_{\Sigma, g} + |A_{\Sigma, g}|^2 + \Ric_g(\nu, \nu),  \label{Equ_Pre_Def Jac Operator}
        \end{align}
        be the Euler-Lagrange operator associated to $Q_{\Sigma, g}$, called \textbf{Jacobi operator}. Any solution $u\in C^2_{\mathrm{loc}}(\Sigma)$ to $L_{\Sigma, g} u = 0$ is called a \textbf{Jacobi field}.
        
        \begin{Def}
            A $C^2$ hypersurface $\Sigma$ is a \textbf{locally stable minimal hypersurface (LSMH)} if for every $p\in M$, there exists some neighborhood $U_p\ni p$ in which $\Sigma$ is stable. 
        \end{Def}
        
        Following \cite{marques_morse_2016}, the Morse index of a two-sided LSMH $\Sigma \subset (M^{n+1}, g)$ is defined by
        \begin{align}
            \begin{split}
                \ind(\Sigma):= \sup\Big\{\dim \cL:&\ \cL \subset \mathfrak{X}(M) \text{ is a linear subsapce such that }\\
                &\ \frac{d^2}{dt^2}\Big|_{t=0}\scH^n(e^{tX}(\Sigma)) < 0, \ \forall \mathbf{0}\neq X\in \cL \Big\}.    
            \end{split} \label{Equ_Pre_Def of Ind(Sigma)}
        \end{align}
        It is proved in \cite[Corollary 3.7]{WangZH20_Deform} that a two-sided LSMH in a closed manifold always has finite index, and coincides with the index of its Jacobi operator. (See Appendix \ref{Sec_App_Reg Deform Thm} for details.) By convention, we define the index of a one-sided LSMH to be the index of its two-sided double cover.
        
        If $\Sigma$ admits a global unit normal field $\nu$, for any $E\subset \Sigma$ and function $u: E\to \RR$, let \[
            \graph_{\Sigma, g}(u):= \{\exp^g_x(u(x)\cdot \nu(x)): x\in E \},
        \]
        be the graph of $u$ over $\Sigma$.  In Appendix \ref{Sec_App_Geom of Minimal Graph}, we carefully analyze the geometry of $\graph_{\Sigma, g}(u)$.   

    \begin{Def} \label{Def_Reg Scale}
    For every $x\in \Sigma$, we define the \textbf{regularity scale} $r_\cS (x):=r_\cS(x; M, g, \Sigma)$ of $\Sigma$ at $x$ to be the supremum among all $r\in (0, \injrad(x; M, g)/2)$ such that,
        \begin{itemize}
            \item $r^2\|\Rm_g\|_{C^0, B^g_r(x)} + r^3\|\nabla \Rm_g\|_{C^0, B^g_r(x)} \leq 1/10$;
            \item In $T_xM$, \[
                      \frac{1}{r}(\exp^g_x)^{-1}(\Sigma)\cap \BB_1 = \graph_L u \cap \BB_1,    \]
                  for some linear hyperplane $L \subset T_xM $ and $u\in C^3(L)$ with $\|u\|_{C^3}\leq 1/10$.
        \end{itemize}
    \end{Def}
        
    It's easy to check by definition that,
    \begin{enumerate} [(i)]
    \item If $(M, g, \Sigma) \neq (\RR^{n+1}, g_{\eucl}, \RR^n)$, then $r_\cS <+\infty$;
    \item $r_\cS$ is locally Lipschitz on $\Sigma$ and $r_\cS (x)\to 0$ when $x\to x_0\in \overline{\Sigma}\setminus \Sigma$;
    \item $|A_\Sigma|^2 + |\Rm_g| \leq C_n\cdot r_\cS^{-2}$ on $\Sigma$;
    \item When $n+1 = 8$ and $\Sigma$ is a LSMH in a closed manifold $(M, g)$, there exists constant $C(M, g, \Sigma)>1$ such that \[
                      C^{-1}r_\cS \leq \dist_g(\cdot, \Sing(\Sigma)) \leq C r_\cS\,.   
          \]
    \end{enumerate}
    
    For every open subset $U\subset \Sigma$ and every integer $k\geq 0$, let $L^2_{loc}(U)$, $W^{1,2}_{loc}(U)$ and $C^k_{loc}(U)$ to be the space of locally $L^2$, $W^{1,2}$ and $C^k$ functions respectively; Also let $L^2(U)$, $W^{1,2}(U)$ and $C^k(U)$ be the space of functions with finite $L^2$-, $W^{1,2}$- and $C^k$- norms respectively. 
    
    For $\phi\in C^k_{loc}(\Sigma)$ where $k = 0, 1$ or $2$, and a subset $E\subset \Sigma$, we define 
    \begin{align}
        \|\phi\|_{C^k_*, E, g}:= \sup_{x\in E}\sum_{j=0}^k r_\cS(x)^{j-1}|\nabla^j_{\Sigma, g} \phi|(x)\,.  \label{Equ_Pre_C^k_* norm}
    \end{align}
    When $E=\Sigma$, we may omit the subscript $E$ and if there's no ambiguity about choice of metric, then we may also omit the subscript $g$. The reason we use this norm is that it's invariant under rescalings, i.e. for every $\lambda>0$, we have that $\ \graph_{\Sigma, \lambda^2 g}(\lambda \phi) = \graph_{\Sigma, g}(\phi)$ and that, \[
      \|\lambda\phi\|_{C^k_*, E,\lambda^2 g} = \|\phi\|_{C^k_*, E, g}.   \]
    
    \subsection{Stable Minimal Hypercone}  \label{Subsec_Stable MHC}
        Let $\mbfC\subset \RR^{n+1}$ be a stable minimal hypercone, with normal field $\nu$.  Since $n+1 = 8$, $\mbfC$ has isolated singularity at $\mathbf{0}$, or equivalently, $S:= \mbfC\cap \SSp^n$ is a regular minimal hypersurface in $\SSp^n$. Such a cone with only isolated singularity at $\mathbf{0}$ is called a regular cone in higher dimensions. By the Łojasiewicz inequality \cite{simonAsymptoticsClassNonLinear1983}, the densities of stable minimal hypercones in $(\RR^8, g_\eucl)$ form a discrete set,
        \begin{align}
            \{\theta_\mbfC(\mathbf{0}): \mbfC\subset \RR^8 \text{ stable minimal hypercone}\} = \{1=\theta_0 <\theta_1<\theta_2<\dots \nearrow +\infty\}\,. \label{Equ_Pre_Density dicrete for SMC}
        \end{align}
        
        We can parametrize $\mbfC$ by 
        \begin{align}
            (0, +\infty)\times S\to \mbfC,\ \ \ (r, \omega)\mapsto x = r\omega.  \label{Equ_Pre_Parametrize minimal hypercone}
        \end{align}
        Recall by \cite{simonsMinimalVarietiesRiemannian1968}, the Jacobi operator of $\mbfC$ is decomposed under this parametrization as 
        \[
        L_\mbfC = \partial_r^2 + \frac{n-1}{r}\partial_r + \frac{1}{r^2}(\Delta_S + |A_S|^2)\,,
    \]
    where $A_S$ is the second fundamental form of $S$ in $\SSp^n$.
    
    To describe the space of Jacobi fields on $\mbfC$, let 
    \[
        \mu_1< \mu_2\leq \mu_3 \leq \dots \nearrow +\infty
    \]
    be the eigenvalues of $-(\Delta_S + |A_S|^2)$, and
    \[
        \varphi_1, \varphi_2, \dots
    \]
    be the corresponding $L^2(S)$-orthonormal eigenfunctions, where $\varphi_1>0$. By \cite{simonsMinimalVarietiesRiemannian1968}, the stability of $\mbfC$ is equivalent to that $\mu_1 \geq -(n-2)^2/4$. And for every nontrivial regular minimal hypercone $\mbfC\subset \RR^{n+1}$, we have $\mu_1\leq -n+1$, $\mu_2\leq 0$.  By \cite{caffarelliHardtSimon1984}, a general Jacobi field $v\in C^\infty(\mbfC)$ is given by a linear combination of homogeneous Jacobi fields,
    \begin{align}
        v(r, \omega) = \sum_{j\geq 1} (v_j^+(r) + v_j^-(r))\varphi_j(\omega), \label{Equ_Pre_Decomp Jac into homogen Jac field}
    \end{align}
    where 
    \begin{align}
        \begin{split}
            v_j^+(r) & = c_j^+\cdot r^{\gamma_j^+}; \\
            v_j^-(r) & = \begin{cases}
                c_j^- \cdot r^{\gamma_j^-},       & \ \text{ if }\mu_j > -\frac{(n-2)^2}{4}; \\
                c_j^- \cdot r^{\gamma_j^-}\log r, & \ \text{ if }\mu_j = -\frac{(n-2)^2}{4}.
            \end{cases}
        \end{split} \label{Equ_Pre_Homogen Jac field}
    \end{align}
    for some $c_j^\pm\in \RR$ and
    \begin{align}
        \gamma_j^\pm = \gamma_j^\pm(\mbfC) := -\frac{n-2}{2}\pm \sqrt{\mu_j+\frac{(n-2)^2}{4}}.  \label{Equ_Pre_Asymp Spectrum gamma_j^pm}
    \end{align}
    Note that $\mu_j = -(n-2)^2/4$ if and only if $j=1$ and $\gamma_1^+ = \gamma_1^- = - (n-2)/2$. Clearly, if $\{\mbfC_k\}_{1\leq k\leq \infty}$ is a family of regular cone such that 
    \[
        \mbfF(|\mbfC_k|, |\mbfC_\infty|) \to 0
    \] then the asymptotic spectrum also converges: For every $j\geq 1$, as $k\to \infty$,
    \[
        \mu_j(\mbfC_k) \to \mu_j(\mbfC_\infty)\,.
    \]
    
    For later reference, we define the \textbf{asymptotic spectrum} of $\mbfC$ by
    \[
        \Gamma(\mbfC):= \{\gamma_j^\pm: j\geq 1\}\,.
    \]
    
    For every $\Lambda>0$, we denote $\scC_\Lambda$,
    \[
        \scC_\Lambda:= \{\text{stable minimal hypercone }\mbfC\subset \RR^{n+1 = 8}: \|\mbfC\|(\BB_1)\leq \Lambda \}\,.
    \]
    By \cite{schoen_regularity_1981} and \cite[Lemma 3.2]{WangZH22_Smoothing}, for each $\Lambda>0$, $\scC_\Lambda$ is compact under $\mbfF$-metric.
    \begin{Lem} \label{Lem_gamma_gap >0}
        For every $\Lambda>0$,
        \[ 
            \gamma_{\gap, \Lambda}:= \inf\{\gamma_2^+(\mbfC) - \gamma_1^+(\mbfC): \mbfC\in \scC_\Lambda \} >0.
        \]
        Here we use the convention that $\inf \emptyset = +\infty$.
    \end{Lem}
    \begin{proof}
        Since $\scC_\Lambda$ is $\mbfF$-compact and $\gamma_2^+(\mbfC) - \gamma_1^+(\mbfC)$ is continuous w.r.t. the $\mbfF$-metric, $\gamma_{\gap, \Lambda}$ (if not $+\infty$) is realized by some $\mbfC \in \scC_\Lambda$, and hence is positive.
    \end{proof}
    
    \subsection{Singularities of LSMH} \label{Subsec_Sing of LSMH}
        It is known from \cite{schoen_regularity_1981,ilmanen_strong_1996} that near each singular point $p$, an LSMH $\Sigma\subset M^{n+1}$ is modeled on stable minimal hypercones. More precisely, for every $r_j\searrow 0$, the blow up sequence $\eta_{p, r_j}(\Sigma)$ $\mbfF$-subconverges to some stable minimal hyercone $\mbfC\subset \RR^{n+1}$ with multiplicity $1$, which is called a tangent cone of $\Sigma$ at $p$. When $n+1 = 8$, since every stable minimal hypercone is regular, by \cite{simonAsymptoticsClassNonLinear1983}, the tangent cone of $\Sigma$ at $p$ is independent of the choice of blow-up radii $\{r_j\}$, and can be denoted by $\mbfC_p\Sigma$ in this paper.  Moreover, by adapting the proof in \cite{simonAsymptoticsClassNonLinear1983}, \cite[Theorem 6.3]{edelen_degeneration_2021} proved a quantitative version of uniqueness of tangent cone.  This will be discussed in Section \ref{Sec_Analysis on SMC} in detail.
        
        Let $\Sigma$ be an LSMH in a closed manifold $(M^8, g)$ with a singularity $p$.  For a $L_{loc}^2$-function $v$ defined near $p$, \cite[Section 3.3]{WangZH20_Deform} introduced the notion of \textbf{asymptotic rate} of $v$ at $p$, defined to be
        \begin{align*}
            \cA\cR_p(v):= \sup\{\gamma: \lim_{s\searrow 0} \int_{A^g(p; s, 2s)} v^2\cdot \rho^{-n-2\gamma}\ d\|\Sigma\| = 0 \},
        \end{align*}
        where $\rho(x):= \dist_g(x, \Sing(\Sigma))$ is the distance function to the singular set, and $A^g(p; s, t)$ is the geodesic annulus in $(M, g)$ centered at $p$ with inner radius $s$ and outer radius $t$.  Heuristically, if $v$ grows like $\rho^\gamma$ near $p$, then $\cA\cR_p(v) = \gamma$.  
        
        The advantage of this asymptotic rate is that it relates the behavior of Jacobi fields on $\Sigma$ near $p$ with the asymptotic spectrum of $\mbfC_p\Sigma$.  For example, by \cite[Section 3.3]{WangZH20_Deform} we have,
        \begin{Lem} \label{Lem_Pre_Asymp Rate takes value in Gamma(C) and int by part}
            Suppose that a nonzero function $u\in W^{2,2}_{loc}(\Sigma)$ satisfies $L_{\Sigma, g} u\in L^\infty(\Sigma)$. Then $\cA\cR_p(u) \in \{-\infty\}\cup \Gamma(\mbfC_p\Sigma)\cup [1, +\infty]$.
            
            In addition, 
            \begin{enumerate} [(i)]
                \item If $u\in W^{1,2}(\Sigma)$, then $\cA\cR_p(u)\geq \gamma_1^+(\mbfC_p\Sigma)$;
                \item If $u>0$ and $L_{\Sigma, g}u$ vanishes near $p$, then $\cA\cR_p(u)\in \{\gamma_1^\pm(\mbfC_p\Sigma)\}$;
                \item If $\cA\cR_q(u) > -(n-2)/2$ for every $q\in \Sing(\Sigma)$, then
                      \[
                          \int_\Sigma |\nabla_{\Sigma, g} u|^2 + \rho^{-2}u^2 \ d\|\Sigma\| <+\infty.
                      \]
                      In particular, if $u, v\in W^{2,2}_{loc}(\Sigma)$ such that $L_{\Sigma, g}u, L_{\Sigma, g}v \in L^\infty(\Sigma)$ and that $\cA\cR_q(u) > -(n-2)/2$, $\cA\cR_q(v) > -(n-2)/2$ for every $q\in \Sing(\Sigma)$, then we can do the following integration by parts, 
                      \[
                          \int_\Sigma u\cdot L_{\Sigma, g}v \ d\|\Sigma\| = \int_\Sigma v\cdot L_{\Sigma, g}u \ d\|\Sigma\|.
                      \]
            \end{enumerate}
        \end{Lem}
        \begin{proof}
            It follows from \cite[Lemma 3.14]{WangZH20_Deform} that if $-\infty < \cA\cR_p(u) < 1$, then
            \[
                \cA\cR_p(u)\in \Gamma(\mbfC_p\Sigma)\,.
            \]
            Furthermore, (i) follows from \cite[Corollary 3.17 (1)]{WangZH20_Deform}, and (ii) follows from \cite[Corollary 3.15]{WangZH20_Deform}.  
            
            To prove (iii), first note that $\cA\cR_q(u)>-(n-2)/2$ implies the existence of $\gamma>-(n-2)/2$ and $s_0>0$ satisfying that for every $s\in (0, s_0)$, 
            \begin{align}
                \int_{A^g(q, s, 2s)} u^2\ d\|\Sigma\| \leq s^{n+2\gamma}.  \label{Equ_Pre_L^2 growth est in annuli}
            \end{align}
            By taking sum of (\ref{Equ_Pre_L^2 growth est in annuli}) over $s = 2^{-l}s_0$, $l\geq 1$ and all singularities, we get 
            \begin{align}
                \int_{B^g(\Sing(\Sigma), s_0)} u^2\rho^{-2}\ d\|\Sigma\| < +\infty.  \label{Equ_Pre_L^2 weighted est in ball}
            \end{align}
            Now to estimate the energy, let $\varphi\in C_c^\infty((1/2,+\infty), [0, 1])$ be a non-negative cut-off function which equals to $1$ on $[1, +\infty)$. Then for every $0<r<<1$, let $\varphi_r:= \varphi(\rho/r)$, and we have an energy estimate,
            \begin{align}
                \int_\Sigma |\nabla_{\Sigma, g} u|^2 \varphi_r^2 \leq \int_\Sigma C(n, g)(|\nabla \varphi_r|^2 + 1 + |A_{\Sigma, g}|^2)u^2 + |L_{\Sigma, g}u|\cdot |u|\cdot \varphi_r^2,  \label{Equ_Pre_Energy est alm Jac equ}
            \end{align}
            Since $|\nabla \varphi_r|\leq C(\varphi)r^{-1}\cdot\chi_{\{r/2\leq\rho\leq r\}}\leq C'(\varphi)\rho^{-1}$, $|A_{\Sigma, g}|\leq C(g, \Sigma)\rho^{-1}$ and $L_{\Sigma, g}u\in L^\infty(\Sigma)$, (\ref{Equ_Pre_L^2 weighted est in ball}) implies that the right hand side is bounded as $r \to 0$. In particular, we have \[
                \int_\Sigma |\nabla_{\Sigma, g} u|^2 + \rho^{-2}u^2 \ d\|\Sigma\|_g <+\infty.   \]
            
            Therefore, we can the integration by parts simply following from a standard cut-off argument with the estimates above.
        \end{proof}
        
        \begin{Def}[Slower growth]
            A function $u\in L^2_{loc}(\Sigma)$ is a \textbf{function of slower growth}, if $\forall p\in \Sing(\Sigma)$, 
            \[
                \cA\cR_p(u) \geq \gamma_2^+(\mbfC_p\Sigma)\,.
            \]
            In particular, if $\Sing(\Sigma) = \emptyset$, the condition is always satisfied for any $L^2$ function.
            
            Similarly, when $\Sigma$ is two-sided, a Jacobi field on $\Sigma$ is a \textbf{Jacobi field of slower growth} if it is a function of slower growth. We denote by $\Ker^+ L_{\Sigma,g}$ the space of all the Jacobi fields of slower growth on $\Sigma \subset (M, g)$.
            
            A two-sided LSMH $\Sigma$ is called \textbf{semi-nondegenerate}, if 
            \[
                \Ker^+ L_{\Sigma, g} = \set{0}\,.
            \]
            Following the notation in (\ref{Equ_Pre_Z2 inv functions}), a one-sided connected LSMH $\Sigma$ is called \textbf{semi-nondegenerate}, if its connected two-sided double cover $\hat{\Sigma}$ has
            \[
                (\Ker^+ L_{\hat{\Sigma}, \hat{g}})^{\ZZ_2} = \set{0}\,.
            \]
        \end{Def}
        
        \begin{Rem}\label{Rem:slow_growth_1}
            By definition, if $\Sing(\Sigma) = \emptyset$, then $\Sigma$ is semi-nondegenerate if and only if $\Sigma$ is nondegenerate.  
            
            By Lemma \ref{Lem_gamma_gap >0}, for every stable minimal hypercone in $\RR^8$, we have $\gamma_2^+ > \gamma_1^+ \geq -(n-2)/2$. Thus by Lemma \ref{Lem_Pre_Asymp Rate takes value in Gamma(C) and int by part}, a Jacobi field of slower growth has finite energy. 
        \end{Rem}
        
        The following Theorem generalizes \cite[Lemma 2.24]{liwang2020generic} to semi-nondegenerate or one-sided LSMH, and plays an important role in the proof of Theorem \ref{Thm_Main}.
        \begin{Thm} \label{Thm_Reg Deform Thm_OneTwo sided}
            Let $k\geq 4$ be an integer, $\alpha\in (0, 1)$, and $\Sigma \subset (M^8, g)$ be a semi-nondegenerate LSMH with $\Sing(\Sing)\neq \emptyset$. Then there exists an open dense subset $\scF^{k,\alpha}\subset C^{k, \alpha}(M)$ depending on $\Sigma, M, g$ with the following property.
            For any $f\in \scF^{k,\alpha}$, and $j\geq 1$, Suppose that
            \begin{enumerate} [(a)]
                \item $g_j = (1+t_jf_j)g$ is a family of metrics, where either $\RR\ni t_j\searrow 0$ and $f_j\to f$ in $C^4(M)$ as $j\to \infty$, or $t_j \equiv 0$ for all $j$.
                \item $\Sigma_j \subset (M, g_j)$ is an LSMH with
                      \[
                          \ind(\Sigma_j) = \ind(\Sigma)\,,
                      \]
                      and 
                      \[
                          \mbfF(|\Sigma_j|_{g_j}, |\Sigma|_{g}) \to 0\,,
                      \] as $j\to \infty$. 
            \end{enumerate}
            Then there exists $p\in \Sing(\Sigma)$ and some neighborhood $U_p\subset M$ both depending on the sequence $\{(g_j, \Sigma_j)\}$ such that for infinitely many $j>>1$, \[
                \Sing(\Sigma_j)\cap U_p  = \emptyset.   \]
        \end{Thm}
        For the sake of completeness, we include a proof of this Theorem in Appendix \ref{Sec_App_Reg Deform Thm}, which is a modification of \cite{liwang2020generic} based on \cite{WangZH20_Deform}.
    
    \subsection{Cone Decompositions \`a la Edelen}  \label{Subsec_Cone Decomp Edelen}
        In the recent work of N. Edelen \cite{edelen_degeneration_2021}, he proved a finite parametrization result of all the minimal hypersurfaces near a given minimal hypersurface in a closed 8-dimensional Riemannian manifold. The main difficulty lies in the parametrization near a singular point. Therein, Edelen introduces a notion of cone decomposition to overcome this issue.
        
        \begin{Def}[Strong Cone Region, {\cite[Definition~6.0.3]{edelen_degeneration_2021}}]\label{Def:strong_cone_region}
            Let $g$ be a $C^2$ metric on $\mathbb{B}(a, R) \subset \mathbb{R}^8$, and $V$ an integral varifold on $(\mathbb{B}(a, R), g)$. Given $\mathbf{C} \in \mathcal{C}, m \in \mathbb{N}, \beta, \tau, \sigma \in [0, 1/4], \rho \in [0, R]$, we say that $V\llcorner(\mathbb{A}(a, \rho, R), g)$ is a \textbf{$(\mathbf{C},m,\beta)$-strong-cone region} if there are $C^2$ functions 
            \[
                \set{u_j: (a + \mathbf{C}) \cap \mathbb{A}(a, \rho/8, R) \rightarrow \mathbf{C}^\perp}^m_{j=1}
            \] so that for any $r \in [\rho, R] \cap (0, \infty)$,
            \begin{enumerate}
                \item Small $C^2$ norms: $r^{-1}|u_j| + |\nabla u_j| + r|\nabla^2 u_j| \leq \beta$, for any $j = 1, \cdots, m$;
                \item Almost constant density ratios: $m \theta_\mathbf{C}(0) - \beta \leq \theta_V(a, r) \leq m\theta_\mathbf{C}(0) + \beta$;
                \item Multigraph: $V \llcorner  \mathbb{A}(a, \rho/8, R) = \sum^m_{j = 1} |\graph_{a + \mathbf{C}}(u_j) \cap  \mathbb{A}(a, \rho/8, R)|$.
            \end{enumerate}
            In this case, for simplicity, we will also call $\AAa(a, \rho, R)$ a strong-cone region for $V$.
        \end{Def}
        
        \begin{Def}[Smooth model, {\cite[Definition~7.0.1]{edelen_degeneration_2021}}]\label{Def:smooth_model}
            Given $\Lambda, \gamma \geq 0, \sigma \in (0, 1/3)$, a tuple $(S, \mbfC, \set{(\mbfC_\alpha,  \mathbb{B}(y_\alpha, r_\alpha))}_\alpha)$ is called a \textbf{$(\Lambda, \sigma, \gamma)$-smooth model} if
            \begin{itemize}
                \item $S$ is a stationary integral $7$-varifold in $(\mathbb{R}^8, g_{\mathrm{eucl}})$ with $$\theta_S(\mathbf{0}, \infty) \leq \Lambda\,;$$
                \item $\mbfC, \set{\mbfC_\alpha}_\alpha \subset \mathcal{C}_\Lambda$;
                \item $\set{ \mathbb{B}(y_\alpha, 2r_\alpha)}$ is a finite collection of disjoint balls in $\mathbb{B}_{1 - 3\sigma}$;
            \end{itemize}
            such that the following are satisfied:
            \begin{enumerate}
                \item $S$ can be represented by a union of disjoint closed smooth embedded minimal hypersurfaces in $\mathbb{R}^8\setminus \set{y_\alpha}_\alpha$ with multiplicities, i.e.,
                      \[
                          S = m_1 |S_1| + m_2 |S_2| + \cdots + m_k |S_k|\,;
                      \]
                \item $S\llcorner( \mathbb{A}(\mathbf{0}, 1, \infty), g_{\eucl})$ is a $(\mbfC, m, \gamma)$-strong-cone region;
                \item For each $\alpha$, there is a $j$ so that $\spt S \cap  \mathbb{A}(y_\alpha, 0, 2r_\alpha) = S_j \cap  \mathbb{A}(y_\alpha, 0, 2r_\alpha)$, and it is a $(\mbfC_\alpha, 1, \gamma)$-strong-cone region.
            \end{enumerate}
            If there is no ambiguity, we will refer to the smooth model as $S$ for simplicity.
        \end{Def}
        
        \begin{Def}[Smooth model scale constant, {\cite[Definition~7.0.2]{edelen_degeneration_2021}}]\label{Def:smooth_model_const}
            Given a $(\Lambda, \sigma, \gamma)$-smooth model $S$, we let $\epsilon_S$ be the largest number $\leq \min(1, \min_\alpha\set{r_\alpha})$ for which the graph map 
            \[
                \graph_S: T^\perp(\bigcup_j S_j) \rightarrow \mathbb{R}^8, \quad \graph_S(x, v):= x + v\,,
            \]
            is diffeomorphism from $\set{(x,v) \in T^\perp(\bigcup_j S_j): x \in \mathbb{B}_2\setminus \bigcup_\alpha \BB(y_\alpha, r_\alpha/8), |v| < 2\epsilon_S}$ onto its image, and satisfies
            \[
                \left|D\graph_S|_{(x,v)} - \mathrm{Id}\right| \leq |\epsilon_S|^{-1}|v|\,.
            \]
        \end{Def}
        
        \begin{Def}[Smooth region, {\cite[Definition~7.0.3]{edelen_degeneration_2021}}]
            Given a smooth model $S$, a $C^2$ metric $g$  on $ \mathbb{B}(a, R) \subset \mathbb{R}^8$, and $\beta \in (0, 1)$, we say that an integral varifold $V\llcorner ( \mathbb{B}(a, R), g)$ is an \textbf{$(S, \beta)$-smooth region} if for each $i = 1, \cdots, k$, there are $C^2$ functions $\set{u_{ij}: S_i \rightarrow S^\perp_i}^{m_i}_{j = 1}$ so that 
            \[
                \left((\eta_{a, R})_\sharp V\right)\llcorner({ \mathbb{B}_1} \setminus \bigcup_\alpha  \mathbb{B}(y_\alpha, r_\alpha /4)) = \sum^k_{i = 1} \sum^{m_i}_{j = 1}\left[\graph_{S_i}(u_{ij})\cap { \mathbb{B}_1} \setminus \bigcup_\alpha  \mathbb{B}(y_\alpha, {r_\alpha /4} )\right]_{R^{-2}(\eta^{-1}_{a,R})^*g}\,,
            \]
            and
            \[
                |u_{ij}|_{C^2(S_i)} \leq \beta \epsilon_{S}, \quad \forall i, j\,,
            \]
            where $\epsilon_S$ is a scale constant in the previous definition.
            
            In this case, for simplicity, we will also call $\BB(a, R)$ a smooth region for $V$.
        \end{Def}
        
        \begin{Def}[Cone decomposition,  {\cite[Definition~7.0.4]{edelen_degeneration_2021}}]\label{Def:cone_decomposition}
            Given $\theta, \gamma, \beta \in \mathbb{R}$, $\sigma \in (0, 1/3)$, and $N \in \mathbb{N}$, we let 
            \begin{itemize}
                \item $g$ be a $C^{k, \alpha}$ metric on $\mathbb{B}(x, R) \subset \mathbb{R}^8$;
                \item $V$ be an integral varifold in $(\BB(x, R), g)$;
                \item $\cS = \set{S_s}_{s}$ be a finite collection of $(\theta, \sigma, \gamma)$-smooth models.
            \end{itemize}
            A \textbf{$(\theta, \beta, \cS, N)$-cone decomposition} of $V \llcorner(\BB(x, R), g)$ consists of the following parameters:
            \begin{itemize}
                \item Integers $N_C$, $N_S$ satisfying $N_C + N_S \leq N$, where $N_C$ is the number of strong-cone regions and $N_S$ the number of smooth regions and ;
                \item Points $\set{x_a}_a, \set{x_b}_b \subset \BB(x, R)$, where $\set{x_a}$ are centers of strong-cone regions and $\set{x_b}$ centers of smooth regions;
                \item Radii $\set{R_a, \rho_a| R_a \geq 2 \rho_a}_a, \set{R_b}_b$, corresponding to radii of annuli in the definition of strong-cone regions and of balls in the definition of smooth regions, respectively;
                \item Cones $\set{\mbfC_a}_a \subset \cC$;
                \item Integers $\set{1 \leq m_a \leq [\theta]}_a$;
                \item Indices $\set{s_b}_b$, corresponding to the smooth model $S_{s_b}$;
            \end{itemize}
            where $a = 1, \cdots, N_C$ and $b = 1, \cdots, N_S$. Moreover, these parameters determine a covering of balls and annuli satisfying:
            \begin{enumerate}
                \item Every $V\llcorner(\AAa(x_a, \rho_a, R_a), g)$ is a $(\mbfC_a, m_a, \beta)$-strong cone region and every $V \llcorner (\BB(x_b, R_b), g)$ is a $(S_{s_b}, \beta)$-smooth region;
                \item In (1), there is either a strong-cone region $\AAa(x_a, \rho_a, R_a)$ for $V$ with $R_a = R$ and $x_a = x$, or a smooth region $B_{R_b}(x_b)$ for $V$ with $R_b = R$ and $x_b = x$;
                \item If $V \llcorner{\AAa(x_a, \rho_a, R_a)}$ is a $(\mbfC_a, m_a, \beta)$-strong-cone region and $\rho_a > 0$, then there exists either a smooth region $\BB(x_b, R_b)$ for $V$ with $R_b = \rho_a$, or another cone region $\AAa(x_{a'}, \rho_{a'}, R_{a'})$ for $V$ with $R_{a'} = \rho_a, x_{a'} = x_a$. If $\rho_a = 0$, then $\theta_{\mbfC_a}(\mathbf{0}) > 1$;
                \item If $V\llcorner (\BB(x_b, R_b), g)$ is a smooth region with $(S, \mbfC, \set{\mbfC_{\hat\alpha}, B(y_{\hat\alpha}, r_{\hat\alpha})}_{\hat\alpha})\in \cS$, then for any ${\hat\alpha}$, there exists a point $x_{b,{\hat\alpha}}$ and a radius $R_{b,{\hat\alpha}}$ satisfying
                      \[
                          |x_{b,{\hat\alpha}} - (x_b + R_b \cdot y_{\hat\alpha})| \leq \beta R_b r_{\hat\alpha}, \quad \frac{1}{2}\leq \frac{R_{b,{\hat\alpha}}}{R_b r_{\hat\alpha}} \leq 1 + \beta\,,
                      \]
                      and either a strong-cone region $\AAa(x_{a'}, \rho_{a'}, R_{a'})$ for $V$ with $R_a = R_{b, {\hat\alpha}}, x_a = x_{b, {\hat\alpha}}$, or another smooth region $\BB(x_{b'}, R_{b'})$ with $R_{b'} = R_{b, {\hat\alpha}}$ and $x_{b'} = x_{b, {\hat\alpha}}$.
            \end{enumerate}
        \end{Def}
        
        By (\ref{Equ_Pre_Density dicrete for SMC}), we can enumerate the set of densities with multiplicities, i.e.,
        \begin{equation}
            \set{\tilde \theta_i}_i := \set{m\theta_i: m, i\in \mathbb{N}, m > 0}\,,
        \end{equation}
        such that $1 = \tilde \theta_0 < \tilde \theta_1 < \cdots$.

        \begin{Thm}[Existence of Cone Decomposition, {\cite[Theorem~7.1]{edelen_degeneration_2021}}]\label{Thm:cone_decomposition}
            Given $l, I \in \mathbb{N}$, $0 < \gamma \leq 1$, $\sigma \in (0, \frac{1}{100(I+1)}]$, there exist constants $\delta_{l, I}$, $N$ and a finite collection of $(\tilde \theta_l, \sigma, \gamma)$-smooth models $\set{S_s}_s = \mathcal{S}$, all depending only on $(l, I, \gamma, \sigma)$ with the following property.
            
            For any $C^3$ metric $g$ on $\mathbb{B}_1$ satisfying $|g - g_{\eucl}|_{C^3(\mathbb{B}_1)} \leq \delta_{l, I}$, any stationary integral varifold $V$ in $(\mathbb{B}_1, g)$ with $\ind(\spt V) \leq I$, $\mathbf{C} \in \mathcal{C}$, and $m \in \mathbb{N}$, if
            \begin{itemize}
                \item $m \theta_{\mathbf{C}}(0) \leq \tilde \theta_l$;
                \item $\dist_H(\spt V \cap B_1, \mathbf{C} \cap B_1) \leq \delta_{l, I}$;
                \item $(m - 1/2)\theta_{\mathbf{C}}(0) \leq \theta_V(0, 1/2)$ and $\theta_V(0, 1) \leq (m + 1/2) \theta_{\mathbf{C}}(0)$;
            \end{itemize}
            Then there exists a radius $r \in (1 - 20(I + 1)\sigma, 1)$ so that $V \llcorner (B_r, g)$ admits a $(\tilde \theta_l, \gamma, \mathcal{S}, N)$-cone decomposition.
        \end{Thm}

    \subsection{Singular Capacity} \label{Subsec_SCAP}
        Let $\scM$ be the space of triples $(\Sigma, N, g)$, where $(N, g)$ is a Riemannian manifold of dimension $8$, not necessarily complete; $\Sigma\subset (N, g)$ be a connected LSMH with finitely many singularities.  Let $\scM^{2\text{-sided}}$ be the space of triples $(\Sigma, N, g)\in \scM$ such that $\Sigma$ is two-sided. As before, we endow topology on $\scM$ and $\scM^{2\text{-sided}}$ by $C^4_{loc}$ convergence in $(N, g)$ and $\mbfF$ convergence in $\Sigma$ with multiplicity $1$.
        
    \cite[Section 4]{liwang2020generic} introduced the following functional $\SCap$ on $\scM^{2\text{-sided}}$, which serves as a way to count the number of singularities with weight such that it is upper semi-continuous under convergence.
        \begin{Thm} [Singular Capacity {\cite[Theorem 4.2]{liwang2020generic}}]  \label{Thm_Pre_SCap exist}
            There exists a functional $\SCap: \scM^{2\text{-sided}}\to \NN$, called \textbf{singular capacity}, such that 
            \begin{enumerate} [(i)]
                \item For every nontrivial stable minimal hypercone $\mbfC\subset \RR^8$, we have \[
                          1\leq \SCap(\mbfC, \RR^8, g_{\eucl}) < +\infty.  \]
                      Denote for simplicity that $\SCap(\mbfC):= \SCap(\mbfC, \RR^8, g_{\eucl})$;
                \item For every triple $(\Sigma, N, g)\in \scM^{2\text{-sided}}$, \[
                          \SCap(\Sigma, N, g) = \sum_{p\in \Sing(\Sigma)} \SCap(\mbfC_p\Sigma),   \]
                      where recall $\mbfC_p\Sigma$ denotes the tangent cone of $\Sigma$ at $p$.  Conventionally, $\SCap(\Sigma, N, g) = 0$ if $\Sing(\Sigma) = \emptyset$;
                \item If $(\Sigma_j, N, g_j)\to (\Sigma_\infty, N, g_\infty)$ in $\scM^{2\text{-sided}}$, and $\forall 1\leq j\leq +\infty$, $\Sigma_j$'s are stable, then for every open subset $U\subset N$ with $\partial U\cap \Sing(\Sigma) = \emptyset$, we have \[
                              \limsup_{j\to \infty} \SCap(\Sigma_j, U, g_j) \leq \SCap(\Sigma_\infty, U, g_\infty).  \]
            \end{enumerate}
        \end{Thm}
        
        In this paper, we shall need a modified singular capacity, which is still upper semi-continuous under convergence with multiplicities. 
        \begin{Thm} \label{Thm_Pre_SCAP under converg w multip}
            There exists a functional $\SCAP: \scM \to \NN$ with the following properties.
            \begin{enumerate} [(i)]
                \item For every nontrivial stable minimal hypercone $\mbfC\subset \RR^8$, we have \[
                          1\leq \SCAP(\mbfC, \RR^8, g_{\eucl}) < +\infty.  \]
                      We denote for simplicity $\SCAP(\mbfC):= \SCAP(\mbfC, \RR^8, g_{\eucl})$;
                \item For every triple $(\Sigma, N, g)\in \scM$,
                      \begin{align}
                          \SCAP(\Sigma, N, g) := \begin{cases}
                              \displaystyle\sum_{p\in \Sing(\Sigma)} \SCAP(\mbfC_p\Sigma)\    & \text{ if }\Sigma \text{ is two-sided}; \\
                              2\displaystyle \sum_{p\in \Sing(\Sigma)} \SCAP(\mbfC_p\Sigma)\  & \text{ if }\Sigma \text{ is one-sided}.
                          \end{cases} \label{Equ_App_SCap for 1 and 2-sided}
                      \end{align}
                      Conventionally, $\SCAP(\Sigma, N, g) = 0$ if $\Sing(\Sigma) = \emptyset$;
                \item Given an integer $I\geq 0$, Suppose that $(\Sigma_j, N, g_j)\to (\Sigma_\infty, N, g_\infty)$ in $\scM$, where $\forall 1\leq j\leq +\infty$, $\ind(\Sigma_j) \equiv I$, and either all two-sided or all one-sided.
                      
                      Then for every open subset $U\subset N$ with $\partial U\cap \Sing(\Sigma) = \emptyset$, we have
                          \begin{align}
                              \limsup_{j\to \infty} \SCAP(\Sigma_j, U, g_j) \leq \SCAP(\Sigma_\infty, U, g_\infty).  \label{Equ_Pre_SCap w ind multip, usc}
                          \end{align}
                \item Suppose $N$ is a closed eight-manifold. Let $\set{g_j}$ be a family of $C^4$ metrics on $N$, and $\Sigma_j \subset (N, g_j)$ be connected stable minimal hypersurfaces, $1\leq j\leq \infty$, such that when $j\to \infty$, $g_j\to g_\infty$ in $C^4$, $|\Sigma_j|$ $\mbfF$-converges to $m|\Sigma_\infty|$ for some integer $m\geq 2$.
                      
                      If $\Sing(\Sigma_\infty) \neq \emptyset$, then we have,
                      \begin{align}
                          \limsup_{j\to \infty} \SCAP(\Sigma_j, N, g_j) \leq \SCAP(\Sigma_\infty, N, g_\infty) - 1.  \label{Equ_Pre_SCap w ind multip, improved usc}
                      \end{align}
            \end{enumerate}
        \end{Thm}
        The proof of Theorem \ref{Thm_Pre_SCAP under converg w multip} uses a fact about the multiplicity of limiting connected stable minimal hypersurface. See Appendix \ref{Sec_App_Sing Cap} for details.

\part{Generic Regularity} \label{Part_Generic Reg}
    
    \section{Canonical Neighborhoods} \label{Sec_Canonical Neighb}
        Throughout this section, let $M$ be a closed smooth manifold of dimension $8$, $k\geq 4$ be an integer, and $\alpha\in (0, 1)$.  
        \begin{Def} \label{Def_injrad, canonical neighb}
            Let $g$ be a $C^{k, \alpha}$-metric on $M$, $\Sigma\subset (M^8, g)$ be a LSMH, $\Lambda \geq 1$, and $\delta>0$. We denote 
            \[
                \injrad(g, \Sigma):= \min\{ \injrad(M, g), \min\{\dist_g(p, p')/2: p\neq p'\in \Sing(\Sigma)\}\}.
            \]
            We define $\cL^{k, \alpha}(\Sigma, g; \Lambda, \delta)$ to be the space of all pairs $(g', \Sigma')\in \cM^{k, \alpha}(M)$ such that,
            \begin{itemize}
                \item $g'$ is a $C^{k, \alpha}$ metric on $M$ with $\|g'\|_{C^{k, \alpha}}\leq \Lambda$ and $\|g'-g\|_{C^{k-1, \alpha}}\leq \delta$;
                \item $\Sigma'$ is a LSMH in $(M, g')$ satisfying
                      \[ \mbfF(|\Sigma|_g, |\Sigma'|_{g'})\leq \delta, \quad \ind(\Sigma'; g')\leq \Lambda\,; \]
                \item For every $p\in \Sing(\Sigma)$, there exists $p'\in \Sing(\Sigma')\cap B^g(p, \injrad(g, \Sigma))$ such that
                      \[ \theta_{|\Sigma'|}(p') = \theta_{|\Sigma|}(p) \]
            \end{itemize}
            Here, the norms of tensors are all measured under metric $g$.

            $\cL^{k, \alpha}(\Sigma, g;\Lambda, \delta)$ is called a \textbf{canonical (pseudo-)neighborhood} of $(g, \Sigma)$ in $\cM^{k, \alpha}(M)$, and its topology is induced from $\cM^{k, \alpha}(M)$.
        \end{Def}
        Clearly, when $g, \Sigma$ is fixed, $\cL^{k, \alpha}(\Sigma, g;\Lambda, \delta)$ is monotonely increasing in $\Lambda, \delta$. Also note that in general, $\cL^{k, \alpha}(g, \Sigma; \Lambda, \delta)$ is NOT a genuine \textit{neighborhood} of $(g, \Sigma)$ in $\cM^{k,\alpha}(M)$. In other words, there might be $(g', \Sigma')\in \cM^{k,\alpha}(M)$ arbitrarily close to $(g, \Sigma)$ but not contained in $\cL^{k,\alpha}(g, \Sigma; \Lambda, \delta)$. 
        
        \begin{Lem}[Compactness of $\cL^{k,\alpha}$] \label{Lem_Cptness for scL^k}
            For every $(g, \Sigma)\in \cM^{k, \alpha}(M)$ and every $\Lambda>1$, there exists $\delta_0(g, \Sigma, \Lambda, k, \alpha)\in (0, 1)$ with the following property. For every $\delta \in (0, \delta_0)$, $\cL^{k,\alpha}(g, \Sigma; \Lambda, \delta)$ is compact under $C^{k-1, \alpha}$-norm in metric and $\mbfF$-metric in minimal hypersurfaces. 
        \end{Lem}
        \begin{proof}
            Suppose otherwise, there exists $\delta_j\searrow 0$ and for each $j\geq 1$, a sequence of pairs $\{(g_j^q, \Sigma_j^q)\}_{q\geq 1}$ in $\cL^{k,\alpha}(g, \Sigma; \Lambda, \delta_j)$ with no limit point in $\cL^{k,\alpha}(g, \Sigma; \Lambda, \delta_j)$. By compact embedding of $C^{k,\alpha}$ into $C^{k-1, \alpha}$ and \cite{schoen_regularity_1981}, after passing to a subsequence, when $q\to \infty$, 
            \[
                g_j^q \to \bar{g}_j \ \text{ in }C^{k-1, \alpha}, \ \ \  \Sigma_j^q \to \bar{V}_j \ \text{ in }\mbfF\text{-metric of varifolds,}    
            \]
            where $\bar{V}_j$ is a LSMH with multiplicity.  Moreover,  $\mbfF(|\Sigma|_g, \bar{V}_j) \leq \delta_j \to 0$ as $j\to \infty$. Hence by connectivity of $\Sigma$ and upper semi-continuity of density, for $j>>1$, $\bar{V}_j = |\Gamma_j|_{\bar{g}_j}$ for some LSMH $\Gamma_j \subset (M, \bar{g}_j)$ with index $\leq \Lambda$.

            By the contradiction assumption, $(\bar{g}_j, \Gamma_j)\notin \cL^k(g, \Sigma; \Lambda, \delta_j)$.
            Therefore, by Allard's regularity theorem \cite{allard_first_1972} and the upper semi-continuity property of density under the convergence $|\Sigma_j^q| \to |\Gamma_j|$, after passing to a subsequence, there exists $p\in \Sing(\Sigma)$, $p_j \in \Sing(\Gamma_j)$ such that 
            \[
                \theta(\mbfC_{p_j}\Gamma_j)>\theta(\mbfC_p\Sigma^q_j) = \theta(\mbfC_p\Sigma)\,, \quad p_j \to p\,,
            \]
            as $j\to \infty$.  Since by (\ref{Equ_Pre_Density dicrete for SMC}), the density of stable minimal hypercones form a discrete subset of $\RR$, we then have \[
                \limsup_{j\to \infty} \theta(\mbfC_{p_j}\Gamma_j) > \theta(\mbfC_p\Sigma),   \]
            which violates the upper semi-continuity of density under convergence $|\Gamma_j| \to |\Sigma|$.   
        \end{proof}
        
        The main reason to consider canonical neighborhoods is that the following Sard-Smale-type Theorem holds.
        \begin{Thm} \label{Thm_Generic Semi-nondeg_Loc}
            For every $(g, \Sigma)\in \cM^{k,\alpha}(M)$ and every $\Lambda>0$, there exists $\kappa_0=\kappa_0(g, \Sigma, \Lambda)>0$ such that 
            \begin{align*}
                \cG^{k, \alpha}(g, \Sigma; \Lambda, \kappa_0) := \{g'\in\  & \cG^{k,\alpha}(M): \text{ either }g'\notin \Pi(\cL^{k,\alpha}(g, \Sigma; \Lambda, \kappa_0)), \text{ or every }         \\
                                                                           & \Sigma' \text{ with } (g', \Sigma')\in \cL^{k, \alpha}(g, \Sigma; \Lambda, \kappa_0) \text{ is semi-nondegenerate} \}. 
            \end{align*}
            is an open dense subset of $\cG^{k,\alpha}(M)$. 
        \end{Thm}
        Heuristically, one can view $\cG^{k,\alpha}(g, \Sigma; \Lambda, \kappa_0)$ as the set of regular values of $\Pi|_{\cL^{k, \alpha}(g, \Sigma; \Lambda, \kappa_0)}$. When we only focus on regular minimal hypersurfaces, such a Sard-Smale Theorem was proved in \cite{white_space_1991,white_bumpy_2017} by showing that $\cM^{k, \alpha}$ is a Banach manifold, and $\Pi$ is a Fredholm map with Fredholm index $0$.  Here, however, it's hard to expect any Banach manifold structure on $\cL^k(g, \Sigma; \Lambda, \delta)$. The way we prove Theorem \ref{Thm_Generic Semi-nondeg_Loc} is to unwrap the proof of infinite dimensional Sard-Smale Theorem by hand. More precisely, we need
        \begin{enumerate} [(i)]
            \item A controlled behavior of Jacobi fields of slower growth under convergence in $\cL^{k, \alpha}(g, \Sigma; \Lambda, \kappa_0)$;
            \item A good way to slice $\cG^{k,\alpha}(M)$ into union of finite dimensional subspaces $\set{\cF_{g', \Sigma'}}$ such that each $\Pi^{-1}(\cF_{g', \Sigma'})\cap \cL^{k,\alpha}(g, \Sigma; \Lambda, \kappa_0)$ is bi-Lipschitz to a compact subset in the finite dimensional vector space $\Ker^+L_{\Sigma', g'}$.
        \end{enumerate}
        Both of these rely on a good control on growth rate of certain functions (Jacobi fields or graphical functions of nearby LSMH) on $(g', \Sigma')\in \cL^{k,\alpha}(g, \Sigma; \Lambda, \kappa_0)$, where we make full use of the technical definition of $\cL^{k,\alpha}(g, \Sigma; \Lambda, \delta)$ and Jacobi fields of slower growth. The proof of Theorem \ref{Thm_Generic Semi-nondeg_Loc} is carried out in Section \ref{Sec:Gen_Semi-nondeg}.
        
        To derive a Sard-Smale type theorem on $\cM^{k,\alpha}(M)$, it suffices to show that it can be covered by countably many canonical neighborhoods.
        
        \begin{Thm} \label{Thm_Countable Decomp}
            Let $\kappa: \cM^{k, \alpha}(M)\times \RR_+\to (0, +\infty)$ be a positive function, not necessarily continuous. Then there exists a countable number of triples $\{(g_j, \Sigma_j; \Lambda_j)\}\in \cM^{k, \alpha}(M)\times \RR_+$ such that 
            \[
                \cM^{k,\alpha}(M) = \bigcup_{j\geq 1}\cL^{k,\alpha}(g_j, \Sigma_j; \Lambda_j, \kappa_j),   
            \]
            where $\kappa_j:= \kappa(g_j, \Sigma_j; \Lambda_j)$.
        \end{Thm}
        Theorem \ref{Thm_Countable Decomp} will be proved in Section \ref{Sec:Count_Decomp}, based on decomposition arguments inspired by the work \cite{edelen_degeneration_2021}.
        
        A direct corollary of Theorem \ref{Thm_Generic Semi-nondeg_Loc} and \ref{Thm_Countable Decomp} is the genericity of semi-nondegeneracy.
        
        \begin{Cor} \label{Cor_Canon Neighb_Global generic semi-nondeg}
            Given an integer $k\geq 4$, $\alpha\in (0, 1)$ and a eight-dimensional closed smooth manifold $M^8$, let $\cG_0^{k,\alpha}(M)$ be the space of metrics $g\in \cG^{k,\alpha}(M)$ such that \begin{itemize}
                \item every regular connected minimal hypersurface $\Sigma$ under $g$ is \textbf{strongly nondegenerate}, i.e. either $\Sigma$ is two-sided and nondegenerate, or $\Sigma$ is one-sided and has a non-degenerate double cover;
                \item every connected LSMH under $g$ is semi-nondegenerate.
            \end{itemize}
            Then $\cG_0^{k,\alpha}(M)$ is a generic subset of $\cG^{k,\alpha}(M)$ in the Baire Category sense, i.e. $\cG_0^{k,\alpha}(M)$ contains a countable intersection of open dense subsets in $\cG^{k,\alpha}(M)$.
        \end{Cor}
        \begin{proof}
            In Theorem \ref{Thm_Countable Decomp}, by choosing $\kappa$ to be $\kappa_0$ in Theorem \ref{Thm_Generic Semi-nondeg_Loc}, we have 
            \[
                \cM^{k,\alpha}(M) = \bigcup_{j\geq 1}\cL^{k,\alpha}(g_j, \Sigma_j; \Lambda_j, \kappa_j),   
            \]
            where $\kappa_j = \kappa_0(g_j, \Sigma_j; \Lambda_j)$. For each $\cL^{k,\alpha}(g_j, \Sigma_j; \Lambda_j, \kappa_j)$, Theorem \ref{Thm_Generic Semi-nondeg_Loc} generates an open dense subset $\cG^{k, \alpha}(g_j, \Sigma_j; \Lambda_j, \kappa_j)$, and we can define
            \[
                \tilde{\cG}^{k,\alpha}(M) := \bigcap^\infty_{j = 1}\cG^{k, \alpha}(g_j, \Sigma_j; \Lambda_j, \kappa_j),
            \]
            which by definition is clearly generic.
            
            By White's structure theorem \cite{white_space_1991,white_bumpy_2017}, there exists a generic subset $\cG^{k, \alpha}_{\operatorname{bumpy}}(M)$ of $\cG^{k,\alpha}(M)$, such that every regular connected minimal hypersurface $\Sigma$ under $g$ is strongly nondegenerate.
            
            It suffices to show that $\tilde{\cG}^{k,\alpha}(M) \cap \cG^{k, \alpha}_{\operatorname{bumpy}}(M) \subset \cG_0^{k,\alpha}(M)$. Indeed, for any $g \in \tilde{\cG}^{k,\alpha}(M) \cap \cG^{k, \alpha}_{\operatorname{bumpy}}(M) \subset \cG_0^{k,\alpha}(M)$, and any LSMH $\Sigma \subset (M, g)$, there exists a $j \in \mathbb{N}^+$ such that $(g, \Sigma) \in \cL^{k,\alpha}(g_j, \Sigma_j; \Lambda_j, \kappa_j)$. Hence, by the fact that $g \in \tilde{\cG}^{k,\alpha}(M) \subset \cG^{k, \alpha}(g_j, \Sigma_j; \Lambda_j, \kappa_j)$, we have $\Sigma$ is semi-nondegenerate. If further $\Sigma$ is regular, then by the definition of $\cG^{k, \alpha}_{\operatorname{bumpy}}(M)$, it is also strongly non-degenerate. In summary, the inclusion holds and $\cG_0^{k,\alpha}(M)$ is also generic.
        \end{proof}

    \section{Proof of Main Theorem} \label{Sec_Pf of Main Thm}
        The goal of this section is to prove Theorem \ref{Thm_Main}. Recall that $\cG_0^{k,\alpha}(M)$ defined in Corollary \ref{Cor_Canon Neighb_Global generic semi-nondeg} is shown to be a generic subset of $\cG^{k, \alpha}(M)$.
        
        Now for $\Lambda >0$ and integer $I\geq 1$, we can define
        \begin{align}
            \begin{split}
                \cG^{k, \alpha}_{\Lambda, I}(M) := \big\{g\in \cG_0^{k, \alpha}(M):& \text{ every }\Sigma \text{ with } (g, \Sigma)\in \cM^{k, \alpha}(M), \ind(\Sigma)\leq I-1 \\
                & \text{ and }\|\Sigma\|_g(M)\leq \Lambda \text{ is regular}\big\}.
            \end{split}  \label{Equ_Def_cG_(Lambda, I)}
        \end{align}
        For simplicity, we may omit the superscript $k, \alpha$. For any $\Lambda > 0$, we also write 
        \begin{align}
            \cG_{\Lambda, 0}(M) := \cG^{k, \alpha}_0(M).  \label{Equ_Def_cG_(Lambda, 0)}
        \end{align}
        \begin{Lem} \label{Lem_cG_I open dense}
            For every $\Lambda>0$ and integer $I\geq 0$, $\cG_{\Lambda, I}(M)$ is relatively open and dense in $\cG_0^{k, \alpha}(M)$.
        \end{Lem}
        
        Given $g\in \cG_0^{k, \alpha}(M)$, $\Lambda>0$ and integer $I\geq 0$, let 
        \begin{align*}
            \cM(g; \Lambda, I)       & := \{(g, \Sigma)\in \cM^{k, \alpha}(M): \|\Sigma\|_g(M)\leq \Lambda, \ind(\Sigma) = I\}; \\
            \SCAP(g; \Lambda, I)     & := \begin{cases}
                \sup\{\SCAP(\Sigma, M, g): (g, \Sigma)\in \cM(g; \Lambda, I)\},\  & \text{ if }\cM(g; \Lambda, I)\neq \emptyset; \\
                0,\                                                               & \text{ if } \cM(g; \Lambda, I) = \emptyset ;
            \end{cases}                                                            \\
            \cM^{max}(g; \Lambda, I) & := \{(g, \Sigma)\in \cM(g; \Lambda, I): \SCAP(\Sigma, M, g) = \SCAP(g; \Lambda, I)\}.
        \end{align*}
        Note that by definition, for every $g\in \cG_{\Lambda, I}(M)$, 
        \begin{align}
            \SCAP(g; \Lambda, I) = 0 \  \text{ if and only if }\  g\in \cG_{\Lambda, I+1}(M).  \label{Equ_g in cG_(I+1) iff SCap(g; I) = 0}
        \end{align}
        
        \begin{Lem} \label{Lem_SCap(g; Lambda, I)<infty}
            Let $\Lambda>0$, integer $I\geq 0$ and $g\in \cG_{\Lambda, I}(M)$.  If $\{(g, \Sigma_j)\}_{j\geq 1}$ is a sequence in $\cM(g; \Lambda, I)$, then $|\Sigma_j|$ subconverges to $m|\Sigma_\infty|$, where 
            \begin{enumerate} [(i)]
                \item $(g, \Sigma_\infty)\in \cM(g; \Lambda, I)$;
                \item Either $m = 1$, or $m\geq 2$ and $I=0$.
            \end{enumerate}
            In particular,  $\SCAP(g; \Lambda, I)<+\infty$.
        \end{Lem}
        \begin{proof}
            By \cite{schoen_regularity_1981} and connectedness of $\Sigma_j$, $|\Sigma_j|$ $\mbfF$-subconverges to $m|\Sigma_\infty|$ for some connected LSMH $\Sigma_\infty \subset (M, g)$ and integer $m\geq 1$ with 
            \[
                m \|\Sigma_\infty\|_g(M) \leq \Lambda, \quad m\ind(\Sigma_\infty)\leq I\,.
            \]
            To see (i), suppose for contradiction $\ind(\Sigma_\infty)\leq I-1$, and hence $I\geq 1$. Since $g\in \cG_{\Lambda, I}(M)$, we know from definition that $\Sigma_\infty$ is regular and strongly nondegenerate.  Hence by \cite{sharp_compactness_2017}, $m = 1$ and $\Sigma_j\equiv \Sigma_\infty$ for $j>>1$.  This contradicts to the fact that $\ind(\Sigma_j) = I$. Therefore, 
            \[
                \ind(\Sigma_\infty) = I\,.
            \]
            (ii) also follows from \cite{sharp_compactness_2017}.  
            
            By Theorem \ref{Thm_Pre_SCAP under converg w multip}, for any $\{(g, \Sigma_j)\}_{j\geq 1}\subset \cM(g; \Lambda, I)$, up to a subsequence,
            \[
                \limsup_{j\to \infty} \SCAP(\Sigma_j, M, g) \leq \SCAP(\Sigma_\infty, M, g) <+\infty\,.
            \]
            Hence, we have $\SCAP(g; \Lambda, I)<+\infty$.
        \end{proof}
        
        \begin{Lem} \label{Lem_cM^(max) finite}
            For every $\Lambda>0$, integer $I\geq 0$ and every $g\in \cG_{\Lambda, I}(M)$, $\cM^{max}(g; \Lambda, I)$ is a finite set.
        \end{Lem}
        \begin{proof}
            Suppose otherwise, there exists a pairwise distinct sequence 
            \[
                \{(g, \Sigma_j)\}_{j\geq 1} \subset \cM^{max}(g; \Lambda, I)\,.
            \]
            By \cite{schoen_regularity_1981} and Lemma \ref{Lem_SCap(g; Lambda, I)<infty}, $|\Sigma_j|$ $\mbfF$-subconverges to $m|\Sigma_\infty|$ for some $(g, \Sigma_\infty)\in \cM(g; \Lambda, I)$.  And by the definition of $\cM^{max}$, we have 
            \begin{align}
                \SCAP(\Sigma_\infty, M, g) \leq \SCAP(g; \Lambda, I) = \SCAP(\Sigma_j, M, g).  \label{Equ_cM^(max) finite_SCap(Sigma_infty) leq SCap(Sigma_j)}
            \end{align}
            Then by Theorem \ref{Thm_Pre_SCAP under converg w multip}, either $m = 1$ or $\Sing(\Sigma_\infty) = \emptyset$. Since $g\in \cG_{\Lambda, I}(M) \subset \cG_0^{k, \alpha}(M)$, the latter case can't happen by the strong nondegeneracy of $\Sigma_\infty$. In other words, we have $m=1$.  
            
            Applying Theorem \ref{Thm_Reg Deform Thm_OneTwo sided} to $\Sigma_j \to \Sigma_\infty$, we know that $\exists\ p\in \Sing(\Sigma_\infty)$ near which $\Sing(\Sigma_j)$ has no singularity for infinitely many $j>>1$.  Then  by Theorem \ref{Thm_Pre_SCAP under converg w multip}, \[
                \liminf_{j\to \infty}\SCAP(\Sigma_j, M, g) \leq \SCAP(\Sigma_\infty, M, g) -1,   \]
            which contradicts to (\ref{Equ_cM^(max) finite_SCap(Sigma_infty) leq SCap(Sigma_j)}). 
        \end{proof}
        
        \begin{Lem} \label{Lem_SCap(g; Lambda, I) strict decrease}
            For every $\Lambda>0$, integer $I\geq 0$ and every $g\in \cG_{\Lambda, I}(M)\cap \Int(\Clos(\cG_{\Lambda, I}(M)))$ such that $\SCAP(g; \Lambda, I)\geq 1$, there exists a sequence of metrics $g_j\in \cG_{\Lambda, I}(M)\cap \Int(\Clos(\cG_{\Lambda, I}(M)))$, $j\geq 1$, such that $g_j \to g$ in $C^{k, \alpha}$ and 
            \begin{align}
                \SCAP(g_j; \Lambda, I) \leq \SCAP(g; \Lambda, I) - 1.  \label{Equ_SCap(g_j) leq SCap(g)-1}
            \end{align}
        \end{Lem}
        \begin{proof}
            By Lemma \ref{Lem_cM^(max) finite}, we can write 
            \[
                \cM^{max}(g; \Lambda, I) =: \{(g, \Sigma^{(1)}), (g, \Sigma^{(2)}), \dots, (g, \Sigma^{(q)})\}\,.
            \]
            For each $1\leq l\leq q$, let $\scF^{(l)}\subset C^{k, \alpha}(M)$ be the open dense subset associated to $(g, \Sigma^{(l)})$ given by Theorem \ref{Thm_Reg Deform Thm_OneTwo sided}.  Thus $\scF:= \bigcap_{l=1}^q \scF^{(l)} \neq \emptyset$.
            
            Take $f\in \scF$, and there exists a sequence of metrics 
            \[
                g_j = (1+ f_j/j)g \in \cG_{\Lambda, I}\cap \Int(\Clos(\cG_{\Lambda, I}(M)))
            \] such that $f_j\to f$ in $C^{k, \alpha}$ when $j\to \infty$.  We now claim that by passing to a subsequence, $g_j$ satisfies (\ref{Equ_SCap(g_j) leq SCap(g)-1}).
            
            To see this, suppose for contradiction that when $j>>1$, there exists $(g_j, \Sigma_j)\in \cM(g_j; \Lambda, I)$ such that 
            \begin{align}
                \SCAP(\Sigma_j, M, g_j) \geq \SCAP(g; \Lambda, I) \geq 1 \,.  \label{Equ_Induct_SCap(g_j) geq SCap(g)}
            \end{align}
            By \cite{schoen_regularity_1981} and the connectedness of $\Sigma_j$, $|\Sigma_j|$ $\mbfF$-subconverges to $m|\Sigma_\infty|$ for some connected LSMH $\Sigma_\infty \subset (M, g)$ and integer $m\geq 1$ with $m \|\Sigma_\infty\|_g(M) \leq \Lambda$ and $m\ind(\Sigma_\infty)\leq I$. 
            
            We shall deal with the following cases separately.\\
            
            \paragraph{\underline{\textit{Case I}}} $\ind(\Sigma_\infty) = I$ and $m=1$. Then $(g, \Sigma_\infty)\in \cM(g; \Lambda, I)$, and by Theorem \ref{Thm_Pre_SCAP under converg w multip} (iii) and (\ref{Equ_Induct_SCap(g_j) geq SCap(g)}), \[
                \SCAP(\Sigma_\infty, M, g) \geq \limsup_{j\to \infty} \SCAP(\Sigma_j, M, g_j) \geq \SCAP(g; \Lambda, I).  \]
            This means $(g, \Sigma_\infty)\in \cM^{max}(g; \Lambda, I)$. WLOG $\Sigma_\infty = \Sigma^{(1)}$.  But since $f\in \scF \subset \scF^{(1)}$, by the choice of $g_j$ and Theorem \ref{Thm_Reg Deform Thm_OneTwo sided} we have,
            \begin{align*}
                \SCAP(\Sigma_\infty, M, g) - 1 & \geq \limsup_{j\to \infty} \SCAP(\Sigma_j, M, g_j)     \\
                                               & = \SCAP(g; \Lambda, I) = \SCAP(\Sigma_\infty, M, g)\,,
            \end{align*}
            which is a contradiction.\\
            
            \paragraph{\underline{\textit{Case II}}} $\ind(\Sigma_\infty) = I$ and $m\geq 2$.  Then $I = 0$ and $(g, \Sigma_\infty)\in \cM(g; \Lambda, I)$.
            \begin{itemize}
                \item \underline{\textit{Case II (1)}}. $\Sing(\Sigma_\infty) = \emptyset$. Since $\Sigma_j$'s are stable minimal hypersurfaces, by \cite{schoen_regularity_1981}, for large enough $j$, $\Sigma_j$ is a multi-graph over $\Sigma_\infty$, and thus,
                      \[
                          \Sing(\Sigma_j) = \emptyset\,,
                      \]
                      contradicting (\ref{Equ_Induct_SCap(g_j) geq SCap(g)}).
                      
                \item \underline{\textit{Case II (2)}}. $\Sing(\Sigma_\infty) \neq \emptyset$. By Theorem \ref{Thm_Pre_SCAP under converg w multip} (iv) and (\ref{Equ_Induct_SCap(g_j) geq SCap(g)}) we have,
                      \begin{align*}
                          \SCAP(\Sigma_\infty, M, g) & \geq \limsup_{j\to \infty} \SCAP(\Sigma_j, M, g_j) + 1              \\
                                                     & \geq \SCAP(g; \Lambda, I) + 1 \geq \SCAP(\Sigma_\infty, M, g) +1\,,
                      \end{align*}
                      which is a contradiction.\\
            \end{itemize}
            
            \paragraph{\underline{\textit{Case III}}} $\ind(\Sigma_\infty) \leq I-1$ and $m = 1$.  Then $I\geq 1$ and by the definition of $g\in \cG_{\Lambda, I}(M)$, $\Sigma_\infty$ is both regular and strongly nondegenerate.  By Allard's regularity theorem \cite{allard_first_1972}, $\Sing(\Sigma_j) = \emptyset$ for $j>>1$, contradicting (\ref{Equ_Induct_SCap(g_j) geq SCap(g)}).\\
            
            \paragraph{\underline{\textit{Case IV}}} $\ind(\Sigma_\infty, g) \leq I-1$ and $m\geq 2$. Then again $I\geq 1$ and $\Sigma_\infty$ is regular and strongly nondegenerate.  By Sharp's compactness theorem \cite{sharp_compactness_2017}, either $\Sigma_\infty$ is two-sided and $\{\Sigma_j\}$ induces a bounded positive Jacobi field on $\Sigma_\infty$, or $\Sigma_\infty$ is one-sided and after passing to double covers, $\{\Sigma_j\}$ induces a bounded positive Jacobi field on the two-sided double cover $\hat{\Sigma}_\infty$ of $\Sigma_\infty$. Both violate the strong nondegeneracy of $\Sigma_\infty$.
        \end{proof}
        
        \begin{proof}[Proof of Lemma \ref{Lem_cG_I open dense}]
            To prove the relative openness, let's consider an arbitrary family of metrics $\{g_j\}_j \subset \cG_0^{k, \alpha}(M)\setminus \cG_{\Lambda, I}(M)$ such that 
            \[
                g_j\to g_\infty \in \cG_0^{k, \alpha}(M) \text{ in } C^{k,\alpha}\,.
            \] The choice of $g_j$ guarantees that there exists an LSMH $\Sigma_j \subset (M, g_j)$ with $\ind(\Sigma_j)\leq I-1$, $\vol(\Sigma_j)\leq \Lambda$ and $\Sing(\Sigma_j) \neq \emptyset$.  By Sharp's compactness again, $|\Sigma_j|$ $\mbfF$-subconverges to some LSMH with multiplicity $m|\Sigma_\infty|$ with area $\leq \Lambda$ and index $\leq I-1$. We now assert that $g_\infty\notin \cG_{\Lambda, I}(M)$, to conclude the proof of relative openness.  If $g_\infty \in \cG_{\Lambda, I}(M)$, by definition, $\Sigma_\infty$ is both regular and strongly nondegenerate. Hence $m=1$ and then by Allard's regularity theorem \cite{allard_first_1972}, for sufficiently large $j$, $\Sing(\Sigma_j) = \emptyset$, which contradicts the choice of $\Sigma_j$.
            
            To prove denseness, we shall do the induction in $I\geq 0$.
            
            \begin{itemize}
                \item Case $I = 0$ follows directly from (\ref{Equ_Def_cG_(Lambda, 0)}).
                \item Now suppose $I\geq 0$ such that $\cG_{\Lambda, I}(M)$ is dense in $\cG_0^{k, \alpha}(M)$, and then by Corollary \ref{Cor_Canon Neighb_Global generic semi-nondeg}, it's also dense in $\cG^{k, \alpha}(M)$.  To prove the denseness of $\cG_{\Lambda, I+1}(M)$, it suffices to show that every metric $g\in \cG_{\Lambda, I}(M)$ can be approximated by metrics in $\cG_{\Lambda, I+1}(M)$.
                      
                      Note that by inductive assumption, $\Clos(\cG_{\Lambda, I}(M)) = \cG^{k, \alpha}(M)$. By Lemma \ref{Lem_SCap(g; Lambda, I)<infty}, $\SCAP(g; \Lambda, I) < \infty$. For any $\epsilon > 0$, we can apply Lemma \ref{Lem_SCap(g; Lambda, I) strict decrease} finitely many times to find a metric $\bar{g}_{\epsilon}\in \cG_{\Lambda, I}(M)$ such that $\|g - g_j\|_{C^{k, \alpha}}< \epsilon$ and $\SCap(\bar{g}_j; \Lambda, I) = 0$.  By (\ref{Equ_g in cG_(I+1) iff SCap(g; I) = 0}), $\bar{g}_{\epsilon}\in \cG_{\Lambda, I+1}(M)$. In conclusion, the arbitrariness of $\epsilon$ implies that $\cG_{\Lambda, I+1}(M)$ is dense in $\cG_{\Lambda, I}(M)$ and thus dense in $\cG_0^{k, \alpha}(M)$.
            \end{itemize}
        \end{proof}

        \begin{proof}[Proof of Theorem \ref{Thm_Main}]
            By Corollary \ref{Cor_Canon Neighb_Global generic semi-nondeg} and Lemma \ref{Lem_cG_I open dense}, we know that \[
                \cG^{k, \alpha}_\infty(M) := \bigcap_{\text{integer } \Lambda>0, I\geq 0}\cG_{\Lambda, I}(M),   \]
            is a generic subset of $\cG^{k, \alpha}(M)$ in the Baire category sense.  By  (\ref{Equ_Def_cG_(Lambda, I)}), (\ref{Equ_Def_cG_(Lambda, 0)}) and the definition of $\cG_0^{k,\alpha}(M)$ in Corollary \ref{Cor_Canon Neighb_Global generic semi-nondeg}, for every $g\in \cG^{k, \alpha}_\infty(M)$, every LSMH $\Sigma \subset (M, g)$ is regular and strongly non-degenerate.
        \end{proof}

    \section{Applications in Almgren-Pitts min-max theory}\label{Sec_Appl}
        
        As mentioned in the Introduction, one obstruction to extend results of (Almgren-Pitts) min-max minimal hypersurfaces to higher dimension is the absence of an analogous White's structure theorem for singular minimal hypersurfaces in general. Fortunately, we have proved generic smoothness for embedded locally stable minimal hypersurfaces in an eight-dimensional closed manifold. This, in principle, suggests that White's structure theorem still holds in dimension 8.
        
        In this section, we shall sketch proofs of various properties of min-max minimal hypersurfaces in an eight-dimensional closed manifold, as applications of our main generic smoothness result. In the following, we will use $\mathcal{G}_{\mathrm{reg}}(M)$ to denote the set of all metrics where every embedded locally stable minimal hypersurface is smooth, and $\mathcal{G}_{\mathrm{reg}, \Lambda, I}$ to denote the set of all metrics where every embedded locally stable minimal hypersurface with both Morse index upper bound $I$ and area upper bound $\Lambda$ is smooth.
        
        \subsection{Generic Denseness}
            In 2017, using the Weyl law for volume spectra proved by Liokumovich-Marques-Neves \cite{liokumovich_weyl_2018}, Irie-Marques-Neves \cite{irie_density_2018} showed that for a generic metric $g$ on a closed Riemannian manifold with dimension between $3$ and $7$, the union of all closed, smooth, embedded minimal hypersurfaces is a dense subset in the ambient manifold. In 2019, the first author \cite{li_existence_2019} improved the density result and proved that the union of all min-max minimal hypersurfaces is also dense under the same dimensional conditions.
            
            The denseness result can be extended to dimension 8 with the aid of generic smoothness.
            
            \begin{Thm}[Density of min-max minimal hypersurfaces]
                For a generic metric on a closed manifold $M^8$, the union of all (smooth) minimal hypersurfaces generated from Almgren-Pitts min-max theory is dense in $M$. 
            \end{Thm}
            \begin{proof}[Sketch of proof]
                Near the end of \cite{li_existence_2019}, it is noted that the generic density of all min-max minimal hypersurfaces with optimal regularity $\mathcal{APR}(M, g)$ follows from Conjecture 4.2 therein. 
                
                Theorem \ref{Thm_Main} together with Sharp's compactness theorem \cite{sharp_compactness_2017} implies that for generic metric $g$ on $M^8$, there are only countably many embedded locally stable minimal hypersurfaces, i.e., countably many singular minimal hypersurfaces with optimal regularity. This confirms Conjecture 4.2, and thus, the generic density result holds.
            \end{proof}

        \subsection{Multiplicity One}
            
            In \cite{marques_morse_2016}, F. C. Marques and A. Neves conjectured that for bumpy metrics, min-max minimal hypersurfaces should have multiplicity exactly one, and they proved this for the first width minimal hypersurface. This multiplicity one conjecture was confirmed by X. Zhou \cite{zhou_multiplicity_2019} for dimensions $3 \leq n+1 \leq 7$, where he applied the min-max construction for hypersurfaces with prescribed mean curvature developed by J. Zhu and himself \cite{zhou_existence_2018}.
            
            Let's denote by $\mathcal{G}_{\mathrm{bumpy}}(M)$ all the metrics where every immersed smooth minimal hypersurface is non-degenerate. B. White has proved that it is also generic in the Baire sense \cite{white_space_1991,white_bumpy_2017}.
            
            \begin{Thm}[Multiplicity one]
                Given a closed manifold $M^8$, for any metric $g \in \mathcal{G}_{\mathrm{reg}}(M) \cap \mathcal{G}_{\mathrm{bumpy}}(M)$ and any $k \in \mathbb{N}$, there exists a disjoint collection of smooth, connected, closed, embedded, two-sided minimal hypersurfaces $\set{\Sigma^k_i: i = 1, \cdots, l_k}$, such that
                \[
                    \omega_k(M, g) = \sum^{l_k}_{i = 1} \vol_g(\Sigma^k_i)\,,\quad \sum^{l_k}_{i = 1} \ind(\Sigma^k_i) \leq k\,.
                \]
                where $\omega_k(M, g)$ is the volume spectrum.
            \end{Thm}
            \begin{proof}[Sketch of proof]
                This essentially follows the proof of \cite[Theorem~5.2]{zhou_multiplicity_2019}, and it suffices to show that we can construct suitable min-max PMC hypersurfaces to approximate min-max minimal hypersurfaces as in the proof of \cite[Theorem~4.1]{zhou_multiplicity_2019}.
                
                First, since every minimal hypersurface in $(M, g)$ is smooth and nondegenerate, we can also find good $h \in \mathcal{S}(g)$ as in Lemma 4.2 therein.
                
                Then, the existence and regularity of $h$-PMC hypersurfaces with optimal regularity follows from \cite{zhou_existence_2018} together with the regularity theory by Bellettini-Wickramasekera \cite{bellettiniStablePrescribedmeancurvatureIntegral2020}.
                
                Finally, we can obtain the Morse index upper bound from the constrained min-max theory in \cite{li_improvedMorseindex_2020}. This resolves the difficulty of the lack of good pairs for singular PMC hypersurfaces as in Section 3.3 therein.
            \end{proof}
        
        \subsection{Morse Index}
            In \cite{marques_morse_2021}, Marques-Neves shows that for bumpy metrics on $M^{n+1} (3 \leq n+1 \leq 7)$, the Morse index of a $p$-width min-max minimal hypersurface of multiplicity one exactly equals to $p$. This together with X. Zhou's affirmation of multiplicity one conjecture gives a fine geometric description of min-max minimal hypersurfaces.
            
            \begin{Thm}[Morse index estimates]
                Given a closed manifold $M^8$, for any metric $g \in \mathcal{G}_{\mathrm{reg}}(M) \cap \mathcal{G}_{\mathrm{bumpy}}(M)$ and any $k \in \mathbb{N}$, there exists a disjoint collection of smooth, connected, closed, embedded, two-sided minimal hypersurfaces $\set{\Sigma^k_i: i = 1, \cdots, l_k}$, such that
                \[
                    \omega_k(M, g) = \sum^{l_k}_{i = 1} \vol_g(\Sigma^k_i)\,,\quad \sum^{l_k}_{i = 1} \ind(\Sigma^k_i) = k\,.
                \]
                where $\omega_k(M, g)$ is the volume spectrum.
            \end{Thm}
            \begin{proof}[Sketch of proof]
                Most of the proofs in \cite{marques_morse_2021} can also work in our setting, and it suffices to prove Proposition 8.5 therein for $M^8$.
                
                Indeed, one only needs to replace $\mathcal{U}_{p, \alpha}$ by $\mathcal{U}_{p, \alpha} \cap \mathcal{G}_{\mathrm{reg}, p, \alpha}(M)$. Then the denseness follows from the fact that both $\mathcal{U}_{p, \alpha}$ and $\mathcal{G}_{\mathrm{reg}, p, \alpha}(M)$ are residual. It is also easy to prove the openness of $\mathcal{U}_{p, \alpha} \cap \mathcal{G}_{\mathrm{reg}, p, \alpha}(M)$ with the aid of Allard's regularity theorem \cite{allard_first_1972}. Hence, Propostion 8.5 therein still holds for $M^8$.
            \end{proof}

        \subsection{Geometric Complication}
            In \cite{chodosh_minimal_2019}, Chodosh-Mantoulidis shows that for bumpy metrics $M^{n+1} (3 \leq n+1 \leq 7)$, there exists a sequence of connected minimal hypersurfaces whose areas tend to infinity. Later in \cite{li_existence_2020}, we improve their results and show that the sequence can have unbounded Morse indices simultaneously. With the generic full regularity in dimension $8$, we can also extend the result to dimension $8$ via the same arguments.

            \begin{Thm}[Main Theorem]
                Given any closed Riemannian manifold $M^8$ with a generic metric $g \in \mathcal{G}_{\mathrm{reg}}(M) \cap \mathcal{G}_{\mathrm{bumpy}}(M)$, there exists a sequence of connected closed embedded two-sided minimal hypersurfaces $\set{\Gamma_i}$ such that
                \begin{equation}
                    \lim_{i \rightarrow \infty} \mathrm{Area}(\Gamma_i) = +\infty\,, \quad \lim_{i \rightarrow \infty} \mathrm{index}(\Gamma_i) = +\infty\,.
                \end{equation}
            \end{Thm}

\part{Technical Work} \label{Part_Tech}

    \section{Analysis on Stable Minimal Hypercones} \label{Sec_Analysis on SMC}
        As in Section \ref{Subsec_Stable MHC}, we let $\mbfC\subset \RR^{8}$ be a stable minimal hypercone with cross section $S := \mbfC\cap \SSp^7$ which is a regular minimal hypersurface in $\SSp^7$. $\mbfC$ will be parametrized by $(0, +\infty)\times S \to \mbfC$, $(r, \omega)\mapsto r\omega$ as in (\ref{Equ_Pre_Parametrize minimal hypercone}).
        
        To begin with, we discuss the following growth rate monotonicity formula as an application of the decomposition (\ref{Equ_Pre_Decomp Jac into homogen Jac field}), which was in fact first introduced in \cite{simonAsymptoticsClassNonLinear1983} in a different form.
        
        \begin{Lem} \label{Lem_Ana on SMC_Growth Rate Monoton}
            For $\sigma>0$, $\Lambda>0$, there exists $K = K(\sigma, \Lambda)>2$ with the following property.  
            
            For $\gamma\in \RR$ and any stable hypercone $\mbfC\in \scC_\Lambda$ satifying 
            \[
                \dist_\RR (\gamma, \Gamma(\mbfC)\cup \{-(n-2)/2\}) \geq \sigma\,,
            \] suppose that $u\in C^2_{loc}(A(K^{-3}, 1))\cap L^2(A(K^{-3}, 1))$ solves $L_\mbfC u = 0$, where $A(r, s):= \AAa(\mathbf{0}, r, s)\cap \mbfC$. Then its growth rate \[
                J_K^\gamma (u; r):= \int_{A(K^{-1}r, r)} u^2\cdot |x|^{-2\gamma - n} d\|\mbfC\|\,,  \]
            satisfies \[
                J_K^\gamma(u; K^{-2}) - 2 J_K^\gamma(u; K^{-1}) + J_K^\gamma(u; 1) \geq 0\,.   \]
            Moreover, the equality holds if and only if $u\equiv 0$.
        \end{Lem}
        This Lemma is a refinement of \cite[Lemma 2.6]{WangZH20_Deform}. The proof is explicit but kind of tedious, so is left to Appendix \ref{Sec_App_Proof Growth Rate Mon}.
        
        \begin{Cor}[Perturbed version of Lemma \ref{Lem_Ana on SMC_Growth Rate Monoton}] \label{Cor_Ana on SMC_Growth Rate Monoton with Perturb}
            For $\sigma>0$ and $\Lambda>0$, let $K(\sigma, \Lambda)>2$ be the same as in Lemma \ref{Lem_Ana on SMC_Growth Rate Monoton}.  Then there exists $\varepsilon(\sigma, \Lambda)>0$ small enough satisfying the following property.
            
            For $\gamma \in (-n+1, 1)$ and any stable hypercone $\mbfC\in \scC_\Lambda$ satifying 
            \[
                \dist_\RR (\gamma, \Gamma(\mbfC)\cup \{-(n-2)/2\}) \geq \sigma\,,
            \] suppose that $0\neq u\in W^{1,2}_{loc}(A(K^{-3}, 1))\cap L^2(A(K^{-3}, 1))$ is a weak solution to 
            \begin{align}
                \operatorname{div}_\mbfC (\nabla_\mbfC u + \vec{b}_0(x)) + |A_\mbfC|^2 u + b_1(x) = 0,  \label{Equ_Pre_Pert Jac equ}
            \end{align}
            where $\vec{b}_0, b_1$ satisfy the following estimates almost everywhere,
            \begin{align}
                |\vec{b}_0|(x) + |b_1|(x) \leq \varepsilon (|u|(x) + |\nabla u|(x)),  \label{Equ_Pre_Small error in pert Jac equ}
            \end{align}
            on $A(K^{-3}, 1)$.   Then, \[
                J_K^\gamma(u; K^{-2}) - 2(1+\varepsilon) J_K^\gamma(u; K^{-1}) + J_K^\gamma(u; 1) > 0.   \]
        \end{Cor}
        \begin{Rem} \label{Rem_Pre_Jac field equ on MH near cone}
            When $\Sigma := \graph_\mbfC(\phi)\cap \AAa(\mathbf{0}, K^{-3}/2, 2)$ is a minimal hypersurface under the metric $g$ parametrized by the cone $\mbfC$, with \[
                \|\phi\|_{C^2_*} + \|g-g_{\eucl}\|_{C^2} \leq \epsilon,
            \]
            then by Appendix \ref{Sec_App_Geom of Minimal Graph}, any Jacobi field $u$ on $\Sigma$, viewed as a function over $\mbfC$ under this parametrization, solves an equation of form (\ref{Equ_Pre_Pert Jac equ}) with \[
                |\vec{b}_0| + |b_1| \leq \Psi(\epsilon|\mbfC, K)\cdot (|u| + |\nabla u|).   \]
            where $\Psi(\epsilon|\mbfC, K)$ is a constant tending to $0$ when $\mbfC, K$ is fixed and $\epsilon\to 0$.
        \end{Rem}
        \begin{proof}[Proof of Corollary \ref{Cor_Ana on SMC_Growth Rate Monoton with Perturb}]
            First note that by multiplying (\ref{Equ_Pre_Pert Jac equ}) with $u\cdot \eta^2$ and integration by parts ($\eta$ is some smooth cut-off function) and using (\ref{Equ_Pre_Small error in pert Jac equ}) with $\varepsilon< 1/2$, we have for every domain $\Omega\subset\subset A(K^{-3}, 1)$, 
            \begin{align}
                \int_{\Omega} |\nabla_\mbfC u|^2 d\|\mbfC\| \leq C(n, K, \Omega)\int_{A(K^{-3}, 1)}(1+|A_\mbfC|^2)u^2.  \label{Equ_Pre_Energy est for pert Jac equ}
            \end{align}
            Now to prove the corollary, suppose for contradiction that there exist $\sigma>0$, $\Lambda>0$, a sequence of $\mbfC_j\in \scC_\Lambda$, a sequence of real number $\set{\gamma_j}\subset (-n+1, 1)$ with 
            \[
                \dist_\RR(\gamma_j, \Gamma(\mbfC_j)\cup\{-(n-2)/2\})\geq \sigma\,,
            \] and a sequence of nonzero function $u_j\in W^{1,2}_{loc}(A(K^{-3},1)) \cap L^2(A(K^{-3}, 1))$ satisfying (\ref{Equ_Pre_Pert Jac equ}) and (\ref{Equ_Pre_Small error in pert Jac equ}) with $\mbfC = \mbfC_j$ and $\varepsilon = 1/j$, but 
            \begin{align}
                J_K^{\gamma_j}(u_j, K^{-2}) - 2(1+\frac{1}{j})J_K^{\gamma_j}(u_j, K^{-1}) + J_K^{\gamma_j}(u_j, 1) \leq 0.  \label{Equ_Pre_Contrad assump for growth rate mon}
            \end{align}
            By a renormalization of $u_j$, WLOG, we can assume $J_K^{\gamma_j}(u_j, K^{-1}) = 1$. By the $\mbfF$-compactness of $\scC_\Lambda$, we can assume by passing to a subsequence of $j\to \infty$ that $\mbfC_j$ will $\mbfF$-converge to $\mbfC_\infty\in \scC_\Lambda$, and $\gamma_j\to \gamma_\infty$.  And then by the discussion in Section \ref{Subsec_Stable MHC}, we have 
            \[
                \dist_\RR(\gamma_\infty, \Gamma(\mbfC_\infty)\cup\{-(n-2)/2\})\geq \sigma\,.
            \]
            
            By (\ref{Equ_Pre_Energy est for pert Jac equ}), (\ref{Equ_Pre_Contrad assump for growth rate mon}) and Reilly compactness, we have $u_j$ subconverges to $u_\infty\in W^{1,2}_{loc}(A(K^{-3}, 1))\cap L^2(A(K^{-3}, 1))$ in $L^2_{loc}(\AAa(\mathbf{0}, K^{-3}, 1))$ with 
            \[
                J_K^{\gamma_\infty}(u_\infty, K^{-1}) = 1\,,
            \] and 
            \[
                (\Delta_{\mbfC_\infty} + |A_{\mbfC_\infty}|^2)u_\infty = 0
            \]
            on $A(K^{-3}, 1)$, but (\ref{Equ_Pre_Contrad assump for growth rate mon}) then implies 
            \[
                J_K^{\gamma_\infty}(u_\infty, K^{-2}) - 2J_K^{\gamma_\infty}(u_\infty, K^{-1}) + J_K^{\gamma_\infty}(u_\infty, 1) \leq 0.
            \]
            This contradicts Lemma \ref{Lem_Ana on SMC_Growth Rate Monoton}.
        \end{proof}
        
        The reason that we care about Jacobi fields on $\mbfC$ is that they are models of minimal hypersurfaces lying near $\mbfC$.  Later in this paper, we shall utilize the following result.
        \begin{Lem} \label{Lem_Ana on SMC_AR_infty(C+x; C) = 0}
            Let $\mbfC\subset \RR^{n+1}$ be a regular minimal hypercone and $\mathbf{x}\neq \mathbf{0}$. Then there exists some $R=R(\mbfC, \mathbf{x})>0$ and $u\in C^\infty(\mbfC)$ such that \[
                (\mbfC+\mathbf{x} )\setminus \BB_R = \graph_\mbfC(u)\setminus \BB_R,    \]
            and $\cA\cR_\infty(\mbfC + \mathbf{x}; \mbfC) = 0$, where
            \begin{align*}
                \cA\cR_\infty(\mbfC+ \mathbf{x}; \mbfC) & := \cA\cR_\infty(u; \mbfC)\\
                                                        & := \inf\{\gamma: \limsup_{R\to +\infty} \int_{\AAa(\mathbf{0}, R, 2R)\cap \mbfC}u^2\cdot |x|^{-2\gamma - n} d\|\mbfC\| = 0\}.     
            \end{align*}
        \end{Lem}
        \begin{proof}
            First by implicit function theorem, there exists $\delta(\mbfC)\in (0, 1)$ such that for every $\mathbf{x}\in \RR^{n+1}$ with $|\mathbf{x}|\leq \delta(\mbfC)$, the map \[
                \Pi_{\mathbf{x}}: \mbfC\cap \AAa(\mathbf{0}, 1/4, 4) \to \mbfC, \ \ \ p\mapsto \text{nearest point projection of }p+\mathbf{x} \text{ onto }\mbfC,   \]
            is a diffeomorphism onto its image, and its image contains $\AAa(\mathbf{0}, 1/2,3)\cap \mbfC$. 
            Also, if write $u_{\mathbf{x}}$ be the graphical function of $\mbfC+\mathbf{x}$ over $\mbfC$ in $\AAa(\mathbf{0}, 1/2,3)$, then we know that \[
                u_{\mathbf{x}}(q) = \langle \Pi^{-1}_{\mathbf{x}}(q) -q +\mathbf{x}, \nu_\mbfC(q)\rangle,\ \ \ \forall q\in \AAa(\mathbf{0}, 1, 2).   \]
            And since $\langle\Pi^{-1}_{\mathbf{x}}(q) - q, \nu_\mbfC(q)\rangle = O(|\mathbf{x}|^2)$, we have, 
            \begin{align}
                |u_{\mathbf{x}}(q) - \langle \mathbf{x}, \nu_\mbfC(q)\rangle| \leq C(\mbfC)\cdot|\mathbf{x}|^2.  \label{Equ_Pre_Graph func transl cone model on transl Jac field}
            \end{align}
            
            Now for a general $\mathbf{x}\in \RR^{n+1}\setminus \mathbf{0}$, since $(\mbfC + \mathbf{x})/R = \mbfC + \mathbf{x}/R$, we know that $\mbfC+\mathbf{x}$ is graphical over $\mbfC$ outside the ball $\BB_{|\mathbf{x}|/\delta(\mbfC)}$, and the graphical function satisfies $u_{\mathbf{x}}(p) = R u_{\mathbf{x}/R}(p/R)$.  Hence, for every $|p|> |\mathbf{x}|/\delta(\mbfC)$ we have, 
            \begin{align}
                |u_{\mathbf{x}}(p) - \langle \mathbf{x}, \nu_\mbfC(p) \rangle| = R\cdot|u_{\mathbf{x}/R}(p/R) - \langle \mathbf{x}/R, \nu_\mbfC(p/R) \rangle| \leq C(\mbfC)\cdot|\mathbf{x}|^2/|p|,  \label{Equ_Pre_Graph func transl cone - transl Jac field is h.o.t}
            \end{align}
            where the last inequality follows by taking $R = |p|$ and applying (\ref{Equ_Pre_Graph func transl cone model on transl Jac field}).
            
            Also, observe that $\langle \mathbf{x}, \nu_\mbfC \rangle$ is a homogeneous function on $\mbfC$ of degree $0$, which is not identically $0$ since otherwise, $\mbfC$ splits in $\mathbf{x}$-direction and is not of isolated singularity.  By combining this with (\ref{Equ_Pre_Graph func transl cone - transl Jac field is h.o.t}), one can check that $\cA\cR_\infty(u_{\mathbf{x}}; \mbfC) = 0$ from the definition.
        \end{proof}
        
        For a general minimal hypersurface $\Sigma$ asymptotic to $\mbfC$ near infinity in $\RR^{n+1}$, it follows from \cite[Section 4]{WangZH20_Deform} that $\cA\cR_\infty(\Sigma; \mbfC)\in \{-\infty\}\cup\Gamma(\mbfC)\cap [-\infty, 1]$.  See more results in this line in \cite[Section 4]{WangZH20_Deform}.

        In \cite[Theorem~6.3]{edelen_degeneration_2021}, Edelen proved a quantitative version of the uniqueness of tangent cones for minimal hypersurfaces. We summarize this result below with a different formulation.
        \begin{Lem} \label{Lem_Ana on SMC_Quanti Uniqueness of Tang Cone}
            Given $\epsilon>0$, integer $I\geq 0$ and $\mbfC$ a stable minimal hypercone in $\RR^{n+1 = 8}$, there exists $\delta = \delta(\epsilon, \mbfC, I)>0$ with the following property.
            
            Let $g$ be a $C^4$ metric on $\BB_4^{n+1}$, $V\in \cR_n(\BB_4)$ be an stationary integral varifold supported on LSMH under $g$ with $\ind(\spt(V), g)\leq I$, and let $r\in (0, 1/16)$. Suppose: 
            \begin{enumerate} [(a)]
                \item $\|g-g_{\eucl}\|_{C^4}\leq \delta$;
                \item $\theta_{V}(\mathbf{0}, 2)\leq \theta_{V}(\mathbf{0}, r) + \delta$;
                \item There exists $s\in (2r, 1/2)$ and $u\in C^2(\mbfC)$ with $\|u\|_{C_*^2}\leq \delta$ such that \[
                          V\llcorner A^g(\mathbf{0}; s, 2s) = |\graph_{\mbfC}(u)|\llcorner A^g(\mathbf{0}; s, 2s).   \]
            \end{enumerate}
            Then, there exists $\tilde{u}\in C^2(\mbfC)$ with $\|\tilde{u}\|_{C^2_*}\leq \epsilon$ such that \[ 
                V\llcorner A^g(\mathbf{0}; 2r, 1) = |\graph_{\mbfC}(\tilde u)|\llcorner A^g(\mathbf{0}; 2r, 1).   \]
        \end{Lem}
        \begin{proof}
            It suffices to check the condition in \cite[Theorem 6.3]{edelen_degeneration_2021} holds, i.e. $V\llcorner A^g(\mathbf{0}; 8r, 1)$ is a $(\mbfC, 1, \Psi(\delta|\mbfC, I))$-cone region. Denote for simplicity $V_t:= \eta_{\mathbf{0}, t}{}_\sharp V$.
            
            First of all, by (a) (b) and volume monotonicity, for every $t\in (8r, 1)$, there exists $m_t\in \NN$ and $\mbfC_t\in \scC$ such that 
            \begin{align}
                \mbfF_{\BB_2}(V_t, m_t|\mbfC_t|)\leq \Psi(\delta|\mbfC, I)\,, \label{Equ_Edelen Lem_V_t almost conic}
            \end{align}
            And by (c), one can choose $(\mbfC_s, m_s) = (\mbfC, 1)$. Since for every $r\in (0, 1)$ and every varifolds $V_1, V_2$, \[
              \mbfF_{\BB_2}(\eta_{\mathbf{0}, r}{}_\sharp V_1, \eta_{\mathbf{0}, r}{}_\sharp V_2) \leq \mbfF_{\BB_2}(V_1, V_2)\cdot r^{-1}\,,\,,
            \]
            we then have for every $t\in (8r, 1/2)$,
            \begin{align}
                \begin{split}
                    \mbfF_{\BB_2}(m_{2t}|\mbfC_{2t}|, m_t|\mbfC_t|) & \leq \mbfF_{\BB_2}(m_{2t}|\mbfC_{2t}|, V_t) + \mbfF_{\BB_2}(m_t|\mbfC_t|, V_t) \\
                    & \leq 2\mbfF_{\BB_2}(m_{2t}|\mbfC_{2t}|, V_{2t}) + \mbfF_{\BB_2}(m_t|\mbfC_t|, V_t) \\
                    & \leq 3\Psi(\delta|\mbfC, I).
                \end{split} \label{Equ_Edelen Lem_Compare cone in differ scal}
            \end{align}
            By \cite[Lemma 4.4]{liwang2020generic}, for every $\mbfC_1, \mbfC_2\in \scC_\Lambda$ and every integer $m\geq 2$, we have \[
             \mbfF_{\BB_2}(|\mbfC_1|, m|\mbfC_2|) \geq \epsilon(\Lambda)\,,   \]
            Hence by taking $\delta<<1$, $m_s = 1$ and (\ref{Equ_Edelen Lem_Compare cone in differ scal}) guarantees that $m_t = 1$ for every $t\in (8r, 1)$. Then by Allard regularity \cite{allard_first_1972}, (\ref{Equ_Edelen Lem_V_t almost conic}) implies that $V\llcorner A^g(\mathbf{0}; t/8, t)$ is a one-sheet $\Psi(\delta|\mbfC, I)$-graph over $\mbfC_t$. Edelen's Theorem thus applies by taking $\delta\leq \delta(\epsilon, \mbfC, I)<<1$.
        \end{proof}

        \begin{Rem}
            Using a similar argument as \cite{edelen_degeneration_2021}, one can actually drop the bounded-index assumption in Lemma \ref{Lem_Ana on SMC_Quanti Uniqueness of Tang Cone}, but this is out of the scope of this paper.
        \end{Rem}
        
        This in particular implies the following result of convergence in all scales.
        \begin{Cor} \label{Cor_Converg in all Scales}
            Let $\{\Sigma_j\}_{j\geq 1}$ be a sequence of two-sided LSMH in $(\BB^8_4, g_j)$ containing $\mathbf{0}$, with $\ind(\Sigma_j, g_j)\leq \Lambda$, $\|\Sigma_j\|(\BB_4)\leq \Lambda$. Suppose when $j\to \infty$, 
            \begin{enumerate} [(a)]
                \item $\|g_j - g_{\eucl}\|_{C^4}\to 0$;
                \item $\theta_{|\Sigma_j|}(\mathbf{0}, 2) - \theta_{|\Sigma_j|}(\mathbf{0}) \to 0$.
            \end{enumerate}
            Then there exists some stable minimal hypercone $\mbfC\in \scC_{4^n\Lambda}$ such that after passing to a subsequence, 
            \begin{align*}
                \mbfF_{\BB_3}(|\Sigma_j|, |\mbfC|) \to 0, &  & \mbfF_{\BB_3}(|\mbfC_{\mathbf{0}}\Sigma_j|, |\mbfC|) \to 0,
            \end{align*}
            where $\mbfC_{\mathbf{0}}\Sigma_j$ is the tangent cone of $\Sigma_j$ at $\mathbf{0}$.
        \end{Cor}
        \begin{proof}
            First recall that by \cite[Theorem B]{ilmanen_strong_1996}, tangent cones at a singular point of a LSMH are all of multiplicity $1$.  Denote for simplicity by $\mbfC_j\in \scC_{4^n\Lambda}$ the tangent cone of $\Sigma_j$ at $\mathbf{0}$, and suppose when $j\to \infty$, $\mbfC_j$ $\mbfF$-converges to some $\mbfC\in \scC_{4^n \Lambda}$. By Allard's regularity theorem \cite{allard_first_1972}, the convergence is in $C^2$ sense away from $\mathbf{0}$.  Hence by Lemma \ref{Lem_Ana on SMC_Quanti Uniqueness of Tang Cone}, $|\Sigma_j|$ also $\mbfF$-subconverges to $|\mbfC|$.
        \end{proof}

    \section{Growth Estimate of Graphical Functions} \label{Sec_Growth est}
        Let $(N^8, g)$ be a Riemannian manifold, not necessarily complete. For $i=0, 1$, suppose $\Sigma^i$ is a hypersurface in $N$ such that $\Clos(\Sigma^i) = \partial U^i$, where $U^i\subset N$ is an open subset.  Let $\nu^i$ be the unit normal field of $\Sigma^i$ under metric $g$, which points away from $U^i$.
        
        Define the \textbf{graphical function} over $\Sigma^0$ of $\Sigma^1$ under metric $g$ as follows.  For every $x\in \Sigma^0$,
        \begin{align} 
            G^{\Sigma^1}_{\Sigma^0, g} (x):= \begin{cases}
                \sup\{t\geq 0: \exp^g_x(s\nu^0(x))\in U^1 \text{ for every }s\in [0, t] \},           & \ \text{if }x\in U^1;    \\
                \inf\{t\leq 0: \exp^g_x(s\nu^0(x))\in N\setminus U^1 \text{ for every }s\in [t, 0]\}, & \ \text{if }x\notin U^1.
            \end{cases} \label{Equ_Growth est_Graph Function in general}
        \end{align}
        \begin{Rem} \label{Rem_Growth est_Graph Function}
            We abuse the notation here since the definition of $G^{\Sigma^1}_{\Sigma^0, g}$ relies also on the choice of $U^i$ and $N$. In applications, these are naturally chosen and there will be no ambiguity.
            
            Heuristically, $\Sigma^1$ is the graph of $G^{\Sigma^1}_{\Sigma^0, g}$ over $\Sigma^0$.  In fact, this is the case when $\Sigma^i$ are smooth hypersurfaces and are $C^2$ close.  In general however, $G^{\Sigma^1}_{\Sigma^0, g}$ need not be even $C^0$ but only $L^\infty$ along $\Sigma^0$.
        \end{Rem}

        The goal for this section is to control the behavior of graphical functions of LSMH. We first prove several local estimates, that is, for LSMH in $N = \BB_2^{n+1 = 8}$ which is close to a minimal hypercone.   Then using these estimates, we study the globally induced Jacobi fields associated to a converging sequence of LSMH.
        
    \subsection{Local Growth Estimate}
        In this subsection, the ambient manifold $N$ to define graphical function is fixed to be $\BB_2^{n+1 = 8}$.  When two connected LSMH $\Sigma^0$, $\Sigma^1$ in $\BB_2$ are sufficiently close, and $\Sigma^0 = \partial U^0$ for some $U^0\subset N$, there's a unique choice of domain $U^1$ bounded by $\Sigma^1$ in $N$ such that $U^0\Delta U^1$ is also small.  We shall fix this choice of $U^i$ in defining the graphical functions.
    
    \begin{Lem}[$L^2$ growth estimate] \label{Lem_L^2 Growth Estimate}
        Let $\Lambda>0$, $0<\sigma<\gamma_{\gap, \Lambda}/2$, where $\gamma_{\gap, \Lambda}$ is given by Lemma \ref{Lem_gamma_gap >0}.   There exists $\delta_1 = \delta_1(\sigma, \Lambda)\in (0,1)$ and $K = K(\sigma, \Lambda)>2$ with the following property,
        
        Let $\Sigma^i = \partial U^i$ ($i=0,1$) be LSMH in $(\BB_2, g)$ with area $\leq \Lambda$ and index $\leq \Lambda$. Assume the following holds for $\delta = \delta_1$,
            \begin{enumerate} [(a)]
                \item $\|g-g_{\eucl}\|_{C^4}\leq \delta$; $\|u\|_{L^2(A^g(\mathbf{0}, 1/2,1), \|\Sigma^0\|_g)} \leq \delta$, where $u:= G^{\Sigma^1}_{\Sigma^0, g} \in L^\infty(\Sigma^0)$;
                \item For $i = 0, 1$, there exists $x^i\in \Sing(\Sigma^i)\cap \BB_{1/2}$ such that $x^0 = \mathbf{0}$ and
                      \begin{align*}
                          \theta_{|\Sigma^i|}(x^i, 1) - \theta_{|\Sigma^i|}(x^i) \leq \delta, & \  & \theta_{|\Sigma^1|}(x^1) = \theta_{|\Sigma^0|}(x^0).
                      \end{align*}
            \end{enumerate}
            Then,
            \begin{enumerate} [(i)]
                \item $\Sigma^1\cap A^g(\mathbf{0}, K^{-3}, 3/2) = \graph_{\Sigma^0, g}(u)\cap A^g(\mathbf{0}, K^{-3}, 3/2)$; $x^1\in B^g_{K^{-3}}(\mathbf{0})$ and \[
                          \|u\|_{C^2_*, A^g(\mathbf{0}, K^{-2}, 1)} \leq C(\sigma, \Lambda)\cdot \|u\|_{L^2(\|\Sigma^0\|_g)}.   \]
                \item We have $L^2$-growth estimate for $u$,
                      \begin{align}
                          J^{\gamma_2^+ - \sigma}_K(u; K^{-1}) \leq J^{\gamma_2^+ - \sigma}_K(u; 1); \label{Equ_L^2 Growth Est at Scal 1}
                      \end{align}
                      where $\gamma_2^+ := \gamma_2^+(\mbfC_\mathbf{0}\Sigma^0)$;
                      \begin{align}
                          \begin{split}
                              J^\gamma_K(u; r)^2 & := \|u\|^2_{L^2_\gamma(A^g(\mathbf{0}, K^{-1}r, r); \|\Sigma^0\|_g)}\\
                              & := \int_{A^g(\mathbf{0}, K^{-1}r, r)} u^2(x) \cdot |\dist_g(\mathbf{0}, x)|^{-2\gamma - n}\ d\|\Sigma^0\|_g(x)     
                          \end{split}  \label{Equ_Def J^sigma_K(u, r)}
                      \end{align}
            \end{enumerate}
    \end{Lem}
    The proof shares some similarity as the proof of \cite[Lemma 4.11]{WangZH20_Deform}.  First note that by its definition we have, 
    \begin{align}
        C(K, \gamma)^{-1}\leq \frac{J^\gamma_K(u; r)r^{\gamma + n/2}}{\|u\|_{L^2(A^g(\mathbf{0}, K^{-1}r, r); \|\Sigma^0\|_g)}} \leq C(K, \gamma).  \label{Equ_J equiv to L^2 with r^power}
    \end{align}
    Thus, (\ref{Equ_L^2 Growth Est at Scal 1}) provides a growth rate bound between different scales.
    \begin{proof}
        Take $K = K(\sigma, \Lambda)>2$ as given by Lemma \ref{Lem_Ana on SMC_Growth Rate Monoton}.
        
        Suppose for contradiction, there exists a family of LSMH $\Sigma_j^i \subset (\BB_2, g_j)$ ($i= 0, 1$, $j\geq 1$) satisfying (a) and (b) with area $\leq \Lambda$, index $\leq \Lambda$, $\delta = 1/j$, $g_j \to g_{\eucl}$ in $C^4$ but either (i) or (ii) fails. Let $u_j:= G^{\Sigma^1_j}_{\Sigma^0_j, g_j}$. By unique continuation of LSMH, WLOG $u_j$ is not identically $0$ on $A^{g_j}(\mathbf{0}, 1/2, 1)$.
        
        When $j\to \infty$, by \cite{sharp_compactness_2017} and assumption (a), $|\Sigma_j^i|$ subconverges to an LSMH with multiplicity $V_\infty$, independent of $i = 0, 1$. Moreover, by assumption (b) and Corollary \ref{Cor_Converg in all Scales}, $V_\infty = |\mbfC|\llcorner \BB_2$ for some stable minimal hypercone $\mbfC\in \scC_\Lambda$. Hence for sufficiently large $j$, (i) holds by Allard's regularity theorem \cite{allard_first_1972} and classical elliptic estimates. By the contradiction assumption, (\ref{Equ_L^2 Growth Est at Scal 1}) fails for $u_j$.
        
        Let $\varepsilon = \varepsilon(\sigma, \Lambda)>0$ to be determined later; 
        \begin{align}
            \begin{split}
                l_j  := \sup\big\{ l\in \mathbb{N}^+: &\ (\Sigma_j^1\Delta \graph_{\Sigma_j^0}(u_j))\cap A^g(\mathbf{0}, K^{-l-1}, 1) = \emptyset;\\
                &\ \|u_j\|_{C^2_*(A^g(\mathbf{0}, K^{-l-1}, 1)\cap \Sigma^0_j)} \leq \varepsilon \big\}.  
            \end{split}  \label{Equ_Def l_j}
        \end{align}
        By Allard'sregularity theorem, $l_j\to +\infty$ as $j\to \infty$ ($l_j$ are possibly $+\infty$ themselves).  
        By Lemma \ref{Lem_Ana on SMC_Quanti Uniqueness of Tang Cone}, there exists $\epsilon_j\searrow 0$ such that for every $s\in (0,1)$, $\eta_{\mathbf{0},s}(\Sigma_j^0)$ is $C^2$-$\epsilon_j$ close to $\mbfC$ in $A^{g_j}(\mathbf{0}, 1/4, 1)$. Therefore, we can choose $\varepsilon(\sigma, \Lambda)$ small enough such that Corollary \ref{Cor_Ana on SMC_Growth Rate Monoton with Perturb}, after rescaling, applies to $u_j$ with $\gamma := \gamma_2^+(\mbfC_\mathbf{0}\Sigma_j^0)-\sigma$ in each $A^g(\mathbf{0}, K^{-l-1}, K^{-l+2})$, $2\leq l\leq l_j$.  This provides,
        \begin{align}
            J^{\gamma_2^+ - \sigma}_K(u_j; K^{-l}) > J^{\gamma_2^+ - \sigma}_K(u_j; K^{-l+1}), \ \ \ \forall 1\leq l< l_j-1. \label{Equ_J(l) monotone in 1<l<l_j}
        \end{align}
        Also by the definition (\ref{Equ_Def l_j}) of $l_j$ and (\ref{Equ_J equiv to L^2 with r^power}), we have,
        \begin{align}
            J^{\gamma_2^+ - \sigma}_K(u_j; K^{-l}) \leq C(K, \Lambda)\epsilon\cdot K^{l (\gamma_2^+ - \sigma -1)}, \ \ \ \forall 1\leq l< l_j -1. \label{Equ_J(l) decay rate in l}
        \end{align}
        Note that by $\mu_2 \leq 0$, we have $\gamma_2^+ -\sigma -1 <0$.  Hence (\ref{Equ_J(l) monotone in 1<l<l_j}) and (\ref{Equ_J(l) decay rate in l}) together implies $l_j < +\infty$.  
        
        Now consider the renormalized sequences $\hat{\Sigma}^i_j:= \eta_{\mathbf{0},K^{-l_j}}(\Sigma_j^i)$, $i = 0, 1$. By assumption (b) and Corollary \ref{Cor_Converg in all Scales}, up to a subsequence, $|\hat{\Sigma}^0_j|\to |\mbfC|$ and $|\hat{\Sigma}^1_j|\to |\mbfC - \mathbf{x}_\infty|$ for some $\mathbf{x}_\infty \in \BB_1$.  Moreover, by the definition (\ref{Equ_Def l_j}) of $l_j$ and unique continuation of LSMH, we must have $\mathbf{x}_\infty \neq \mathbf{0}$.  
        
        On the other hand, by Allard's Regularity Theorem \cite{allard_first_1972}, $\hat{u}_j(x):= K^{l_j}u_j(K^{-l_j}x)$ $C^1$-subconverges to $\hat{u}_\infty := G_{\mbfC, g_{\eucl}}^{\mbfC - \mathbf{x}_\infty}$ outside a large ball.  Thus by (\ref{Equ_J(l) monotone in 1<l<l_j}), we have \[
        \cA\cR_\infty(\hat{u}_\infty; \mbfC) \leq \gamma_2^+ -\sigma \leq -\sigma <0\,.  \]
        This contradicts Lemma \ref{Lem_Ana on SMC_AR_infty(C+x; C) = 0}.
    \end{proof}
    
    By applying Lemma \ref{Lem_L^2 Growth Estimate} to all suitable scales with a bootstrap argument, we get,
    \begin{Cor} \label{Cor_L^2 Growth Est}
        Let $0<\sigma<\gamma_{\gap, \Lambda}$, $\sigma' := \gamma_{\gap, \Lambda}-\sigma$. Let $K = K(\sigma, \Lambda)>2$, $\delta_1=\delta_1(\sigma, \Lambda)>0$ be in Lemma \ref{Lem_L^2 Growth Estimate}.  There exists $\delta_2(\sigma, \Lambda)\in (0, \delta_1)$ with the following property.
        
        Suppose $\Sigma^i\ (i=0, 1), g, \gamma_2^+$ be in Lemma \ref{Lem_L^2 Growth Estimate} with (a) (b) holds for $\delta\leq \delta_2(\sigma, \Lambda)$ and $u:= G^{\Sigma^1}_{\Sigma^0, g}\in L^2(\Sigma^0)$. Let 
        \[
            r_u:= \|u\|_{L^2(A^g(\mathbf{0}, K^{-1}, 1); \|\Sigma^0\|_g)}^{2/(n-\sigma')}.
        \]
        Then,
        \begin{enumerate}[(i)]
            \item $u\in C^2(\Sigma^0\cap A^g(\mathbf{0}, r_u/2, 3/2))$ and \[
            \graph_{\Sigma^0, g}(u)\cap A^g(\mathbf{0}, r_u/2, 3/2) = \Sigma^1\cap A^g(\mathbf{0}, r_u/2, 3/2)\,;  \]
            \item For every $r \in [r_u, 1]$, we have 
                  \[
                      \|u\|_{C^2_*, A^g(\mathbf{0}, r, 1)} \leq C(\sigma, \Lambda)\|u\|_{L^2(A^g(\mathbf{0}, K^{-1}, 1);\|\Sigma^0\|_g)}\cdot r^{\gamma_2^+ - \sigma - 1} \leq C(\sigma, \Lambda) r_u^{\sigma'/2}\,.  
                  \]
        \end{enumerate}
    \end{Cor}
    \begin{proof}
        Assume that $u$ is not identically $0$. Then by unique continuation of LSMH, $r_u>0$.
        
        For each $\tau \in (0, 1]$, let $g_\tau:= \tau^{-2}\eta_{\mathbf{0},\tau}^*g$, $\Sigma^i_\tau := \eta_{\mathbf{0},\tau}^{-1}(\Sigma^i)$, $x^i_\tau:= \eta_{\mathbf{0}, \tau}^{-1}(x^i)$, $i=0,1$.  By the volume monotoncity formula for LSMH, we can choose $\delta_1'(\sigma, \Lambda)\in (0,\delta_1)$ small such that if (a), (b) hold for $\delta\leq \delta_1'$, then for every $\tau\in (0,1]$, we have 
        \begin{align*}
            \|g_\tau - g_{\eucl}\|_{C^3(\BB_2)}\leq \delta_1, & \  & \theta_{|\Sigma^i_\tau|}(x_\tau^i, 1) - \theta_{|\Sigma^i_\tau|}(x^i_\tau) \leq \delta_1.
        \end{align*}
        
        Let 
        \begin{align*}
            l_0 := \sup\big\{l\in \mathbb{N}^+: & \|u\|_{L^2(A^g(\mathbf{0}, \tau/K,\tau); \|\Sigma^0\|_g)}\leq \delta_1' \cdot \tau^{1+ n/2}, \\
                                                & \ \forall \tau = K^{-m}\in [K^{-l+1}, 1],\ m \text{ integer } \big\}.     
        \end{align*}
        \textbf{Claim.} By taking $\delta_2(\sigma, \Lambda)\in (0, \delta_1')$ small enough, we have $r_u\geq K^{-l_0}$ provided (a) and (b) hold for $\delta<\delta_2$.

        \noindent \textit{Proof of the Claim.} WLOG assume that $l_0<+\infty$. Choose $\delta_2(\sigma, \Lambda)\in (0,\delta_1')$ small enough such that \[
            \|u\|_{L^2(A^g(\mathbf{0}, K^{-1}, 1); \|\Sigma^0\|_g)} < \varepsilon,  \]
        where $\varepsilon = \varepsilon(\sigma, \Lambda)<1$ small enough, which will be determined later.
        
        By definition of $l_0$ and Lemma \ref{Lem_L^2 Growth Estimate}, inductively on $l$ we have, $\forall 1\leq l\leq l_0$, \[
            J_K^{\gamma_2^+ -\sigma}(u, K^{-l})\leq J_K^{\gamma_2^+ - \sigma}(u, 1)\leq C(\sigma, \Lambda)\|u\|_{L^2(A^g(\mathbf{0}, K^{-1}, 1);\|\Sigma^0\|_g)}.   \]
        Hence combined with (\ref{Equ_J equiv to L^2 with r^power}) we have for every $1\leq l\leq l_0$, 
        \begin{align}
            \|u\|_{L^2(A^g(\mathbf{0}, K^{-l-1}, K^{-l}))}\leq C(\sigma, \Lambda)\|u\|_{L^2(A^g(\mathbf{0}, K^{-1}, 1); \|\Sigma^0\|_g)}\cdot K^{-l(\gamma_2^+ - \sigma + n/2)}.  \label{Equ_u L^2 growth r^gamma_2^+ - sigma + n/2}
        \end{align}
        Therefore if $K^{-l_0}> r_u$, then we must have, 
        \begin{align}
            \begin{split}
                \|u\|_{L^2(A^g(\mathbf{0}, K^{-l_0-1}, K^{-l_0}))}
                & \leq C(\sigma, \Lambda)r_u^{\sigma'/2} \cdot r_u^{n/2 - \sigma'}\cdot K^{-l_0(\gamma_2^+ -\sigma + n/2)} \\
                & \leq C(\sigma, \Lambda)r_u^{\sigma'/2} \cdot K^{-l_0(n + \gamma_2^+ -\sigma -\sigma')} \\
                & \leq \bar{C}(\sigma, \Lambda)r_u^{\sigma'/2} \cdot K^{-l_0(1+ n/2)},
            \end{split}  \label{Equ_u L^2 growth K^-l(1+n/2)}
        \end{align}
        where the first inequality follows from (\ref{Equ_u L^2 growth r^gamma_2^+ - sigma + n/2}) and the definition of $r_u$, the second inequality follows by $K^{-l_0}>r_u$ and $\sigma'<\gamma_{\gap, \Lambda}< (n-2)/2$, the third inequality follows by the fact that $\sigma+\sigma' = \gamma_{\gap, \Lambda}$ and that 
        \begin{align}
            \frac{n-2}{2} + \gamma_2^+ - \gamma_{\gap, \Lambda} \geq -\gamma_1^+ + \gamma_2^+ -\gamma_{\gap, \Lambda}\geq 0.  \label{Equ_(n-2)/2 + gamma_2^+ - gamma_gap > 0}            
        \end{align}
        
        Choose $\varepsilon(\sigma, \Lambda)$ small enough such that $\bar{C}(\sigma, \Lambda)\varepsilon^{\sigma'/(n-\sigma')}\leq \delta_1'/2$. (\ref{Equ_u L^2 growth K^-l(1+n/2)}) contradicts the choice of $l_0$.  This concludes the proof of the Claim.$\hfill\blacksquare$
        
        Now by the Claim and Lemma \ref{Lem_L^2 Growth Estimate}, (i) holds; (ii) also holds by combining the estimate (\ref{Equ_u L^2 growth r^gamma_2^+ - sigma + n/2}) and (\ref{Equ_(n-2)/2 + gamma_2^+ - gamma_gap > 0}).
    \end{proof}
    
    For later applications, we also need to compare the graphical functions of different hypersurfaces.  This is also derived from Lemma \ref{Lem_L^2 Growth Estimate}.
    \begin{Cor} \label{Cor_L^2 Growth est for pair}
        Let $\Sigma^i = \partial U^i$ ($i=0,1, 2$) be LSMH in $(\BB_2^8, g)$ with area $\leq \Lambda$ and index $\leq \Lambda$. Assume
        \begin{enumerate} [(a)]
        \item $\|g-g_{\eucl}\|_{C^4}\leq \delta$,  $\|u^{(i)}\|_{L^2(A^g(\mathbf{0},1/2, 1); \|\Sigma^0\|_g)} \leq \delta$, where $u^{(i)}:= G_{\Sigma^0, g}^{\Sigma^i} \in L^\infty(\Sigma^0)$, $i = 1,2$;
        \item For $i = 0, 1, 2$, there exists $x^i\in \Sing(\Sigma^i)\cap \BB_{1/2}$ such that $x^0 = \mathbf{0}$ and
          \begin{align*}
           & \theta_g(x^i, 1; \|\Sigma^i\|_g) - \theta(x^i; \|\Sigma^i\|_g) \leq \delta;                   \\
           & \theta(x^0; \|\Sigma^0\|_g) = \theta(x^1; \|\Sigma^1\|_g) = \theta(x^2; \|\Sigma^2\|_g). 
          \end{align*}
        \end{enumerate}
        Let $0<\sigma<\gamma_{\gap, \Lambda}/2$, $\sigma' := \gamma_{\gap, \Lambda} - \sigma$.  Then for $\delta\leq \delta_3(\sigma, \Lambda)$ and for $r\in (40 \sup\{r_{u^{(1)}}, r_{u^{(2)}}\}, 1/4)$, we have \[
        \|u^{(1)} - u^{(2)}\|_{L^2(A^g(\mathbf{0},r/2, r); \|\Sigma^0\|_g)} \leq C(\sigma, \Lambda)\|G^{\Sigma^2}_{\Sigma^1, g}\|_{L^2(\|\Sigma^1\|_g)}\cdot r^{n/2 + \gamma_2^+(\mbfC_{x^1}\Sigma^1) - \sigma}.   \]
    \end{Cor}
    \begin{proof}
        Denote for simplicity, $v:= G_{\Sigma^1, g}^{\Sigma^2} \in L^\infty(\Sigma^1)$;  We have,
        \begin{align*}
            \|u^{(2)} - u^{(1)}\|_{L^2(A^g(\mathbf{0}, r/2, r))} 
             & \leq C(\Lambda)r^{n/2}\cdot \sup_{A^g(\mathbf{0}, r/2, r)\cap \Sigma^0}|u^{(2)}- u^{(1)}|                                     \\
             & \leq C(\Lambda) r^{n/2}\cdot \sup_{A^g(\mathbf{0}, r/4, 2r)\cap \Sigma^1}|v|                                                  \\
             & \leq C(\Lambda) r^{n/2 + 1}\cdot \|v\|_{C^2_*, A^g(x^1, r/10, 1)\cap \Sigma^1}                                                \\
             & \leq C(\Lambda, \sigma) r^{n/2 + 1}\cdot r^{\gamma_2^+(\mbfC_{x^1}\Sigma^1) - \sigma -1} \|v\|_{L^2(\|\Sigma^1\|_g)}.            
        \end{align*}
        Here, the first inequality follows by volume upper bound for minimal hypersurface $\Sigma^0$ in $A^g(\mathbf{0}, r/2, r)$; The second inequality, after taking $\delta$ sufficiently small, is given by the following geometric Lemma \ref{Lem_Comparing G_Sigma^1, g^Sigma^2 and u^2 - u^1} comparing $G_{\Sigma^1, g}^{\Sigma^2}$ with $|u^{2} - u^{1}|$; The third inequality follows directly from definition of $C^2_*$ norm and that $A^g(\mathbf{0},r/4, 2r)\subset A^g(x^1; r/10, 1)$; The last inequality is given by Corollary \ref{Cor_L^2 Growth Est} (ii). 
    \end{proof}
    
    \begin{Lem} \label{Lem_Comparing G_Sigma^1, g^Sigma^2 and u^2 - u^1}
        There is a universal constant $\varepsilon_0$ with the following property. Let $\Sigma^i = \partial U^i$ ($i= 0, 1, 2$) be connected $C^2$ hypersurfaces in $\BB_1$. Suppose $g$ is a $C^2$ metric on $\BB_1$ such that
            \begin{enumerate} [(a)]
                \item $|A_{\Sigma^0}|_g \leq \varepsilon_0$ and $\mathbf{0}\in \Sigma^0$; $\|g- g_{\eucl}\|_{C^2}\leq \varepsilon_0$;
                \item $\Sigma^i = \graph_{\Sigma^0}(u^{(i)})$ with $\|u^{(i)}\|_{C^2}\leq \varepsilon_0$, $i=1,2$.
            \end{enumerate}
            Then we have \[
                |u^{(2)}(\mathbf{0}) - u^{(1)}(\mathbf{0}) | \leq 2\sup_{\BB_{1/4}}|G_{\Sigma^1, g}^{\Sigma^2}| \leq 4\sup_{\BB_{1/2}}|u^{(2)} - u^{(1)}|.   \]
    \end{Lem}
    \begin{proof}
        This follows directly from the triangle inequality.
    \end{proof}

    \subsection{Induced Jacobi Field for Pairs}
        Let $\Sigma$ be a connected two-sided LSMH in a closed manifold $(M^8, g)$, with $\nu=\nu_{\Sigma, g}$ to be a global unit normal field. We also fix an integer $k\geq 4$ and $\alpha\in (0, 1)$.
        
        Since near each singularity, $\Sigma$ is a $C^{1,\alpha}$ diffeomorphic to a minimal hypercone restricted in unit ball, we know that there exists an open subset $\cN\subset M$ being a sufficiently small neighborhood of $\Clos(\Sigma)$ such that $\cN\setminus \Clos(\Sigma)$ has two connected component $\cN_+ \sqcup \cN_-$ and $\nu$ points into $\cN_+$.  
        Moreover, for every $(g', \Sigma')\in \cM^{k, \alpha}(M)$ close enough to $(g, \Sigma)$, we know that $\Clos(\Sigma') \subset \cN$ and that there is a unique connected component $\cN'_-$ of $\cN\setminus \Clos{\Sigma'}$ such that 
        \begin{align*}
            \cN'_- \Delta \cN_- \subset\subset \cN.
        \end{align*}
        
        In this subsection, to define a graphical function as in (\ref{Equ_Growth est_Graph Function in general}), the ambient manifold $N$ is chosen to be $\cN$. Let $\delta_4 = \delta_4(g, \Sigma, \Lambda)>0$ such that for every $(g', \Sigma')\in \cL^{k, \alpha}(g, \Sigma; \Lambda, \delta_4)$, $\Clos(\Sigma')\subset \cN$ and bounds a unique $\cN_-$ as above.  When defining graphical functions between any two elements $(g^1, \Sigma^1), (g^2, \Sigma^2)\in \cL^{k, \alpha}(g, \Sigma; \Lambda, \delta_4)$ as in (\ref{Equ_Growth est_Graph Function in general}), we always choose $U^i:= \cN^i_-$.
        
        We define a semi-metric on $\cL^{k,\alpha}(g, \Sigma; \Lambda, \delta_4)$ by \[
            \mbfD^\cL[(g_1,\Sigma_1), (g_2, \Sigma_2)]:= \|g_1-g_2\|_{L^\infty(M)} + \|G^{\Sigma^2}_{\Sigma_1, g_1}\|_{L^2(\|\Sigma_1\|_{g_1})} + \|G^{\Sigma^1}_{\Sigma_2, g_2}\|_{L^2(\|\Sigma_2\|_{g_2})}.   \]
        Here, norms of tensors are taken with respect to metric $g$ which we have fixed at the beginning. By \cite{schoen_regularity_1981}, $(g_j, \Sigma_j)$ converges to $(g_\infty, \Sigma_\infty)$ in $\cL^{k,\alpha}$ if and only if $\mbfD^\cL[(g_j, \Sigma_j), (g_\infty, \Sigma_\infty)]\to 0$ and $\|g_j - g_\infty\|_{C^{k-1,\alpha}(M)}\to 0$ when $j\to \infty$.
        
        The following Lemma is an application of Corollary \ref{Cor_L^2 Growth est for pair}.
        \begin{Lem} \label{Lem_Tangent vectors to scL^k}
            Let $r\in (0, \injrad(g, \Sigma))$.  Suppose for $1\leq j< \infty$,
            \begin{enumerate}[(a)]
                \item $(g_j^0, \Sigma_j^0) \neq (g_j^1, \Sigma_j^1)\in \cL^{k,\alpha}(g, \Sigma; \Lambda, \delta_1)$ are distinct pairs, and $(\bar{g}_j, \bar{\Sigma}_j) \in \cL^{k,\alpha}(g, \Sigma; \Lambda, \delta_1)$ such that when $j\to \infty$, $(g^i_j, \Sigma^i_j) \to (g, \Sigma)$ in $\cL^{k,\alpha}$, $i = 0, 1$ and $(\bar{g}_j, \bar{\Sigma}_j)\to (g, \Sigma)$ also in $\cL^{k,\alpha}$;
                \item For $i=0, 1$, $g_j^i = (1+t_jf_j^i)\bar{g}_j$, where $t_j\searrow 0$, $f_j^i\in C^{k,\alpha}(M)$ with uniformly bounded $C^{k,\alpha}$ norms, such that $\spt(f_j^i)\cap B^g(\Sing(\Sigma), r) = \emptyset$, and that either $f_j^0\equiv f_j^1$ (in which case we set $f_\infty \equiv 0$), or $f_j^1 - f_j^0\to f_\infty$ in $C^{k,\alpha}(M)$ satisfying that $\nu_{\Sigma, g}(f_\infty)$ is not identically $0$ along $\Sigma$.
            \end{enumerate}
            Let $d_j:= \mbfD^\cL [(g^0_j, \Sigma_j^0), (g^1_j, \Sigma_j^1)] >0$ and $u^{(i)}_j:= G^{\Sigma_j^i}_{\bar{\Sigma}_j, \bar{g}_j}\in  L^2(\bar{\Sigma}_j)$.  Then we have, after passing to a subsequence of $j\to \infty$,
            \begin{align*}
                (u^{(1)}_j - u^{(0)}_j)/ d_j \to \hat{u}_\infty  \ \text{ in }C^2_{loc}(\Sigma); & \  & t_j(f_j^1 - f_j^0)/ d_j \to \hat{f}_\infty \ \text{ in }C^2(M),
            \end{align*}
            where $\hat{f}_\infty \in \RR_{\geq 0}\langle f_\infty\rangle$, $\hat{u}_\infty$ is a nonzero solution of \[
                L_{\Sigma, g} \hat{u}_\infty = \frac{n}{2}\nu_{\Sigma, g}(\hat{f}_\infty);   \]
            Moreover, $\hat{u}_\infty$ satisfies $\cA\cR_p(\hat{u}_\infty) \geq \gamma_2^+(\mbfC_p\Sigma)$ for every $p\in \Sing(\Sigma)$.  In other words, $\hat{u}_\infty$ is a function of slower growth.

        \end{Lem}
        \begin{proof}
            Fix $\sigma\in (0, \gamma_{\gap, \Lambda}/2)$.  By Corollary \ref{Cor_L^2 Growth Est} and \ref{Cor_L^2 Growth est for pair}, there exists $r_0(\sigma,g,\Sigma, \Lambda)\in (0, r/4)$ and $r_j\searrow 0$ such that 
            \begin{enumerate} [(i)]
                \item For $i= 0, 1$, $\Sigma_j^i \setminus B^{\bar{g}_j}(\Sing(\bar{\Sigma}_j), r_j) = \graph_{\bar{\Sigma}_j, \bar{g}_j}(u_j^{(i)}) \setminus B^{\bar{g}_j}(\Sing(\bar{\Sigma}_j), r_j)$, and when $j\to \infty$, \[
                          \|u_j^{(i)}\|_{C^2_*, \bar{\Sigma}_j\setminus B^{\bar{g}_j}(\Sing(\bar{\Sigma}_j), {r_j})} \to 0.   \]
                \item For every $p\in \Sing(\bar{\Sigma}_j)$ and every $\tau\in [2r_j, 2r_0]$, let $p^0\in \Sing(\Sigma_j^0)\cap B^{g_j^0}(p, r_j)$. Then $w_j:= u_j^{(1)} - u_j^{(0)}$ satisfies
                      \begin{align}
                          \|w_j\|_{L^2(A^{\bar{g}_j}(p, \tau/2, \tau); \|\bar{\Sigma}_j\|_{\bar{g}_j})} \leq C(g, \Sigma, \Lambda, \sigma, r_0)d_j\cdot \tau^{n/2 + \gamma_2^+(\mbfC_{p^0}\Sigma_j^0) - \sigma}.   \label{Equ_Graph Diff Bd |w_j|_A_tau < d_j tau^n/2+gamma_2^+}
                      \end{align}
            \end{enumerate}
            
            Also by Theorem \ref{Thm_Append_MSE for graphs}, for $j>>1$, $w_j$ satisfies the following equation on $\bar{\Sigma}_j \setminus B^{\bar{g}_j}(\Sing(\bar{\Sigma}_j), {r_j})$
            \begin{align}
                -L_{\bar{\Sigma}_j, \bar{g}_j} w_j + \frac{n}{2}t_j\nu_{\bar{\Sigma}_j, \bar{g}_j}(f_j^1 - f_j^0) = \mathrm{div}_{\bar{\Sigma}_j, \bar{g}_j}(\cE^1_j) + \frac{1}{ \rho_j}\cE^2_j ,   \label{Equ_Apprx minimal surface equa of graph diff w_j}.  
            \end{align}
            where $\rho_j$ be the distant function under $\bar{g}_j$ to $\Sing(\bar{\Sigma})$, $\cE^1_j$ and $\cE^2_j$ are error terms with pointwise estimate,
            \begin{align}
                \begin{split}
                    |\cE^1_j|, |\cE^2_j| \leq C(\Sigma, g)
                    & \cdot\big[ \sum_{i = 0}^1 (\frac{|u_j^{(i)}|}{\rho_j} + |\nabla u_j^{(i)}| + t_j\sum_{l=0}^2|\nabla^l_M f_j^i|) \big]  \\
                    & \cdot \big(\frac{|w_j|}{\rho_j} + |\nabla w_j| + t_j\sum_{l=0}^2|\nabla^l_M (f_j^1 - f_j^0)| \big).
                \end{split} \label{Equ_Est on Error of Apprx MS Equ of w_j}
            \end{align}
            
            Note that if $f_j^0\neq f_j^1$, then by definition, for $j>>1$,
            \begin{align}
                d_j \geq \|g_j^1 - g_j^0\|_{L^\infty} \geq C(g)t_j\cdot\|f_\infty\|_{L^\infty}.   \label{Equ_d_j > t_j|f_infty|}
            \end{align}
            Thus by assumption (b), after passing to a subsequence, we have, \[
                \frac{1}{d_j}\cdot t_j(f_j^1 - f_j^0) \to \hat{f}_\infty\in \RR_{\geq 0}\langle f_\infty \rangle \ \ \text{ in }C^k(M).    \]
            While if $f_j^0\equiv f_j^1$, then (\ref{Equ_d_j > t_j|f_infty|}) and the convergence above hold automatically.
            
            Let $d_j' := \|w_j\|_{L^2(\bar{\Sigma}_j\setminus B^{\bar{g}_j}(\Sing(\bar{\Sigma}_j), r_0); \|\bar{\Sigma}_j\|_{\bar{g}_j})}$.
            \begin{claim}\label{claim:dj_d'j}
                $\liminf_{j\to \infty} d_j / d_j' \in (0, +\infty)$.
            \end{claim} 
            
            Once the claim is established, recall by (\ref{Equ_Graph Diff Bd |w_j|_A_tau < d_j tau^n/2+gamma_2^+}), (\ref{Equ_d_j > t_j|f_infty|}) and elliptic estimate for (\ref{Equ_Apprx minimal surface equa of graph diff w_j}), up to a subsequence, we have for $j>>1$, \[
                \|w_j\|_{C^2, \bar{\Sigma}_j\setminus B^{\bar{g}_j}(\Sing(\bar{\Sigma}_j), 2r_0)} \leq C(g, \Sigma, r_0) (d_j' + d_j + t_j\|f_j^1 - f_j^0\|_{C^2(M)}).   \]
            Thus combined with (\ref{Equ_Graph Diff Bd |w_j|_A_tau < d_j tau^n/2+gamma_2^+}), (\ref{Equ_d_j > t_j|f_infty|}) and the Claim, we have, after passing to a subsequence of $j\to\infty$, $w_j/d_j \to \hat{u}_\infty$ in $C^2_{loc}(\Sigma)$.
            By (\ref{Equ_Apprx minimal surface equa of graph diff w_j}), (\ref{Equ_Est on Error of Apprx MS Equ of w_j}) and the Claim, we know that $\hat{u}_\infty \neq 0$ and satisfies \[
                -L_{\Sigma, g} \hat{u}_\infty + \frac{n}{2}\nu_{\Sigma, g}(\hat{f}_\infty) = 0.   \]
            Also by (\ref{Equ_Graph Diff Bd |w_j|_A_tau < d_j tau^n/2+gamma_2^+}) and Corollary \ref{Cor_Converg in all Scales}, for every $p\in \Sing(\Sigma)$ and every $\tau\in (0, r_0)$, \[
                \|\hat{u}_\infty\|_{L^2(A^g(p; \tau/2, \tau); \|\Sigma\|_g)} \leq C\tau^{n/2 + \gamma_2^+(\mbfC_p\Sigma) -\sigma}.   \]
            By definition of $\cA\cR_p$, this means $\cA\cR_p(\hat{u}_\infty)\geq \gamma_2^+(\mbfC_p\Sigma)-\sigma$.  Therefore, by Lemma \ref{Lem_Pre_Asymp Rate takes value in Gamma(C) and int by part} we have, \[
                \cA\cR_p(\hat{u}_\infty)\geq \gamma_2^+(\mbfC_p\Sigma).  
            \]

            \noindent \textit{Proof of Claim \ref{claim:dj_d'j}.} Suppose otherwise, $c:=\liminf_{j\to \infty }d_j /d_j'$ is either $0$ or $+\infty$.
            
            \noindent \underline{\textit{Case I.}} If $c = 0$, then using (\ref{Equ_Graph Diff Bd |w_j|_A_tau < d_j tau^n/2+gamma_2^+}), (\ref{Equ_Apprx minimal surface equa of graph diff w_j}), (\ref{Equ_Est on Error of Apprx MS Equ of w_j}) and (\ref{Equ_d_j > t_j|f_infty|}), we can apply the same argument as above to obtain, after passing to a subsequence of $j\to \infty$,
            \begin{align*}
                \frac{w_j}{d_j'} \to \tilde{w}_\infty\ \ \text{ in }C^2_{loc}(\Sigma), \ \ \ \frac{t_j(f_j^1 - f_j^0)}{d_j'} \to 0\ \ \text{ in } C^k(M)\,.
            \end{align*}
            And $-L_{\Sigma, g}\tilde{w}_\infty = 0$, $\|\tilde{w}_\infty\|_{L^2(\Sigma\setminus B^g(\Sing(\Sigma)), r_0)} = 1$. But by (\ref{Equ_Graph Diff Bd |w_j|_A_tau < d_j tau^n/2+gamma_2^+}), $\tilde{w}_\infty \equiv 0$ on $\Sigma\cap B^g(\Sing(\Sigma), 2r_0)$. This violates the unique continuation property for solutions to elliptic equations. \\
            
            \noindent \underline{\textit{Case II.}} If $c = +\infty$, then again using (\ref{Equ_Graph Diff Bd |w_j|_A_tau < d_j tau^n/2+gamma_2^+}), (\ref{Equ_Apprx minimal surface equa of graph diff w_j}), (\ref{Equ_Est on Error of Apprx MS Equ of w_j}) and (\ref{Equ_d_j > t_j|f_infty|}),  after passing to a subsequence of $j\to \infty$, we have 
            \[
                \frac{w_j}{d_j} \to \tilde{w}_\infty \ \ \text{ in }C^2_{loc}(\Sigma), \ \ \ \frac{t_j(f_j^1 - f_j^0)}{d_j} \to \hat{f}_\infty\,.
            \]
            And $-L_{\Sigma, g}\tilde{w}_\infty + \frac{n}{2}\nu_{\Sigma, g}(\hat{f}_\infty) = 0$. But $\tilde{w}_\infty\equiv 0$ on $\Sigma\setminus B^g(\Sing(\Sigma), r_0)$, and by the choice of $r_0$, $\spt(\hat{f}_\infty)\cap B^g(\Sing(\Sigma), 2r_0) = \emptyset$. Hence by unique continuation property,  $\tilde{w}_\infty \equiv 0$, and then since $\nu_{\Sigma, g}(f_\infty)$ is not identically $0$, we must have $\hat{f}_\infty \equiv 0$, in other words,
            \begin{align}
                \|g_j^1 - g_j^0\|_{L^\infty} / d_j \to 0.  \label{Equ_|g_j^1 -g_j^0| /d_j to 0}
            \end{align}
            We now show that this contradicts to the definition of $d_j$.
            
            Denote for simplicity 
            \begin{align*}
                 & \ u_j^{(0, 1)} := G_{\Sigma_j^0, g_j^0}^{\Sigma_j^1} \in L^\infty(\Sigma_j^0),\ \ \ u_j^{(1,0)}:= G_{\Sigma_j^1, g_j^1}^{\Sigma_j^0} \in L^\infty(\Sigma_j^1),                     \\
                 & \ \tilde{d}_j := \|u_j^{(0,1)}\|_{L^2(\Sigma_j^0\setminus B^{g_j^0}(\Sing(\Sigma_j^0)), {r_0/K})} + \|u_j^{(1,0)}\|_{L^2(\Sigma_j^1\setminus B^{g_j^1}(\Sing(\Sigma_j^1)), {r_0/K})},
            \end{align*}
            where $K = K(\sigma, \Lambda)>2$ is given by Lemma \ref{Lem_Ana on SMC_Growth Rate Monoton}.
            First of all, by Lemma \ref{Lem_Comparing G_Sigma^1, g^Sigma^2 and u^2 - u^1} and the fact that $w_j/d_j$ subconverges to $0$ in $C^2_{loc}(\Sigma)$, we know that $u_j^{(0,1)}/d_j$ and $u_j^{(1,0)}/d_j$ both subconverges to $0$ in $C^2_{loc}(\Sigma)$. In particular, 
            \begin{align}
                \liminf_{j\to \infty} \tilde{d}_j /d_j = 0.   \label{Equ_tilde(d)_j/d_j to 0}
            \end{align}
            Also when $j>>1$, by Corollary \ref{Cor_L^2 Growth Est} (ii), on $\Sigma_j^0\setminus B^{g_j^0}(\Sing(\Sigma_j^0), {s_j})$, $u_j^{(0,1)}$ is $C^2$ and pointwisely, 
            \[
                |u_j^{(0,1)}| \leq C(\Sigma, g, \Lambda, \sigma, r_0)\tilde{d}_j\cdot \rho_{j, 0}^{-n/2 + 1},
            \]
            where $s_j \leq C(\Sigma, g, r_0)\tilde{d}_j^{\ 2/(n-\sigma')}$, and $\rho_{j, 0}$ denotes the distance function to $\Sing(\Sigma_j^0)$ with respect to $g_j^0$. Here we are using the fact that $\gamma_2^+ -\sigma \geq \gamma_1^+\geq -n/2 + 1$. 
            In particular, 
            \begin{align}
                \begin{split}
                    \|u_j^{(0,1)}\|_{L^2(\Sigma_j^0)} 
                    & \leq \tilde{d}_j + \|u_j^{(0, 1)}\|_{L^2(A^{g_j^0}(\Sing(\Sigma_j^0)), s_j, r_0)} + \|u_j^{(0, 1)}\|_{L^2(B^{g_j^0}(\Sing(\Sigma_j^0)), s_j)} \\
                    &\leq \tilde{d}_j + C(\Sigma, g, \Lambda, \sigma, r_0)\tilde{d}_j\cdot\|\rho_{j, 0}^{-n/2 + 1}\|_{L^2(A^{g_j^0}(\Sing(\Sigma_j^0)), {s_j, r_0})} \\
                    &\ \ + C(g)\|\Sigma_j^0\|(B^{g_j^0}(\Sing(\Sigma_j^0)), {s_j})^{1/2} \\
                    & \leq \tilde{d}_j\Big( 1 + C(\Sigma, g, \Lambda, \sigma, r_0) + C(g, \Sigma, \Lambda) \tilde{d}_j^{\ \sigma'/(n-\sigma')} \Big)  = O(\tilde{d}_j).
                \end{split} \label{Equ_|u_j^(0,1)|< O(tilde(d)_j)}
            \end{align}
            Similarly, 
            \begin{align}
                \|u_j^{(1,0)}\|_{L^2(\Sigma_j^1)} \leq O(\tilde{d}_j).  \label{Equ_|u_j^(1,0)|< O(tilde(d)_j)}
            \end{align}
            Together with (\ref{Equ_tilde(d)_j/d_j to 0}) and (\ref{Equ_|g_j^1 -g_j^0| /d_j to 0}), (\ref{Equ_|u_j^(0,1)|< O(tilde(d)_j)}) and (\ref{Equ_|u_j^(1,0)|< O(tilde(d)_j)}) contract the definition of $d_j$. This finishes the proof of Claim. $\hfill\blacksquare$

        \end{proof}

    \section{Proof of Theorem \ref{Thm_Generic Semi-nondeg_Loc}}\label{Sec:Gen_Semi-nondeg}
        The goal of this section is to prove Theorem \ref{Thm_Generic Semi-nondeg_Loc}.  
        To deal with one-sided and two-sided cases simultaneously, we first slightly extend the notions of canonical neighborhoods $\cL^{k,\alpha}$ and Jacobi fields of slower growth to manifolds with symmetry.
        
        \begin{Def} \label{Def_G-inv}
            Let $G\subset \text{Diff}(M)$ be a discrete subgroup of the diffeomorphism group of $M$. We call a two-sided LSMH $\Sigma \subset (M, g)$ \textbf{$G$-invariant} if for every $\gamma\in G$, we have \[
                \gamma^* g = g, \ \ \ \gamma(\Sigma) = \Sigma.   \]
            
            Given a $G$-invariant two-sided LSMH $\Sigma \subset (M, g)$ with a unit normal field $\nu$, there is an induced $G$-representation on $C^0_{loc}(\Sigma)$ given by 
            \begin{align*}
                (\gamma\cdot f)(x) & := \langle \frac{d}{dt}\Big|_{t=0} \gamma(\exp^g_y(t\cdot f(y)\nu(y))), \nu(\gamma(y)) \rangle_g\Big|_{y=\gamma^{-1}(x)} \\
                                   & = f(\gamma^{-1}(x))\cdot \langle \gamma_* \nu(\gamma^{-1}(x)), \nu(x)\rangle_g,
            \end{align*}
            for every $\gamma\in G$, $f\in C^0_{loc}(\Sigma)$, $x\in \Sigma$.  Under this representation, for every subset $\cK\subset C^0_{loc}(\Sigma)$, let $\cK^G$ be the set of all $G$-invariant elements in $\cK$, i.e. \[
                \cK^G := \{f\in \cK: \gamma\cdot f = f, \forall \gamma\in G\}.   \]
            
            We call a $G$-invariant two-sided LSMH $\Sigma \subset (M, g)$ \textbf{$G$-semi-nondegenerate} if $(\Ker^+ L_{\Sigma, g})^G = \{\mathbf{0}\}$.
            We also use the notation 
            \[
                \cL^{k,\alpha}(g, \Sigma; \Lambda, \delta)^G := \{(g', \Sigma')\in \cL^{k,\alpha}(g, \Sigma; \Lambda, \delta): \gamma^*g'= g', \gamma(\Sigma')=\Sigma', \forall \gamma\in G\}.   
            \]
        \end{Def}
        \begin{Rem} \label{Rem_one-sided, G-inv, induced rep}
            Throughout this section, unless otherwise stated, $G < \mathrm{Diff}(M)$ is fixed and $\Sigma \subset (M, g)$ is a connected $G$-invariant two-sided LSMH.
            \begin{enumerate}[(i)]
                \item Under the induced $G$-representation on $C^0_{loc}(\Sigma)$, we have, for  every $\gamma\in G$ and every $f\in C^0_{loc}(\Sigma)$, \[
                          \gamma(\graph_\Sigma (f)) = \graph_\Sigma(\gamma\cdot f).   \]
                \item If $\Sigma_0 \subset (M_0, g_0)$ be a connected one-sided LSMH, let $\pi: \hat{M}_0 \to M_0$ be the double cover given by Section \ref{Sec_Prelim} such that $\hat{\Sigma}_0 := \pi^{-1}(\Sigma_0) \subset (\hat{M_0}, \hat{g}_0 := \pi^* g_0)$ is a connected two-sided LSMH. This makes $\hat{\Sigma}_0\subset (\hat{M}_0, \hat{g}_0)$ a $\ZZ_2$-invariant LSMH, where $\ZZ_2$ acts on $\hat{M}_0$ by deck transformation.  It's not hard to verify that 
                \[
                      \cL^{k, \alpha}(\hat{g}_0, \hat{\Sigma}_0; \Lambda, \delta)^{\ZZ_2} \subset \{(\pi^* g', \pi^{-1}(\Sigma')): (g', \Sigma') \in \cL^{k, \alpha}(g_0, \Sigma_0; \Lambda, \delta)\} \subset \cL^{k,\alpha}(\hat{g}_0, \hat{\Sigma}_0; 2\Lambda, 2\delta)^{\ZZ_2}
                \]
            \end{enumerate}
        \end{Rem}
        
        Note that functions of slower growth are preserved under convergence in $(\cL^{k,\alpha})^G$.
        \begin{Lem}[Compactness of Jacobi fields of slower growth] \label{Lem_Cptness for Slower Growth Jac Fields}
            Let $\Sigma \subset (M, g)$ be a connected $G$-invariant two-sided LSMH, and $\delta_0= \delta_0(g, \Sigma, \Lambda)>0$ be as in Lemma \ref{Lem_Cptness for scL^k}.   
            Suppose that there exists a sequence $\set{(g_j, \Sigma_j)}^\infty_{j = 1} \subset \cL^{k,\alpha}(g, \Sigma; \Lambda, \delta_0)^G$ satisfying
            \begin{enumerate} [(a)]
                \item $(g_j, \Sigma_j) \to (g_\infty, \Sigma_\infty) \in \cL^{k,\alpha}(g, \Sigma; \Lambda, \delta_0)^G$ in $\cL^{k,\alpha}$;
                \item There exists $0\neq u_j \in (\Ker^+ L_{\Sigma_j, g_j})^G$ a $G$-invariant Jacobi field of slower growth on $\Sigma_j$ with $\|u_j\|_{L^2(\Sigma_j)} = 1$.
            \end{enumerate}
            Then $u_j$ subconverges to some $G$-invariant Jacobi field of slower growth $u_\infty$ in $C^2_{loc}(\Sigma_\infty)$, with $\|u_\infty\|_{L^2(\Sigma_\infty)} = 1$.  
        \end{Lem}
        \begin{proof}
            Fix $\sigma\in (0, \gamma_{\gap, \Lambda}/2)$, where $\gamma_{\gap, \Lambda}$ is defined in Lemma \ref{Lem_gamma_gap >0}. Let $K = K(\sigma, \Lambda)>2$ be given by Lemma \ref{Lem_Ana on SMC_Growth Rate Monoton}.
            
            By standard elliptic estimates, for every $r>0$, there exist an integer $j(r) > 0$ and a constant $C(\Sigma, g, r) > 0$ such that when $j\geq j(r)$, we have 
            \[
                \|u_j\|_{C^3(\Sigma_j\setminus B^g(\Sing(\Sigma), r))} \leq C(\Sigma, g, r)\,.
            \]
            Hence $u_j$ subconverges to some Jacobi field $u_\infty$ in $C^2_{loc}(\Sigma)$. Since all $u_j$'s are $G$-invariant, so is $u_\infty$. 
            
            The major effort here is to show $\|u_\infty\|_{L^2(\|\Sigma\|_g)} = 1$ and that $u_\infty$ is of slower growth.
            
            First notice that for every $p\in \Sing(\Sigma)$, let $p_j\in \Sing(\Sigma_j)$ such that $p_j\to p$ as $j\to \infty$, then by the definition of $\cL^{k,\alpha}$ convergence and the monotonicity formula of volume ratio, we have 
            \begin{equation}
                \limsup_{r\to 0} \limsup_{j\to \infty} \theta_{|\Sigma_j|}(p_j, r) - \theta_{|\Sigma_j|}(p_j) = 0\,. \label{Eqn:ratio_limit}
            \end{equation}
            Hence by Lemma \ref{Lem_Ana on SMC_Quanti Uniqueness of Tang Cone}, we can choose $r_0 = r_0(\Sigma_\infty, g_\infty)\in (0, \injrad(g, \Sigma))$ such that for $j>>1$ and every $p_j\in \Sing(\Sigma_j)$, we can parametrize $\Sigma_j\cap \BB_{r_0}$ by $\mbfC_{p_j}\Sigma_j$ using \[
                F_j(x):= x + \phi_j(x)\nu_j(x),   \]
            where $\nu_j$ be the unit normal field of $\mbfC_{p_j}\Sigma_j$, i.e. 
            \[
                \Sigma_j \cap \BB_{r_0} = \graph_{\mbfC_{p_j}\Sigma_j}(\phi_j)\,.
            \]
            Furthermore, $\phi_j \in C^2_{loc}(\mbfC_{p_j}\Sigma_j\cap \BB_{r_0})$ is the graphical function with estimates
            \begin{align}
                \|\phi_j\|_{C^2_*, \BB_s} \leq \Psi(s|g_\infty, \Sigma_\infty), \ \ \ \forall s\in (0, r_0),\ \forall j\geq j_s= j(s, g_\infty, \Sigma_\infty)\,.  \label{Equ_graphical func phi_j est in B_s}
            \end{align}
            Here $\Psi(s|g_\infty, \Sigma_\infty)$ is some positive constant depending only on $s, g_\infty, \Sigma_\infty$, increasing in $s$ and tending to $0$ when $s\searrow 0$. We emphasis that the right hand side does not depend on $j$ provided $j$ is sufficiently large. In addition, for every function $u_j$ on $\Sigma_j$, we let $\tilde{u}_j := u_j\circ F_j$. 
            
            By (\ref{Equ_graphical func phi_j est in B_s}) and Remark \ref{Rem_Pre_Jac field equ on MH near cone}, there exists $s_0 = s_0(g_\infty, \Sigma_\infty)\in (0, r_0)$ such that when $j\geq j_{s_0}$, Corollary \ref{Cor_Ana on SMC_Growth Rate Monoton with Perturb} applies to $\tilde{u}_j$ on each $\AAa(\mathbf{0}, K^{-3}s, s)\cap \mbfC_{p_j}\Sigma_j$ and we obtain 
            \begin{align}
                J^{\gamma_j}_{j, K}(\tilde{u}_j; K^{-2} s) -2J^{\gamma_j}_{j, K}(\tilde{u}_j; K^{-1} s) + J^{\gamma_j}_{j, K}(\tilde{u}_j; s) > 0  \label{Equ_Growth Monont for tilde(u_j)}
            \end{align}
            for every $s\in (0, s_0)$. Here $\gamma_j:= \gamma_2^+(\mbfC_{p_j}\Sigma_j)- \sigma$ and \[
                J_{j, K}^\gamma(v; \tau):= \int_{\AAa(\mathbf{0}, K^{-1}\tau, \tau)} v(x)^2|x|^{-n-2\gamma}\ d\|\mbfC_{p_j}\Sigma_j\|.   \]
            But since $u_j$ is of slower growth at $p_j$, and $\gamma_j < \gamma^+_2(\mbfC_{p_j}\Sigma_j) \leq \cA\cR_{p_j}(u_j)$, we know that
            \[
                \lim_{s\to 0} J^{\gamma_j}_{j, K}(\tilde{u}_j; s) = 0\,.
            \]
            Thus for any $N \in \mathbb{N}^+$, there is some $N_0 > N$ and $s = s_0\cdot K^{-N_0}$ such that 
            \[
                J^{\gamma_j}_{j, K}(\tilde{u}_j; K^{-1}s) < J^{\gamma_j}_{j, K}(\tilde{u}_j; s).
            \]
            And then inductively, (\ref{Equ_Growth Monont for tilde(u_j)}) implies, \[
                J^{\gamma_j}_{j, K}(\tilde{u}_j; K^{-l}s_0) < J^{\gamma_j}_{j, K}(\tilde{u}_j; s_0), \ \ \ \forall l\geq 1,\ \forall j\geq j_{s_0}.   \]
            By taking sum over $l$ and change the parametrization, this implies (recall $\|u_j\|_{L^2(\|\Sigma_j\|_{g_j})} = 1$),  
            \begin{align}
                \|u_j\|_{L^2(B^{g_j} (p_j, \tau);\|\Sigma_j\|_{g_j})} \leq C(\Sigma, g, s_0)\tau^{n/2 + \gamma_2^+(\mbfC_{p_j}\Sigma_j) -\sigma}, \ \ \ \forall \tau\in (0, s_0).   \label{Equ_growth bd for unit Jac field}
            \end{align}
            The key is that the coefficients on RHS does not depend on $j$ and $\tau$. 
            
            \begin{claim}\label{Claim:Lem_Cptness for Slower Growth Jac Fields}
                If $p_j\to p\in \Sing(\Sigma)$, then up to a subsequence, 
                \[
                    \gamma_2^+(\mbfC_{p_j}\Sigma_j)\to \gamma_2^+(\mbfC_p\Sigma)\,.
                \]
            \end{claim}
            \noindent \textit{Proof of Claim \ref{Claim:Lem_Cptness for Slower Growth Jac Fields}}
                It suffices to show that $\mbfC_{p_j}\Sigma_j$ $\mbfF$-subconverges to $\mbfC_p\Sigma$.
                
                By the fact that $\Sigma_j$ $\mbfF$-converges to $\Sigma$ and (\ref{Eqn:ratio_limit}), there exists a sequence $\set{r_k}$ with $r_k \rightarrow 0$ such that
                \[
                    \eta_{p_{j_k}, r_k} (\Sigma_{j_k}) \rightarrow \mbfC_p\Sigma
                \]
                in $\BB_3$, and furthermore,
                \[
                    \lim_k \theta_{|\eta_{p_{j_k}, r_k} (\Sigma_{j_k})|}(p_{j_k}, 2) - \theta_{|\eta_{p_{j_k}, r_k} (\Sigma_{j_k})|}(p_{j_k}) = 0
                \]
                
                Therefore, by Corollary \ref{Cor_Converg in all Scales}, we know that $\mbfC_{p_j}\Sigma_j$ $\mbfF$-converges to $\mbfC_p\Sigma$. $\hfill\blacksquare$

            Now, since 
            \[
                \gamma_2^+(\mbfC_{p_j}\Sigma_j) - \sigma >\gamma_1^+(\mbfC_{p_j}\Sigma_j) \geq -\frac{n-2}{2}\,,
            \] 
            (\ref{Equ_growth bd for unit Jac field}) then guarantees that there is no concentration of $L^2$ norm for $\{u_j\}_{j\geq 1}$ near $\Sing(\Sigma)$, i.e. 
            \[
                \limsup_{\tau\to 0}\limsup_{j\to \infty} \|u_j\|_{L^2(B^g(\Sing(\Sigma_j)), \tau)} = 0\,,   
            \] 
            and thus $\|u_\infty\|_{L^2} = 1$.  It also implies that for every $p\in \Sing(\Sigma)$,  
            \[
                \cA\cR_p(u_\infty) \geq \gamma_2^+(\mbfC_p\Sigma) -\sigma\,.
            \]
            Then by Lemma \ref{Lem_Pre_Asymp Rate takes value in Gamma(C) and int by part}, $u_\infty$ is of slower growth.
        \end{proof}
        
        \begin{Cor} \label{Cor_Upper Semi-conti of dim Ker^+ L}
            There exists $\kappa_1(g,\Sigma, \Lambda, G)\in (0, \delta_0)$ such that for every $(g', \Sigma')\in \cL^k(g, \Sigma; \Lambda, \kappa_1)^G$, we have $\dim (\Ker^+ L_{\Sigma', g'})^G \leq \dim (\Ker^+ L_{\Sigma, g})^G$.
        \end{Cor}
        \begin{proof}
            Suppose otherwise, there exist $(g_j, \Sigma_j)\to (g, \Sigma)$ in $\cL^{k,\alpha}(g, \Sigma; \Lambda, \delta_0)^G$ with \[
                \dim (\Ker^+ L_{\Sigma_j, g_j})^G > \dim (\Ker^+ L_{\Sigma, g})^G =: I.   \]
            For each $j\geq 1$, let $u_j^{(1)}, ..., u_j^{(I+1)}$ be an orthonormal class in $(\Ker^+ L_{\Sigma_j, g_j})^G$, i.e. $\langle u_j^{(s)}, u_j^{(t)} \rangle_{L^2(\|\Sigma_j\|_{g_j})} = \delta^{st}$, $1\leq s, t\leq I+1$.
            
            By Lemma \ref{Lem_Cptness for Slower Growth Jac Fields}, for each $1\leq s\leq I+1$, $u_j^{(s)}$ subconverges to some slower growth Jacobi field $u^{(s)}$ in $C^2_{loc}(\Sigma)$, and 
            \[
                \limsup_{r\to 0}\limsup_{j\to \infty} \|u_j^{(s)}\|_{L^2(B^g(\Sing(\Sigma), r); \|\Sigma_j\|_{g_j})} = 0.   
            \]
            Hence $\{u^{(s)}\}_{1\leq s\leq I+1} \subset (\Ker^+ L_{\Sigma, g})^G$ are also orthonormal, contradicting that $\dim (\Ker^+ L_{\Sigma, g})^G = I$.
        \end{proof}
        
        Let 
        \begin{align*}
            \cL^{k,\alpha}_{top}(g, \Sigma; \Lambda, \delta)^G:= \{(g', \Sigma')\in \cL^{k,\alpha}(g, \Sigma; \Lambda, \delta)^G:\ & \dim( \Ker^+ L_{\Sigma', g'})^G \\
            = &\ \dim (\Ker^+ L_{\Sigma, g})^G\}\,.
        \end{align*}
        Lemma \ref{Lem_Cptness for Slower Growth Jac Fields} and Corollary \ref{Cor_Upper Semi-conti of dim Ker^+ L} implies that $\cL^{k,\alpha}_{top}(g, \Sigma; \Lambda, \delta)^G$ is closed in $\cL^{k,\alpha}(g, \Sigma; \Lambda, \delta)^G$ for $\delta\leq \kappa_1$.
        
        Let $I:= \dim (\Ker^+ L_{\Sigma, g})^G < \infty$. For $(\bar{g}, \bar{\Sigma})\in \cL^{k,\alpha}_{top}(g, \Sigma; \Lambda, \kappa_1)^G$, let $\pi^{L^2, G}_{\bar{\Sigma}, \bar{g}}: L^2(\bar{\Sigma}) \to (\Ker^+ L_{\bar{\Sigma}, \bar{g}})^G$ be the $L^2$-orthogonal projection to the finite dimensional subspace. The following Lemma guarantees that we can parametrize slices of $\cL^{k,\alpha}_{top}$ by compact subset of $(\Ker^+ L)^G$.
        
        \begin{Lem} \label{Lem_Loc Parametriz of scL^k}
            Let $\kappa_1$ be as in Corollary \ref{Cor_Upper Semi-conti of dim Ker^+ L}. Then there exist constants $\kappa_2(g, \Sigma, \Lambda, G)\in (0, \kappa_1)$, $r_0(g, \Sigma, \Lambda, G)>0$ and an $I$-dimensional linear subspace $\cF\subset C^{k,\alpha}_c(M\setminus B^g(\Sing(\Sigma), {10r_0}))^G$ also depending only on $\Sigma, g, \Lambda, G$ with the following property.
            \begin{enumerate}[(i)]
                \item For every $(\bar{g}, \bar{\Sigma})\in \cL^{k,\alpha}_{top}(g, \Sigma; \Lambda, \kappa_2)^G$, the map
                      \begin{align}
                          \cF \to (\Ker^+ L_{\bar{\Sigma}, \bar{g}})^G,\ \ \ f\mapsto \pi^{L^2, G}_{\bar{\Sigma}, \bar{g}}(\nu_{\bar{\Sigma}, \bar{g}} (f)),  \label{Equ_Map cF to Ker^+ L}
                      \end{align}
                      is a linear isomorphism with biLipschitz constant $\leq C(g, \Sigma, \Lambda, k, \alpha, G)$.
                      
                      In particular, for every nonzero $f\in \cF$, there is no $G$-invariant solution of slower growth $v$ to $L_{\bar{\Sigma}, \bar{g}} v = \nu_{\bar{\Sigma}, \bar{g}}(f)$.
                \item Denote for simplicity $\cF\cdot \bar{g}:= \{(1+f)\bar{g}: f\in \cF\}$.  For every $(\bar{g}, \bar{\Sigma})\in \cL^{k,\alpha}_{top}(g, \Sigma; \Lambda, \kappa_2)^G$, the map
                      \begin{align*}
                          P_{\bar{g}, \bar{\Sigma}}: \cL^{k,\alpha}_{top}(g, \Sigma; \Lambda, \kappa_2)^G\cap \Pi^{-1}(\cF \cdot\bar{g}) & \to (\Ker^+ L_{\bar{\Sigma}, \bar{g}})^G,                                                                                                 \\
                          (g', \Sigma')                                                                                                  & \mapsto  \pi_{\bar{\Sigma}, \bar{g}}^{L^2, G}\big( G^{\Sigma'}_{\bar{\Sigma}, \bar{g}}\cdot \zeta_{\bar{\Sigma}, \bar{g}, r_0}\big),     
                      \end{align*}
                      is bi-Lipschitz onto its image with bi-Lipschitz constant $\leq C(g, \Sigma, \Lambda, G)$ and thus, injective, where 
                      \begin{align}
                          \zeta_{\bar{\Sigma}, \bar{g}, r_0}(x):= \zeta(\dist_{\bar{g}}(x, \Sing(\bar{\Sigma}))/r_0)   \label{Equ_Def eta_{Sigma, g, r_0}}
                      \end{align}
                      and $\zeta\in C^\infty(\RR; [0,1])$ is a cut-off function such that $\zeta \equiv 0$ on $(-\infty, 1]$ and $\zeta \equiv 1$ on $[2, +\infty)$.
            \end{enumerate}
        \end{Lem}
        
        \begin{Rem} \label{Rem_Metric Equi on cF}
            In this lemma, we choose an $L^2(M)$-metric on $\cF$. Since $\cF$ is of finite dimensional, this metric is equivalent to the $C^{k,\alpha}(M)$ metric.  The semi-metric we use on $\cL^{k,\alpha}_{top}$ is $\mbfD^\cL$.
        \end{Rem}
        
        \begin{proof} Fix a cut-off function $\zeta\in C^\infty(\RR; [0,1])$ such that $\zeta \equiv 0$ on $(-\infty, 1]$ and $\zeta \equiv 1$ on $[2, +\infty)$.
            
            To prove (i), we start with specifying the choice of $r_0$ and $\cF$. Take $r_0(g, \Sigma, \zeta)\in (0, \injrad(g, \Sigma)/40)$ small enough such that \[
                (\Ker^+ L_{\Sigma, g})^G \to (\Ker^+ L_{\Sigma, g})^G, \ \ \ v \mapsto \pi^{L^2, G}_{\Sigma, g}(\zeta_{\Sigma, g, r_0}\cdot v)   \]
            is an isomorphism, where $\zeta_{\Sigma, g, r_0}$ is defined in (\ref{Equ_Def eta_{Sigma, g, r_0}}). 
            Take $u_1, ..., u_I$ be an $L^2$-orthonormal basis of $(\Ker^+ L_{\Sigma, g})^G$ and choose $f_i\in C^{k,\alpha}_c(M\setminus \Sing(\Sigma))^G$ such that $\nu_{\Sigma, g}(f_i) = \zeta_{\Sigma, g, r_0}\cdot u_i$ along $\Sigma$. We let 
            \[
                \cF:= \mathrm{span}_\RR\langle f_i: 1\leq i\leq I\rangle\,.
            \]
            
            To prove (\ref{Equ_Map cF to Ker^+ L}) is an isomorphism, since $\dim (\Ker^+ L_{\bar{\Sigma}, \bar{g}})^G=\dim \cF$, it suffices to verify that for small $\delta$, for every $f\in \cF$ and every $(\bar{\Sigma},\bar{g})\in \cL^{k,\alpha}_{top}(g, \Sigma; \Lambda, \delta)^G$, \[
                \|\pi^{L^2, G}_{\bar{\Sigma}, \bar{g}}(\nu_{\bar{\Sigma},\bar{g}}(f))\|_{L^2(\bar{\Sigma})} \geq (L+1)^{-1}\|f\|_{L^2(M)},   \]
            where $L$ is the Lipschitz constant of the inverse of $\pi^{L^2, G}_{\Sigma, g}\comp\nu_{\Sigma, g}\big|_\cF$ on $(\Ker^+ L_{\Sigma, g})^G$.
            
            Suppose otherwise, there exists $(g_j, \Sigma_j) \to (g, \Sigma)$ in $\cL^{k,\alpha}_{top}(g, \Sigma; \Lambda, \kappa_1)^G$ and $L^2(M)$-unit $f_j\in \cF$ such that $v_j:= \pi^{L^2, G}_{\Sigma_j, g_j}(\nu_{\Sigma_j, g_j}(f_j))$ has $L^2(\Sigma_j)$-norm $< (L+1)^{-1}$.  Suppose WLOG $f_j \to f_\infty \in \cF$ in $C^{k,\alpha}(M)$ and by Lemma \ref{Lem_Cptness for Slower Growth Jac Fields}, $v_j$ subconverges to $v_\infty \in (\Ker^+ L_{\Sigma, g})^G$ in $C^2_{loc}(\Sigma)$, where by Fatou's Lemma $\|v_\infty\|_{L^2(\Sigma)}\leq (L+1)^{-1}$. Also by Lemma \ref{Lem_Cptness for Slower Growth Jac Fields}, orthonormal bases of $(\Ker^+ L_{\Sigma_j, g_j})^G$ subconverge to an orthonormal basis of $(\Ker^+ L_{\Sigma, g})^G$. Thus 
            \[
                \pi^{L^2, G}_{\Sigma, g}(\nu_{\Sigma, g}(f_\infty)) = \lim_{j\to \infty} v_j = v_\infty,
            \]
            which contradicts the choice of $L$.
            
            To prove the second part, it suffices to show that \textit{if a slower growth function $v$ solves $L_{\bar{\Sigma}, \bar{g}} v = \nu_{\bar{\Sigma}, \bar{g}}(f)$ for some $f\in \cF$, then $f\equiv 0$}. Multiply by elements of $(\Ker^+ L_{\bar{\Sigma}, \bar{g}})^G$ on both sides of the equation and integrate over $\bar{\Sigma}$. By Lemma \ref{Lem_Pre_Asymp Rate takes value in Gamma(C) and int by part} (iii), we can use integration by parts to derive
            \[
                \nu_{\bar{\Sigma}, \bar{g}}(f) \perp (\Ker^+ L_{\bar{\Sigma}, \bar{g}})^G,\text{ i.e. } \pi^{L^2, G}_{\bar{\Sigma}, \bar{g}}(\nu_{\bar{\Sigma}, \bar{g}}(f)) = 0\,,
            \] which implies $f\equiv 0$ by the isomorphism (\ref{Equ_Map cF to Ker^+ L}).
            
            To prove (ii), we also argue by contradiction. Suppose $P_{\bar{g}_j, \bar{\Sigma}_j}$ is not uniformly bi-Lipschitz for $(\bar{g}_j, \bar{\Sigma}_j)\in \cL^{k,\alpha}_{top}(g, \Sigma; \Lambda, \kappa_1)^G$ with $(\bar{g}_j, \bar{\Sigma}_j)\to (g, \Sigma)$ in $\cL^{k,\alpha}$. Then there exists a sequence of pairs $(g_j^0, \Sigma_j^0), (g_j^1, \Sigma_j^1)\to (g, \Sigma)$ in $\cL^{k,\alpha}_{top}(g, \Sigma; \Lambda, \delta)^G$ such that 
            \begin{enumerate} [(a)]
                \item $g_j^i = (1+f_j^i)\bar{g}_j$, $i=0, 1$, $f_j^i\in \cF$;
                \item One of the following holds,
                      \begin{align}
                          \text{either } & \  \|\pi^{L^2, G}_{\bar{\Sigma}_j, \bar{g}_j}((u_j^{(1)}-u_j^{(0)})\zeta_j)\|_{L^2(\bar{\Sigma}_j)} \leq \frac{d_j}{j}, \label{Equ_Either Large Lip inverse} \\  
                          \text{or }     & \ \|\pi^{L^2, G}_{\bar{\Sigma}_j, \bar{g}_j}((u_j^{(1)}-u_j^{(0)})\zeta_j)\|_{L^2(\bar{\Sigma}_j)} \geq j\cdot d_j; \label{Equ_Or Large Lip}
                      \end{align}
                      where $u_j^{(i)} := G_{\bar{\Sigma}_j, \bar{g}_j}^{\Sigma_j^i}$ is $G$-invariant; $\zeta_j:= \zeta_{\bar{\Sigma}_j, \bar{g}_j, r_0}$, $d_j:= \mbfD^{\cL}((g_j^0, \Sigma_j^0), (g_j^1, \Sigma_j^1))$.
            \end{enumerate}
            Since $\cF$ is of finite dimension and (\ref{Equ_Map cF to Ker^+ L}) is an isomorphism, Lemma \ref{Lem_Tangent vectors to scL^k} applies to conclude that when $j\to \infty$, $(f_j^1 - f_j^0)/d_j$ subconverges in $C^{k,\alpha}(M)$ to some $\hat{f}_\infty \in \cF$; $(u_j^{(1)}-u_j^{(0)})/d_j$ subconverges to some non-zero $G$-invariant slower growth $\hat{u}_\infty$ in $C^2_{loc}$, solving 
            \begin{align}
                L_{\Sigma, g} \hat{u}_\infty = \frac{n}{2}\nu_{\Sigma, g}(\hat{f}_\infty).   \label{Equ_Tang vec of scL cap cF solves L u = c nu(f)}
            \end{align}
            And since $\zeta_j$ is supported away from $\Sing(\Sigma)$, by the convergence above and Lemma \ref{Lem_Cptness for Slower Growth Jac Fields}, we know that (\ref{Equ_Or Large Lip}) can not hold for $j>>1$.  
            
            Also, by (\ref{Equ_Tang vec of scL cap cF solves L u = c nu(f)}) and (i), we derive $\hat{f}_\infty = 0$, in other words, $\mathbf{0}\neq \hat{u}_\infty \in (\Ker^+ L_{\Sigma, g})^G$. 
            Then since $\zeta_j(u_j^{(1)}-u_j^{(0)})/d_j$ subconverges in $L^2(\Sigma)$ to $\hat{u}_\infty\cdot \zeta_{\Sigma, g, r_0}$, by (\ref{Equ_Either Large Lip inverse}), $\pi^{L^2, G}_{\Sigma, g}(\hat{u}_\infty\cdot \zeta_{\Sigma, g, r_0}) = 0$. This violates the choice of $r_0$ in (i).   
        \end{proof}

        Having established this, we can prove Theorem \ref{Thm_Generic Semi-nondeg_Loc} by unwrapping the proof of Sard Theorem.  Let us start with the following abstract finite dimensional version of the theorem. 
        
        Let $Z\subset \RR^n$ be a compact subset; $f:Z \to \RR^n$ be a Lipschitz map. we call $x\in Z$ a \textbf{critical point} of $f$ if there exists some proper linear subspace $T_x \subset \RR^n$ such that for every sequence of distinct points $Z\ni x_j \to x$, if $(f(x_j)-f(x))/|x_j - x| \to v$ in $\RR^n$, then $v\in T_x$.  Let $\Crit_f$ be the set of critical points of $f$ on $Z$.  
        \begin{Lem} [Sard-Smale-type Theorem for Lipschitz Maps] \label{Lem_Sard Thm}
            Let $Z\subset \RR^n$ be a compact subset and $f:Z \to \RR^n$ be a Lipschitz map.   Then we have \[
                \scH^n(f(\Crit_f)) = 0.  \]
        \end{Lem}
        \begin{proof}
            Fix some $\epsilon>0$. For every $x\in \Crit_f$, let $\cB_x$ be the collection of balls $B$ in $\RR^n$ centered at $x$ with radius $\leq 1$ and such that $\scH^n(f(5B\cap Z))\leq \epsilon\scH^n(B)$, where $5B$ is the ball with same center and three times radius of $B$.  The definition of $\Crit_f$ guarantees that $\cB_x$ contains balls with arbitrarily small radius.
            
            By the Vitali covering lemma, there exists a disjoint family of balls $\{B_j\}_{j\geq 1}\subset \bigcup_{x\in \Crit_f} \cB_x$ such that $\Crit_f \subset \bigcup_{j\geq 1} 5B_j$.
            Thus, \[
                \scH^n(f(\Crit_f)) \leq \sum_{j\geq 1}\scH^n (f(5B_j\cap Z)) \leq \epsilon\sum_{j\geq 1}\scH^n(B_j) \leq \epsilon \scH^n(\BB_1(Z)).   \]
            Take $\epsilon\to 0$, we derive $\scH^n(f(\Crit_f)) = 0$.
        \end{proof}
        
        \begin{Thm} \label{Thm_Loc Generic G-inv Semi-nondeg}
            Let $G < \text{Diff}(M)$ be a fixed finite subgroup, $\Sigma \subset (M, g)$ be a connected $G$-invariant two-sided LSMH, $\Lambda>0$, and $\kappa_2 = \kappa_2(g, \Sigma, \Lambda, G)>0$ be given by Lemma \ref{Lem_Loc Parametriz of scL^k}.   Then for every $\delta\in (0, \kappa_2)$, 
            \begin{align}
                \begin{split}
                    \cG(g, \Sigma; \Lambda, \delta)^G:= \big\{g'\in \cG^{k, \alpha}(M)^G:  \text{either }g'\notin \Pi(\cL^{k,\alpha}(g, \Sigma; \Lambda, \delta)^G)& ;\\
                    \text{or }\forall (g', \Sigma')\in \cL^{k,\alpha}(g, \Sigma; \Lambda, \delta)^G & , \\
                    \text{ we have }(\Ker^+ L_{\Sigma', g'})^G = \set{0}& . \big\}  
                \end{split}  \label{Equ_Def_scG(g, Sigma; Lambda, delta)}
            \end{align}
            is both open and dense in $\cG^{k,\alpha}(M)^G:= \{g \in \cG^{k,\alpha}(M): \gamma^* g = g, \forall \gamma\in G\}$.
        \end{Thm}
        \begin{proof}
            We shall prove it inductively on $I:= \dim (\Ker^+ L_{\Sigma, g})^G\geq 0$.  Note that by Corollary \ref{Cor_Upper Semi-conti of dim Ker^+ L}, when $I = 0$, we always have $\cG(g, \Sigma; \Lambda, \delta)^G = \cG^{k,\alpha}(M)^G$, which is definitely open and dense.  Now consider $I\geq 1$.\\
            
            \noindent   \textit{Openness.} If $g_j\notin \cG(g, \Sigma; \Lambda, \delta)^G$ for $1\leq j<+\infty$ and $g_j\to g_\infty$ in $\cG^{k,\alpha}(M)^G$, then by definition, there exists $(g_j, \Sigma_j)\in \cL^{k,\alpha}(g, \Sigma; \Lambda, \delta)^G$ with a non-trivial $G$-invariant Jacobi field of slower growth $u_j \in (\Ker^+ L_{\Sigma_j, g_j})^G$.  
            By the compactness Lemma \ref{Lem_Cptness for scL^k}, up to a subsequence, $(g_j, \Sigma_j)\to (g_\infty, \Sigma_\infty)$ in $\cL^{k,\alpha}(g, \Sigma; \Lambda, \delta)^G$, and by Lemma \ref{Lem_Cptness for Slower Growth Jac Fields}, after renormalization, $u_j$ subconverges to some non-trivial $G$-invariant slower growth Jacobi field $u_\infty\in (\Ker^+ L_{\Sigma_\infty, g_\infty})^G$, which implies $g_\infty \notin \cG(g, \Sigma; \Lambda, \delta)^G$. Therefore, $\cG(g, \Sigma; \Lambda, \delta)^G$ is an open subset.\\
            
            \noindent   \textit{Denseness.} By induction, we can assume that the denseness has been established for every pair $(g, \Sigma)$ with $\dim (\Ker^+ L_{\Sigma, g})^G \leq I-1$. The goal now is to prove it for $I$($\geq 1$).
            \begin{claim}\label{claim:exist_seq}
                For each $\bar{g}\in \cG^{k,\alpha}(M)^G\setminus \cG(g, \Sigma; \Lambda, \delta)^G$, there exists a sequence of metrics $g_j \in \cF\cdot \bar{g}\setminus \Pi(\cL^{k,\alpha}_{top}(g, \Sigma; \Lambda, \delta)^G)$ such that $g_j\to \bar{g}$ in $C^{k,\alpha}(M)$ when $j\to \infty$.
            \end{claim}
            
            We first finish the proof of denseness assuming this claim. It suffices to show that each $g_j$ in the claim can be approximated by metrics in $\cG(g, \Sigma; \Lambda, \delta)^G$. 
            
            Let $j\geq 1$ be fixed. Suppose WLOG that $g_j\notin \cG(g, \Sigma; \Lambda, \delta)^G$. Hence $\Pi^{-1}(g_j)\cap \cL^k(g, \Sigma; \Lambda, \delta)^G\neq \emptyset$. 
            By definition of $(\cL^{k,\alpha}_{top})^G$ and Corollary \ref{Cor_Upper Semi-conti of dim Ker^+ L} we have, for every $(g_j, \Sigma')\in \cL^{k,\alpha}(g, \Sigma; \Lambda, \delta)^G$, $\dim (\Ker^+ L_{\Sigma', g_j})^G \leq I-1$. Hence by the induction assumption, there exists $\delta_{g_j, \Sigma'} = \delta(g_j, \Sigma', \Lambda)>0$ such that $\cG(g_j, \Sigma'; \Lambda, \delta_{g_j, \Sigma'})^G$ is open and dense in $\cG^{k,\alpha}(M)^G$. 
            Since $\cL^{k,\alpha}(g_j, \Sigma'; \Lambda, \delta_{g_j, \Sigma'})^G$ contains an open neighborhood of $(g_j, \Sigma')$ in $\cL^{k,\alpha}(g, \Sigma; \Lambda, \delta)^G$, and by Lemma \ref{Lem_Cptness for scL^k}, $\Pi^{-1}(g_j) \cap \cL^{k,\alpha}(g, \Sigma; \Lambda, \delta)^G$ is compact, we thus know that there exists a finite sequence of pairs $(g_j, \Sigma^{(1)})$, $(g_j, \Sigma^{(2)})$, $\dots$, $(g_j, \Sigma^{(q)}) \in \Pi^{-1}(g_j)\cap \cL^{k,\alpha}(g, \Sigma; \Lambda, \delta)^G$, $1\leq q<+\infty$, such that 
            \begin{align}
                \bigcup_{l=1}^q \cL^{k,\alpha}(g_j, \Sigma^{(l)}; \Lambda, \delta_{g_j, \Sigma^{(l)}})^G \supset \Pi^{-1}(g_j) \cap \cL^{k,\alpha}(g, \Sigma; \Lambda, \delta)^G.  \label{Equ_Pi^-1(g_j) covered by smaller nullity neighb}
            \end{align}
            Since $\cG_j := \bigcap_{l=1}^q \cG(g_j, \Sigma^{(l)}; \Lambda, \delta_{g_j, \Sigma^{(l)}})^G$ is still open and dense in $\cG^(M)^G$, there exists a sequence $\{g'_m\}_{m\geq 1} \subset \cG_j$ such that $g'_m \to g_j$ in $C^{k,\alpha}$ as $m\to \infty$. 
            
            We now verify that $g'_m \in \cG(g, \Sigma; \Lambda, \delta)^G$ for $m>>1$.
            Suppose otherwise, for infinitely many $m$, there exists $(g'_m, \Sigma'_m)\in \cL^{k,\alpha}(g, \Sigma; \Lambda, \delta)^G$ with $(\Ker^+ L_{\Sigma'_m, g'_m})^G \neq \set{0}$. By Lemma \ref{Lem_Cptness for scL^k}, up to a subsequence of $m\to \infty$, $(g'_m, \Sigma'_m)\to (g_j, \Sigma')$ in $\cL^{k,\alpha}(g, \Sigma; \Lambda, \delta)^G$. By (\ref{Equ_Pi^-1(g_j) covered by smaller nullity neighb}), WLOG $(g_j, \Sigma')\in \cL^{k,\alpha}(g_j, \Sigma^{(1)}; \Lambda, \delta_{g_j, \Sigma^{(1)}})^G$, and so does $(g'_m, \Sigma'_m)$ for $m>>1$. But since $g'_m \in \cG_j \subset \cG(g_j, \Sigma^{(1)}; \Lambda, \delta_{g_j, \Sigma^{(1)}})^G$, by its definition (\ref{Equ_Def_scG(g, Sigma; Lambda, delta)}), we must have $(\Ker^+ L_{\Sigma'_m, g'_m})^G = \set{0}$ for those $m$, which contradicts the choice of $(g'_m, \Sigma'_m)$.  This concludes the proof of denseness.

            \noindent\textit{Proof of Claim \ref{claim:exist_seq}.} Suppose WLOG that $(\bar{g}, \bar{\Sigma})\in \cL^{k,\alpha}_{top}(g, \Sigma; \Lambda, \delta)^G$, otherwise just taking $g_j \equiv \bar{g}$ proves the Claim. 
            
            Recall by Lemma \ref{Lem_Loc Parametriz of scL^k}, $\cL^k_{top}(g, \Sigma; \Lambda, \delta)^G\cap \Pi^{-1}(\cF \cdot \bar{g})$ is bi-Lipschitzly embedded as a compact subset of finite dimensional Euclidean space $(\Ker^+ L_{\bar{\Sigma},\bar{g}})^G$. We denote for simplicity $\cZ:= P_{\bar{g}, \bar{\Sigma}}(\cL^k_{top}(g, \Sigma; \Lambda, \delta)^G\cap \Pi^{-1}(\cF\cdot \bar{g}))$.  Consider the projection map, \[
                \tilde{\Pi}: \cZ \cong \cL^k_{top}(g, \Sigma; \Lambda, \delta)^G\cap \Pi^{-1}(\cF\cdot \bar{g}) \overset{\Pi}{\to} \cF\cdot \bar{g}\cong \cF.   \]
            More precisely, $\tilde{\Pi}$ is a Lipschitz map between (compact subset of) vector spaces of same dimension, given by $\tilde{\Pi}(P_{\bar{g}, \bar{\Sigma}}((1+f)\bar{g}, \Sigma')) = f$. 

            We shall show that $\cZ$ consists of critical points in Claim \ref{claim:Z_subset_crit}. Then by Lemma \ref{Lem_Sard Thm}, we see that $\Pi(\cL^{k,\alpha}_{top}(g, \Sigma; \Lambda, \delta)^G)$ is a compact subset of $I$-dimensional measure $0$ in $\cF\cdot\bar{g}\cong \cF$.  Therefore, we can find $g_j \in \cF\cdot \bar{g}\setminus \Pi(\cL^{k,\alpha}_{top}(g, \Sigma; \Lambda, \delta)^G)$ to approximate $\bar{g}$, and thus, finish the proof of Claim \ref{claim:exist_seq}. $\hfill\blacksquare$
            
            \begin{claim}\label{claim:Z_subset_crit}
                $\cZ\subset \Crit_{\tilde{\Pi}}$.
            \end{claim}

            \noindent \textit{Proof of Claim \ref{claim:Z_subset_crit}.} By the definition of $\Crit_f$ and biLipschitz property of $P_{\bar{g}, \bar{\Sigma}}$, it suffices to show that \textit{whenever $((1+f_j)\bar{g}, \Sigma'_j) \to ((1+f)\bar{g}, \bar{\Sigma}')$ in $\cL^{k,\alpha}_{top}(g, \Sigma; \Lambda, \delta)^G$, then we always have when $j\to \infty$,}
            \begin{align}
                (f_j - f)/d_j \to 0\,,  \label{Equ_(f_j - f)/d_j to 0}
            \end{align}
            where with $d_j:= \mbfD^{\cL}(((1+f_j)\bar{g}, \Sigma'_j), ((1+f)\bar{g}, \bar{\Sigma}'))>0$.

            To see this, recall by Lemma \ref{Lem_Tangent vectors to scL^k}, after passing to a subsequence, 
            \begin{align*}
                \begin{cases} 
                    u_j/d_j \to \hat{u}_\infty            & \ \text{ in } C^2_{loc}(\bar{\Sigma}'); \\
                    (f_j - f)/((1+f)d_j) \to \hat{f}_\infty & \ \text{ in } C^{k,\alpha}(M).
                \end{cases}
            \end{align*}
            where $u_j:= G^{\Sigma_j'}_{\bar{\Sigma}', (1+f)\bar{g}}$ is the graphical function of $\Sigma_j'$ over $\bar{\Sigma}'$, $\hat{u}_\infty$ is of slower growth, and $\hat{f}_\infty$ solves the equation 
            \[
                L_{\bar{\Sigma}', (1+f)\bar{g}} \hat{u}_\infty = \frac{n}{2}\nu_{\bar{\Sigma}', (1+f)\bar{g}} (\hat{f}_\infty).  
            \]
            But by Lemma \ref{Lem_Loc Parametriz of scL^k} (i), this equation has no slower growth solution unless $\hat{f}_\infty \equiv 0$, which proves (\ref{Equ_(f_j - f)/d_j to 0}). $\hfill\blacksquare$
            
        \end{proof}
        
        \begin{proof}[Proof of Theorem \ref{Thm_Generic Semi-nondeg_Loc}.]
            The case where $\Sigma$ is two-sided follows immediately from Theorem \ref{Thm_Loc Generic G-inv Semi-nondeg} with $G = \{\mathrm{Id}_M\}$. 
            
            When $\Sigma$ is one-sided, by Remark \ref{Rem_one-sided, G-inv, induced rep} (ii), we can consider the two-sided double cover $\hat{\Sigma} \subset (\hat{M}, \hat{g})$, a $\ZZ_2$-invariant LSMH.  Recall that by definition, one-sided $(g', \Sigma') \in \cL^{k,\alpha}(g, \Sigma; \Lambda, \delta)$ is semi-nondegenerate if and only if its double cover $\hat{\Sigma}'$ is $\ZZ_2$-semi-nondegenerate. 
            Hence Theorem \ref{Thm_Loc Generic G-inv Semi-nondeg} also implies Theorem \ref{Thm_Generic Semi-nondeg_Loc} with $G = \ZZ_2$.
        \end{proof}
        
{}        
    \section{Proof of Theorem \ref{Thm_Countable Decomp}}\label{Sec:Count_Decomp}
        The goal of this section is to prove Theorem \ref{Thm_Countable Decomp}. To achieve it, we shall introduce a two-step decomposition of $\cM^{k, \alpha}(M)$, where the first step is inspired by Edelen's cone decomposition while the latter is obtained by a compactness argument.
        
        \subsection{Tree representation and first decomposition}
            
            First, we introduce a tree structure representing the cone decomposition of a minimal hypersurface near a cone.
            
            \begin{Def}[Tree representation of cone decomposition]
                Given a $(\theta, \beta, \cS, N)$-cone decomposition of $V\llcorner (\BB(x, R), g)$ as in Definition \ref{Def:cone_decomposition} with parameters:
                \begin{itemize}
                    \item Integers $N_S$, $N_C$ satisfying $N_S + N_C \leq N$;
                    \item Points $\set{x_a}_a, \set{x_b}_b \subset \BB(x, R)$;
                    \item Radii $\set{R_a, \rho_a| R_a \geq 2 \rho_a}_a, \set{R_b}_b$;
                    \item Indices $\set{s_b}_b$;
                    \item Integers $\set{1 \leq m_a \leq [\theta]}_a$;
                    \item Cones $\set{\mathbf{C}_a}_a \subset \cC$,
                \end{itemize}
                where $a = 1, \cdots, N_C$ and $b = 1, \cdots, N_S$.
                The corresponding \textbf{tree representation} of the cone decomposition is a rooted tree (\cite[Section~B.5]{cormen2022introduction}) uniquely defined by:
                \begin{enumerate}
                    \item There are two types of nodes: every node of \textit{type I} is labeled with $(\mathbf{C}_a, m_a, x_a, R_a, \rho_a)$, while every node of \textit{type II} with $(S_{s_b}, x_b, R_b)$;
                    \item The root is labeled with either $(\mathbf{C}_a, m_a, x_a = x, R_a = R, \rho_a)$ or $(S_{s_b}, x_b = x, R_b = R)$;
                    \item For any type-I node $(\mathbf{C}_a, m_a, x_a, R_a, \rho_a)$, either $\rho_a = 0$, $\theta_{\mbfC_a}(0) > 1$ and it is a leaf; Or $\rho_\alpha > 0$ and it has a unique child of either
                          \begin{itemize}
                              \item type I $(\mathbf{C}_{a'}, m_{a'} = m_a, x_{a'} = x_a, R_{a'} = \rho_a, \rho_{a'})$, or
                              \item type II $(S_{s_{b'}}, x_{b'} = x_a, R_{b'} = \rho_a)$;
                          \end{itemize}
                    \item For any type-II node $(S_{s_b}, x_b, R_b)$ where $S_{s_b} = (S, \mathbf{C}, \set{\mathbf{C}_{{\hat\alpha}}, m_{{\hat\alpha}}, \BB(y_{{\hat\alpha}}, r_{{\hat\alpha}})}_{{\hat\alpha}\in I_b})$, it has $\sharp I_b$ child nodes such that for each ${\hat\alpha}$, there exists $R_{b,{\hat\alpha}}$ and $x_{b,{\hat\alpha}}$ such that
                          \[
                              |x_{b,{\hat\alpha}} - (x_b + R_b \cdot y_{\hat\alpha})| \leq \beta R_b r_{\hat\alpha}, \quad \frac{1}{2}\leq \frac{R_{b,{\hat\alpha}}}{R_b r_{\hat\alpha}} \leq 1 + \beta\,,
                          \]
                          so that the corresponding child node is either 
                          \begin{itemize}
                              \item of type I $(\mathbf{C}_{a'} = \mathbf{C}_{{\hat\alpha}}, m_{a'} = m_{\hat\alpha}, x_{a'} = x_{b, {\hat\alpha}}, R_{a'} = R_{b, {\hat\alpha}}, \rho_{a'})$, or
                              \item of type II $(S_{s_{b'}}, x_{b'} = x_{b, {\hat\alpha}}, R_{b'} = R_{b, {\hat\alpha}})$;
                          \end{itemize}
                \end{enumerate}
                
                The \textbf{coarse tree representation} of the cone decomposition is obtained by relabeling the rooted tree above, i.e., replacing the type-I node $(\mathbf{C}_a, m_a, x_a, R_a, \rho_a)$ by $(\theta_{\mathbf{C}_a}(0), m_a)$, and the type-II node $(S_{s_b}, x_b, R_b)$ by $S_{s_b}$.
            \end{Def}
            
            \begin{Def}
                For $\gamma \in (0, 1/100)$, two $(\theta, \beta, \cS, N)$-tree representaions with parameters
                \begin{itemize}
                    \item $(N_S, N_C, \set{x_a}, \set{x_b}, \set{R_a}, \set{\rho_a}, \set{R_b}, \set{m_a}, \set{\mbfC_a} \set{s_b})$,
                    \item $(N'_S, N'_C, \set{x'_a}, \set{x'_b}, \set{R'_a}, \set{\rho'_a}, \set{R'_b}, \set{m'_a}, \set{\mbfC'_a}, \set{s'_b})$,
                \end{itemize}
                are said to be \textbf{$\gamma$-close} if $N'_S = N_S$, $N'_C = N_C$ and they have the same coarse tree representations, s.t.
                \begin{enumerate}
                    \item If the corresponding two nodes are both of type I, then
                          \begin{itemize}
                              \item $\dist_H(\mathbf{C}_a \cap \partial B_1,\mathbf{C}_{a'} \cap \partial B_1) \leq \gamma$,
                              \item If $\rho_a > 0$, then
                                    \begin{itemize}
                                        \item $|\rho_a - \rho_{a'}|\leq \gamma \min(\rho_a, \rho_{a'})$;
                                        \item $|x_a - x_{a'}| \leq \gamma \min(\rho_a, \rho_{a'})$;
                                        \item $|R_a - R_{a'}| \leq \gamma \min(\rho_a, \rho_{a'})$;
                                    \end{itemize}
                                    Otherwise $\rho_a = 0$, then
                                    \begin{itemize}
                                        \item $\rho_{a'} = 0$;
                                        \item $|x_a - x_{a'}| \leq \gamma \min(R_a, R_{a'})$;
                                        \item $|R_a - R_{a'}| \leq \gamma \min(R_a, R_{a'})$.
                                    \end{itemize}
                          \end{itemize}
                    \item If the corresponding two nodes are both of type II, then
                          \begin{itemize}
                              \item $|x_b - x_{b'}| \leq \gamma \min(R_b, R_b') \min_{{\hat\alpha} \in I_b}(r_{\hat\alpha})$,
                              \item $|R_b - R_{b'}| \leq \gamma \min(R_b, R_b') \min_{{\hat\alpha} \in I_b}(r_{\hat\alpha})$.
                          \end{itemize}
                \end{enumerate}
            \end{Def}
            
            In order to deal with minimal hypersurfaces in a closed Riemannian manifold directly, we also introduce a large-scale cone decomposition for a minimal hypersurface in a closed Riemannian manifold.
            
            \begin{Def}[Large-scale Cone Decomposition]
                Given $\theta, \gamma, \beta \in \mathbb{R}_+$, $\sigma \in (0, 1/3)$, and $N \in \mathbb{N}$, let \begin{itemize}
                    \item $g_0$, $g$ be two $C^{k, \alpha}$ metrics on $N$;
                    \item $\Sigma_0$, $\Sigma$ be two minimal hypersurfaces in $(N, g_0)$ and $(N, g)$ respectively;
                    \item $\cS = \set{S_s}_{s}$ be a finite collection of $(\theta, \sigma, \gamma)$-smooth models;
                \end{itemize}
                A \textbf{large-scale $(\theta, \beta, g_0, \Sigma_0, \mathcal{S}, N)$-cone decomposition} of $\Sigma$ consists of:
                \begin{itemize}
                    \item a collection of radii $\set{r_{\hat\alpha}}_{\hat\alpha}$ corresponding to the singular sets $\Sing(\Sigma_0) = \set{p_{\hat\alpha}}_{\hat\alpha}$, such that $\set{B^g(p_{\hat\alpha}, r_{\hat\alpha})}_{\hat\alpha}$ are pairwise disjoint;
                    \item a $(\theta, \beta, \cS, N)$-cone decomposition for each $|\Sigma| \llcorner B^g(p_{\hat\alpha}, r_{\hat\alpha})$;
                    \item a $C^2$ function $u: \Sigma_0 \setminus \bigcup_{p_{\hat\alpha} \in \Sing(\Sigma_0)} B^g(p_{\hat\alpha}, {r_{\hat\alpha}/2}) \rightarrow \Sigma_0^\perp$ so that for $r_0 = \min_{\hat\alpha}\set{r_{\hat\alpha}} > 0$,
                          \[
                              r_0^{-1}|u| + |\nabla u| + r_0 |\nabla^2 u| \leq \beta\,,
                          \]
                          and $\Sigma \setminus B^g(p_{\hat\alpha}, {r_{\hat\alpha}})$ coincides with $\graph_{\Sigma_0}(u) \setminus B^g(p_{\hat\alpha}, {r_{\hat\alpha}})$;
                \end{itemize}
            \end{Def}
            
            Similarly, we can define the corresponding tree representation and the $\gamma$-closeness.
            
            \begin{Def}[Tree representation of Large-scale cone decomposition]
                Given a large-scale $(\theta, \beta, g_0, \Sigma_0, \cS, N)$-cone decomposition of $\Sigma \subset (N, g)$ with parameters:
                \begin{itemize}
                    \item $\Sing(\Sigma_0) = \set{p_{\hat\alpha}}_{\hat\alpha}$;
                    \item radii $\set{r_{\hat\alpha}}_{\hat\alpha}$;
                    \item $(\theta, \beta, \cS, N)$-cone decompositions for each $\Sigma \cap B^g(p_{\hat\alpha}, r_{\hat\alpha})$;
                \end{itemize}
                The corresponding \textbf{tree representation} of the large-scale cone decomposition is a rooted tree uniquely defined by:
                \begin{enumerate}
                    \item The root node is a labeled by a tuple $(\Sigma_0, g_0, \set{p_{\hat\alpha}}, \set{r_{\hat\alpha}})$;
                    \item The root node has $\sharp \Sing(\Sigma_0)$ children, indexed by $\hat\alpha$. The corresponding subtree rooted at the $\hat\alpha$-child is the tree representation of the $(\theta, \beta, \cS, N)$-cone decomposition for each $|\Sigma|\llcorner B^g(p_{\hat\alpha}, {r_{\hat\alpha}})$.
                \end{enumerate}
                
                Similarly, the \textbf{coarse tree representation} will be the same directed rooted tree with the subtrees above replaced by their corresponding coarse trees.
            \end{Def}
            
            \begin{Def}
                For $\gamma \in (0, 1/100)$, two $(\theta, \beta, g_0, \Sigma_0, \cS, N)$-tree representaions ($\beta \leq \gamma$) are said to be \textbf{$\gamma$-close} if 
                \begin{itemize}
                    \item Their root nodes have the same label;
                    \item Their subtrees corresponding to the $\hat\alpha$-child are $\gamma$-close for each $\hat\alpha$.
                \end{itemize}
            \end{Def}
            
            \begin{Thm}
                For any given $(g, \Sigma) \in \cM^{k, \alpha}(M)$, $\beta \in (0, 1/100)$, and $I , \Lambda \in \mathbb{N}$ with
                \[
                    \scH^7(\Sigma, g) \leq \Lambda,\ \ind(\Sigma, g) \leq I\,,
                \]
                there exist $\delta(g, \Sigma, \beta, \Lambda, I) > 0$ and $l(g, \Sigma) > 0$ satisfying the following property.
                
                For a $\delta$-``neighborhood'' of $(g, \Sigma)$, 
                \[
                    \begin{split}
                        \cM^{k, \alpha}(g, \Sigma; \Lambda, I, \beta) := \Big\{ (g', \Sigma') \in \cM^{k,\alpha}(M): \|g'\|_{C^{k, \alpha}} \leq \Lambda; \|g - g'\|_{C^{k - 1, \alpha}} + \mathbf{F}(\Sigma, \Sigma') \leq \delta\,,\\
                        \scH^7(\Sigma', g') \leq \Lambda,\ \ind(\Sigma', g') \leq I\Big\} \,.
                    \end{split}
                \]
                we have 
                \begin{enumerate}
                    \item There exist a finite collection of $(\theta, \sigma, \beta)$-smooth models $\cS$ and an integer $N$, so that any $(g', \Sigma') \in \cM^{k, \alpha}(g, \Sigma; \Lambda, I, \beta)$ admits a large-scale $(\tilde \theta_l, \beta, g, \Sigma, \cS, N)$-cone decomposition, where $l$ depends on $(g, \Sigma)$ only;
                    \item There exists a countable collection $\set{(g_v, \Sigma_v)}_{v \in \mathbb{N}} \subset \cM^{k, \alpha}(g, \Sigma; \Lambda, I, \beta)$ with fixed large-scale $(\tilde \theta_l, \beta, g, \Sigma, \cS, N)$-cone decompositions with the following property. Every $(g', \Sigma') \in \mathcal{M}^{k, \alpha}(g, \Sigma; \Lambda, I)$ admits a large-scale $(\tilde \theta_l, \beta, g, \Sigma, \cS, N)$-cone decomposition whose tree representation is $\beta$-close to that of some $(g_v, \Sigma_v)$.
                \end{enumerate}
            \end{Thm}
            \begin{Rem}\label{Rem:Fst_Decomp}
                By the second bullet of the theorem, for each $v$, we can define an ``intermediate canonical neighborhood'' $\cL^{k, \alpha}_0(g_v, \Sigma_v; \Lambda, I ,\beta)$ (we omit $g$ and $\Sigma$ here for simplicity) consisting of those pairs $(g', \Sigma')$'s satisfying:
                \begin{itemize}
                    \item $\|g'\|_{C^{k, \alpha}} \leq \Lambda$;
                    \item $\|g - g'\|_{C^{k-1, \alpha}} + \mathbf{F}(\Sigma, \Sigma') \leq \delta$;
                    \item $\scH^7(\Sigma', g') \leq \Lambda,\ \ind(\Sigma', g') \leq I$;
                    \item $(g', \Sigma')$ admits a large-scale $(\Lambda, \beta, g, \Sigma, \cS, N)$-cone decomposition whose tree representation is $\beta$-close to that of $(g_v, \Sigma_v)$.
                \end{itemize}
                
                Then, the theorem implies that 
                \[
                    \cM^{k, \alpha}(g, \Sigma; \Lambda, I, \beta) = \bigcup_v \cL^{k, \alpha}_0(g_v, \Sigma_v; \Lambda, I ,\beta)\,.
                \]
            \end{Rem}
            \begin{proof}
                \noindent{(1)} The first part follows from the existence theorem of cone decompositions, i.e., Theorem \ref{Thm:cone_decomposition}. 
                
                Indeed, given $(g, \Sigma)$, we can take $l \in \mathbb{N}$ so that $\tilde \theta_l = \sup_{x \in \Sing \Sigma}\set{\theta_\Sigma(x, 0)}$. Then for each ${p_{\hat\alpha}} \in \Sing(\Sigma)$, there exists a radius $r_{\hat\alpha} \in (0, 1)$ such that 
                \begin{itemize}
                    \item $\set{B^g({p_{\hat\alpha}}, {2r_{{\hat\alpha}}})}$ are pairwise disjoint;
                    \item $|(2r_{\hat\alpha})^{-2} (\eta_{{p_{\hat\alpha}}, 2r_{\hat\alpha}}^{-1})^*g - g_{\mathrm{eucl}}|_{C^3(\mathbb{B}_1)} \leq \delta_{l, I}/2$;
                    \item There exists a cone $\mathbf{C}_{p_{\hat\alpha}} \in \mathcal{C}$ such that
                          \[
                              \dist_H(\eta_{{p_{\hat\alpha}}, 2r_{\hat\alpha}}(\Sigma)\cap \mathbb{B}_1, \mathbf{C}_{p_{\hat\alpha}} \cap \mathbb{B}_1) \leq \delta_{l, I}/2\,,
                          \]
                          \[
                              \theta_{\mathbf{C}_{p_{\hat\alpha}}}(0) \leq \tilde \theta_l\,,
                          \]
                          and 
                          \[
                              \frac{3}{4} \theta_{\mathbf{C}_{p_{\hat\alpha}}}(0) \leq \theta_{\eta_{{p_{\hat\alpha}}, 2r_{\hat\alpha}}(\Sigma)}(0, 1/2),\quad \theta_{\eta_{{p_{\hat\alpha}}, 2r_{\hat\alpha}}(\Sigma)}(0, 1)\leq \frac{5}{4} \theta_{\mathbf{C}_{p_{\hat\alpha}}}(0)\,. 
                          \]
                \end{itemize}
                
                By taking $\delta$ small enough, we may see that for any $(g', \Sigma') \in \cM^{k, \alpha}(g, \Sigma; \Lambda, I, \beta)$, we have
                \begin{itemize}
                    \item By Allard's regularity theorem \cite{allard_first_1972}, there exists a $C^2$ function $u: \Sigma \setminus \bigcup_{p_{\hat\alpha} \in \Sing(\Sigma)} B^g({p_{\hat\alpha}}, r_{\hat\alpha}/2) \rightarrow \Sigma^\perp$ so that for $r_0 = \min\set{r_{\hat\alpha}}$,
                          \[
                              r_0^{-1}|u| + |\nabla u| + r_0 |\nabla^2 u| \leq \beta\,,
                          \]
                          and $\Sigma' \setminus \bigcup B^g({p_{\hat\alpha}}, r_{\hat\alpha}) $ coincides with $\graph_{\Sigma}(u) \setminus \bigcup B^g({p_{\hat\alpha}}, r_{\hat\alpha}) $;
                    \item $|(r_{\hat\alpha})^{-2}(\eta_{{p_{\hat\alpha}}, r_{\hat\alpha}}^{-1})^*g'  - g_{\eucl}|_{C^3(\mathbb{B}_1)} \leq \delta_{l, I}$ for each ${p_{\hat\alpha}}$;
                    \item \[
                              \dist_H(\eta_{{p_{\hat\alpha}}, r_{\hat\alpha}}(\Sigma')\cap \mathbb{B}_1, \mathbf{C}_{p_{\hat\alpha}} \cap \mathbb{B}_1) \leq \delta_{l, I}\,,
                          \]
                          and 
                          \[
                              \frac{1}{2}\theta_{\mathbf{C}_{p_{\hat\alpha}}}(0) \leq \theta_{\eta_{{p_{\hat\alpha}}, r_{\hat\alpha}}(\Sigma')}(0, 1/2),\quad \theta_{\eta_{{p_{\hat\alpha}}, r_{\hat\alpha}}(\Sigma')}(0, 1)\leq \frac{3}{2} \theta_{\mathbf{C}_{p_{\hat\alpha}}}(0)\,; 
                          \]
                \end{itemize}
                
                By Theorem \ref{Thm:cone_decomposition}, the last two bullets imply the existence of a $(\tilde \theta_l, \beta, \mathcal{S}_\alpha, N_\alpha)$-cone decomposition for each $\Sigma' \cap B^g({p_{\hat\alpha}}, {r_{{\hat\alpha}}}) $. Together with the first bullet, we obtain a large-scale $(\tilde \theta_l, \beta, g, \Sigma, \cS, N)$-cone decomposition of $(\Sigma', g')$, where $N = \sum_{\hat\alpha} N_{\hat\alpha}$ and $\mathcal{S} = \bigcup_{\hat\alpha} \mathcal{S}_{\hat\alpha}$.\\
                
                \noindent{(2)}
                By (1), for each $(g', \Sigma')$, we can fix a large-scale $(\tilde \theta_l, \beta, g, \Sigma, \mathcal{S}, N)$-cone decomposition, and its corresponding tree representation. It follows from the finiteness of $N$, $\cS$, $m_{\hat\alpha}$'s and the discreteness of $\tilde \theta_i$'s that there are only finitely many distinct coarse tree representations. 
                
                Let's fix one such a coarse tree representation $\tilde T = (\tilde V, \tilde E)$, and consider the set $\cL'(\tilde T)$ of all the pairs $(g', \Sigma')$ associated to this coarse tree representation. The goal will be to find a countable covering of $\cL'(\tilde T)$.
                
                To begin with, we construct countable coverings for each node of the coarse tree.
                
                \begin{enumerate}[(a)]
                    \item By the definition of coarse tree representation, all tree representations of pairs in $\cL'$ have the same root node;
                    \item Each type-I node $\tilde u$ associated with $S_{s_b} = (S, \mathbf{C}, \set{\mathbf{C}_{{\hat\alpha}}, m_{{\hat\alpha}}, \BB(y_{{\hat\alpha}}, r_{\hat\alpha})}_{{\hat\alpha}\in I_b}) \in \mathcal{S}$ in $\tilde V$ corresponds to the set of type-I nodes in $\cL'_{\tilde u}$
                          \[
                              \cL'_{\tilde u} = \set{(S_{s_b}, x_{\hat\beta}, R_{\hat\beta})}_{{\hat\beta}}\,,
                          \]
                          where $x_{\hat\beta} \in M$ and $R_{\hat\beta} \in (0, \infty)$.
                          
                          Let $r_0 = \min(\min_{{\hat\alpha} \in I_b}(r_{\hat\alpha}), 1/2)$ and define a countable collection of intervals 
                          \[
                              \mathrm{IC}_{\tilde u}:= \set{\mathrm{Int}_k := [(1 + \frac{\gamma r_0}{2})^k, (1 + \frac{\gamma r_0}{2})^{k+1}]}_{k \in \mathbb{Z}}
                          \] which covers $(0, \infty)$ (``I'' stands for interval and ``C'' for covering). For each $k$, we associate it with a finite covering $\mathrm{BC}_{\tilde u,k}$ of $(M, g)$ by geodesic balls of radius no greater than 
                          \[
                              \min(\mathrm{injrad}(M, g), (1 + \frac{\gamma r_0}{2})^k \gamma r_0/10).
                          \]
                          Therefore, we obtain a countable covering of the product space $M \times (0, \infty)$,
                          \[
                              \mathrm{PC}_{\tilde u} := \set{(B, \mathrm{Int}_k): k \in \mathbb{Z}, B \in \mathrm{BC}_{\tilde u, k}}\,.
                          \]
                          
                          Note that if $(x_{\hat\beta}, R_{\hat\beta})$ and $(x_{{\hat\beta}'}, R_{{\hat\beta}'})$ are in the same $(B, \mathrm{Int}_k) \in \mathrm{PC}_{\tilde u}$, then $(S_{s_b}, x_{\hat\beta}, R_{\hat\beta})$ and $(S_{s_b}, x_{{\hat\beta}'}, R_{{\hat\beta}'})$ satisfy the inequalities in the first part of the definition of $\gamma$-closeness.
                          
                    \item Each type-II node $\tilde w$ associated with $(\theta_{\mathbf{C}}(\mathbf{0}), m)$ in $\tilde V$ corresponds to the set of type-I nodes in $\cL'_{\tilde w}$
                          \[
                              \cL'_{\tilde w} = \set{(\mathbf{C}_{\hat\beta}, m, x_{\hat\beta}, R_{\hat\beta}, \rho_{\hat\beta})}
                          \]
                          where $\theta_{\mathbf{C}_{\hat\beta}}(\mathbf{0}) = \theta_{\mathbf{C}}(\mathbf{0})$, $x_{\hat\beta} \in M$, $R_{\hat\beta} \in (0, \infty)$ and $\rho_{\hat\beta} \in [0, \infty)$.
                          
                          By the compactness of stable cones with the same density $\theta_{\mathbf{C}}(\mathbf{0})$, we can find a finite set of cones
                          \[
                              \set{\mathbf{C}_i}^N_{i = 1}\,,
                          \]
                          where $N$ depends on $\gamma$ such that any other stable cone $\mathbf{C}'$ with $\theta_{\mathbf{C}'}(\mathbf{0}) = \theta_{\mathbf{C}}(\mathbf{0})$ has
                          \[
                          \dist_H(\mathbf{C}_{i_0} \cap \partial B_1,\mathbf{C}' \cap \partial B_1) \leq \gamma
                              \]
                              for some $i_0 \in \set{1, \cdots, N}$. Thus, we can define
                              \[
                              \mathrm{CC}_{\tilde w} := \set{\set{\mathbf{C}':\theta_{\mathbf{C}'}(\mathbf{0}) = \theta_{\mathbf{C}}(\mathbf{0}) ,\ \dist_H(\mathbf{C}_{i} \cap \partial B_1,\mathbf{C}' \cap \partial B_1) \leq \gamma}: i = 1, \cdots, N}\,.
                              \]
                              
                              To cover $[0, \infty)$ related to $\rho$, we first define a countable covering,
                              \[
                                  \mathrm{IC}_{\tilde w} := \set{0} \cup \set{\mathrm{Int}_{k} := [(1 + \frac{\gamma}{2})^k, (1 + \frac{\gamma}{2})^{k+1}]}_{k \in \mathbb{Z}}\,.
                              \]
                              
                              For the singleton $\set{0}$, we define a countable covering of $(0, \infty)$ related to $R$ as
                              \[
                                  \mathrm{IC}_{\tilde w, \set{0}} := \set{\mathrm{Int}_{\set{0}, k'} := [(1 + \frac{\gamma}{2})^{k'}, (1 + \frac{\gamma}{2})^{k'+1}]}_{k' \in \mathbb{Z}}
                              \]
                              and for each fixed $k'$, we consider a finite covering $\mathrm{BC}_{\tilde u, \set{0}, k'}$ of $(M, g)$ by geodesic balls of radius no greater than 
                              \[
                                  \min(\mathrm{injrad}(M, g), (1 + \frac{\gamma}{2})^{k'} \gamma/10).
                              \]
                              
                              For each $k$, we define
                              \[
                                  \mathrm{IC}_{\tilde w, k} := \set{\mathrm{Int}_{k, k'} := [(1 + \frac{\gamma}{2}\min(1, (1 + \frac{\gamma}{2})^{k}))^{k'}, (1 + \frac{\gamma}{2}\min(1, (1 + \frac{\gamma}{2})^{k}))^{k'+1}]}_{k' \in \mathbb{Z}}
                              \]
                              and in this case, for each fixed $k'$, similarly, we consider a finite covering $\mathrm{BC}_{\tilde u, k, k'}$ of $(M, g)$ by geodesic balls of radius no greater than 
                              \[
                                  \min\left(\mathrm{injrad}(M, g), (1 + \frac{\gamma}{2}\min(1, (1 + \frac{\gamma}{2})^{k}))^{k'} \gamma\min(1, (1 + \frac{\gamma}{2})^{k})/10\right)\,.
                              \]
                              
                              Therefore, for the product space $\set{\mathbf{C}':\theta_{\mathbf{C}'}(\mathbf{0}) = \theta_{\mathbf{C}}(\mathbf{0})} \times M \times (0, \infty) \times [0, \infty)$, we obtain a countable covering
                              \[
                                  \begin{aligned}
                                      \mathrm{PC}_{\tilde w} := \mathrm{CC}_{\tilde w} \otimes \Big\{ & \set{(B, \mathrm{Int}_{k, k'}, \mathrm{Int}_k): k, k' \in \mathbb{Z}, B \in \mathrm{BC}_{\tilde w, k, k'}} \cup       \\
                                                                                                      & \set{(B, \mathrm{Int}_{\set{0}, k'}, \set{0}): k' \in \mathbb{Z}, B \in \mathrm{BC}_{\tilde w, \set{0}, k'}}\Big\}\,.
                                  \end{aligned}
                              \]
                              
                              As above, note that if $(\mathbf{C}_{\hat\beta}, x_{\hat\beta}, R_{\hat\beta}, \rho_{\hat\beta})$ and $(\mathbf{C}_{{\hat\beta}'}, x_{{\hat\beta}'}, R_{{\hat\beta}'}, \rho_{{\hat\beta}'})$ are in the same element of $\mathrm{PC}_{\tilde w}$, then $(\mathbf{C}_{\hat\beta}, m, x_{\hat\beta}, R_{\hat\beta}, \rho_{\hat\beta})$ and $(\mathbf{C}_{{\hat\beta}'}, x_{{\hat\beta}'}, R_{{\hat\beta}'}, \rho_{\hat\beta})$ satisfy the inequalities in the first part of the definition of $\gamma$-closeness.
                \end{enumerate}
                
                With the construction above, we see that the tree representations of all pairs in $\cL'(\tilde T)$ are contained in the countable union
                \begin{equation}
                    \mathrm{PC}(\tilde T) := \bigotimes_{\tilde u} \mathrm{PC}_{\tilde u}\otimes \bigotimes_{\tilde w} \mathrm{PC}_{\tilde w}\,.
                \end{equation}
                
                Hence, we can find a countable subsets $\cL''(\tilde T)$ of $\cL'(\tilde T)$ satisfying the following property. For each element of $PC(\tilde T)$ containing at least one tree representation of pairs in $\cL'(\tilde T)$, it contains exactly one tree representation of a pair in $\cL''(\tilde T)$.
                
                In summary, it follows from the finiteness of coarse tree representations and the construction above that the set of pairs
                \begin{equation}
                    \bigcup_{\tilde T: \text{ coarse tree representation}}\cL''(\tilde T)
                \end{equation}
                is countable and satisfies the desired property.
            \end{proof}

            For any $\Lambda > 0$ and any integer $I > 0$, let
            \[
                \cM^{k, \alpha}(M; \Lambda, I) := \Big\{ (g, \Sigma) \in \cM^{k,\alpha}(M): \|g\|_{C^{k, \alpha}} \leq \Lambda; \scH^7(\Sigma', g') \leq \Lambda,\ \ind(\Sigma', g') \leq I\Big\} \,,
            \]
            and
            \[
                \begin{split}
                    \cM^{k, \alpha}_{>1}(M; \Lambda, I, \delta) := \Big\{ (g', \Sigma') \in \cM^{k,\alpha}(M): \|g'\|_{C^{k, \alpha}} \leq \Lambda; \scH^7(\Sigma', g') \leq \Lambda,\ \ind(\Sigma', g') \leq I\,,\\
                    \|g - g'\|_{C^{k - 1, \alpha}} + \mathbf{F}(m|\Sigma|, |\Sigma'|) < \delta \text{ for some }(g, \Sigma) \in \cM^{k, \alpha}(M; \Lambda, I), m \in \mathbb{N}^{>1}\,,\Big\} \,.
                \end{split}
            \]
            It follows from Sharp's compactness theorem \cite{sharp_compactness_2017} that for any $\delta > 0$, $ \cM^{k, \alpha}(M; \Lambda, I) \setminus \cM^{k, \alpha}_{>1}(M; \Lambda, I, \delta)$ is a compact set, and moreover,
            \[
                \cM^{k, \alpha}(M; \Lambda, I) = \bigcup^\infty_{j = 1}\left(\cM^{k, \alpha}(M; \Lambda, I) \setminus \cM^{k, \alpha}_{>1}(M; \Lambda, I, \frac{1}{j})\right)\,.
            \]
            
            On the one hand, for any fixed $\Lambda, I$ and $j$, by the compactness property, there exist a finite collection of metrics on $M$, $\set{g_i}$ and a finite collection of $g_i$-minimal hypersurfaces $\set{\Sigma_i}$, such that
            \begin{equation}
                \cM^{k, \alpha}(M; \Lambda, I) \setminus \cM^{k, \alpha}_{>1}(M; \Lambda, I, \frac{1}{j}) \subset \bigcup_{i} \cM^{k,\alpha}(g_i, \Sigma_i; \Lambda, I, \beta)\,.
            \end{equation}
            In other words, we have
            \[
                \begin{split}
                    \cM^{k,\alpha}(M) &= \bigcup^\infty_{I = 0} \bigcup^\infty_{\Lambda = 1} \cM^{k, \alpha}(M; \Lambda, I)\\
                        &= \bigcup^\infty_{I = 0} \bigcup^\infty_{\Lambda = 1}\bigcup^\infty_{j = 1} \cM^{k, \alpha}(M; \Lambda, I) \setminus \cM^{k, \alpha}_{>1}(M; \Lambda, I, \frac{1}{j})\\
                        &= \bigcup^\infty_{I = 0} \bigcup^\infty_{\Lambda = 1}\bigcup^\infty_{j = 1}\bigcup_{i} \cM^{k,\alpha}(g_i, \Sigma_i; \Lambda, I, \beta)\,.
                \end{split}
            \]
            By relabeling all the indices, we obtain a countable covering
            \begin{equation}
                \cM^{k,\alpha}(M) = \bigcup^\infty_{i = 1} \cM^{k,\alpha}(g_i, \Sigma_i; \Lambda_i, I_i, \beta)\,.
            \end{equation}
            
            On the other hand, by Remark \ref{Rem:Fst_Decomp},
            \[
                \cM^{k, \alpha}(g_i, \Sigma_i; \Lambda_i, I_i, \beta) = \bigcup^\infty_{v=1} \cL^{k, \alpha}_0(g_{i,v}, \Sigma_{i,v}; \Lambda_i, I_i ,\beta)\,,
            \]
            By relabeling the indices $\set{i, v}$, we have
            \begin{equation}\label{Eqn:Snd_Decomp}
                \cM^{k,\alpha}(M) = \bigcup^\infty_{i = 1} \cL^{k, \alpha}_0(g_i, \Sigma_i; \Lambda_i, I_i ,\beta)\,.
            \end{equation}
        
        \subsection{Compactness and second decomposition}
            
            \begin{Prop}[Sequential compactness of $\cL^{k, \alpha}_0$]\label{prop:comp_of_cL0}
               The intermediate ``canonical neighborhood'' $\cL^{k, \alpha}_0(g_v, \Sigma_v; \Lambda, I ,\beta)$ is sequentially compact. In other words, given a sequence $\set{(g_j, \Sigma_j)} \subset \cL^{k, \alpha}_0(g_v, \Sigma_v; \Lambda, I ,\beta)$, there exists a pair $(g_\infty, \Sigma_\infty) \in \cL^{k, \alpha}_0(g_v, \Sigma_v; \Lambda, I ,\beta)$, such that up to a subsequence, $g_j \rightarrow g_\infty$ in $C^{k-1, \alpha}$, and $\Sigma_j \rightarrow \Sigma_\infty$ in the varifold sense.
            \end{Prop}
            \begin{proof}
                By definition, we have
                \begin{itemize}
                    \item $\|g_j\|_{C^{k, \alpha}} \leq \Lambda$;
                    \item $\|g_j - g_v\|_{C^{k-1, \alpha}} + \mathbf{F}(\Sigma_j, \Sigma_v) \leq \delta$;
                    \item $\scH^7(\Sigma_j, g_j) \leq \Lambda,\ \ind(\Sigma_j, g_j) \leq I$.
                \end{itemize}
                The compactness of $C^{k, \alpha} \hookrightarrow C^{k-1, \alpha}$ and Sharp's compactness \cite{sharp_compactness_2017} together imply that there exists a pair $(g_\infty, \Sigma_\infty)$ such that
                \begin{itemize}
                    \item $\|g_\infty\|_{C^{k, \alpha}} \leq \Lambda$ and $\Sigma_\infty$ is a LSMH;
                    \item Up to a subsequence, $g_j \rightarrow g_\infty$ in $C^{k-1, \alpha}$, and $\Sigma_j \rightarrow \Sigma_\infty$ in the varifold sense;
                    \item $\|g_\infty - g_v\|_{C^{k-1, \alpha}} + \mathbf{F}(\Sigma_\infty, \Sigma_v) \leq \delta$;
                    \item $\scH^7(\Sigma_\infty, g_\infty) \leq \Lambda,\ \ind(\Sigma_\infty, g_\infty) \leq I$.
                \end{itemize}
                Now, it suffices to show that $(g_\infty, \Sigma_\infty)$ has a tree representation $\beta$-close to the one of $(g_v, \Sigma_v)$. 
                
                Indeed, by large-scale cone decomposition and the constancy theorem, we know that $\Sigma_\infty$ should be a multiplicity one minimal hypersurface. Moreover, the $\beta$-graphic property indicates that the convergence are all smooth and thus, the coarse tree representation will be the same. Since all the closeness property will be preserved along the limiting process, $(g_\infty, \Sigma_\infty)$ has a tree representation $\beta$-close to the one of $(g_v, \Sigma_v)$.
            \end{proof}
            
            Compactness is equivalent to the existence of a finite subcover of its open covers. Hence,  a finite decomposition of intermediate canonical neighborhoods as follows.
            
            \begin{Prop}[Finite covering of $\cL^{k,\alpha}_0$]
                For any $g, \Sigma, \Lambda, I, \beta$ as in (\ref{Eqn:Snd_Decomp}) (with subscripts omitted), there exists a countable set of pairs $\set{(g_{v}, \Sigma_{v})}_v \subset \cL^{k, \alpha}_0(g, \Sigma; \Lambda, I ,\beta)$ with corresponding $\set{\epsilon_v}_v$, such that
                \[
                    \cL^{k, \alpha}_0(g, \Sigma; \Lambda, I ,\beta) \subset \bigcup^{V}_{v = 1} \cL^{k,\alpha}(g_{v}, \Sigma_{v};\Lambda', \kappa_{v})\,.
                \]
                where $\Lambda' = \Lambda + I + 1$, $V \in \mathbb{N}$ and $\kappa_v = \kappa(g_{v}, \Sigma_{v}; \Lambda')$ defined in Theorem \ref{Thm_Countable Decomp}.
            \end{Prop}
            \begin{proof}
                By Proposition \ref{prop:comp_of_cL0}, we know that $\cL^{k, \alpha}_0(g_i, \Sigma_i; \Lambda_i, I ,\beta)$ is compact, so it suffices to show that for any $(g', \Sigma') \in \cL^{k, \alpha}_0(g, \Sigma; \Lambda, I ,\beta)$ and $\delta' > 0$, the corresponding canonical neighborhood $\cL^{k,\alpha}(g', \Sigma';\Lambda', \delta')$ contains an open neighborhood of $(g', \Sigma')$ in $\cL^{k, \alpha}_0(g, \Sigma; \Lambda, I ,\beta)$.
                
                Suppose otherwise that there exists some $(g', \Sigma')$ and a sequence of $\{(g_i, \Sigma_i)\} \subset \cL^{k, \alpha}_0(g, \Sigma; \Lambda, I ,\beta) \setminus \cL^{k,\alpha}(g', \Sigma';\Lambda', \delta')$ so that
                \[
                    (g_i, \Sigma_i) \rightarrow (g', \Sigma')\,.
                \]
                
                For large enough $i$, we can verify that
                \begin{itemize}
                    \item $g_i$ is a $C^{k, \alpha}$ metric on $M$ with $\|g_i\|_{C^{k, \alpha}}\leq \Lambda'$ and $\|g'-g_i\|_{C^{k-1, \alpha}}\leq \delta'$;
                    \item $\Sigma_i$ is a LSMH in $(M, g_i)$ satisfying
                          \[ \mbfF(|\Sigma_i|_{g_i}, |\Sigma'|_{g'})\leq \delta', \quad \ind(\Sigma_i)\leq \Lambda'\,; \]
                \end{itemize}
                Furthermore, the last condition is implied by the $\beta$-closeness of tree representations where $\beta < 1/100$.
                
                In other words, for $i$ large enough, $(g_i, \Sigma_i) \in \cL^{k,\alpha}(g', \Sigma';\Lambda, I, \epsilon')$ which gives a contradiction.
            \end{proof}
            
            \begin{proof}[Proof of Theorem \ref{Thm_Countable Decomp}]
                By (\ref{Eqn:Snd_Decomp}) and the previous covering proposition, we have
                \begin{align*}
                    \cM^{k,\alpha}(M) & = \bigcup^\infty_{i = 1} \cL^{k, \alpha}_0(g_i, \Sigma_i; \Lambda_i, I_i ,\beta)                                    \\
                                      & = \bigcup^\infty_{i = 1} \bigcup^{V_i}_{v = 1} \cL^{k,\alpha}(g_{i, v}, \Sigma_{i, v};\Lambda'_i, \kappa_{i, v})\,,
                \end{align*}
                where $\Lambda'_i = \Lambda_i + I_i + 1$, $V_i \in \mathbb{N}$ and $\kappa_{i, v} = \kappa(g_{i, v}, \Sigma_{i, v}; \Lambda'_i)$.
                
                Hence, Theorem \ref{Thm_Countable Decomp} follows from relabeling the indices $\set{i, v}$.
            \end{proof}
            
\appendix
    \section{Proof of Growth Rate Monotonicity} \label{Sec_App_Proof Growth Rate Mon}
        The goal for this section is to prove Lemma \ref{Lem_Ana on SMC_Growth Rate Monoton}.  Let $u\in C^2_{loc}A(K^{-3}, 1)) \cap L^2(A(K^{-3}, 1))$ be a given solution to $L_\mbfC u = 0$ as in Lemma \ref{Lem_Ana on SMC_Growth Rate Monoton}, where $K>2$ will be determined later.  By (\ref{Equ_Pre_Decomp Jac into homogen Jac field}), we can write 
        \[
            u(r, \omega) = \sum_{j\geq 1}(v_j^+(r)+v_j^-(r))\cdot \varphi_j(\omega)\,, 
        \]
        and by (\ref{Equ_Pre_Homogen Jac field}), (\ref{Equ_Pre_Asymp Spectrum gamma_j^pm}) and the definition of $J^\gamma_K(u; r)$, we have
        \begin{align*}
            J^\gamma_K(u; r) & = \int_{K^{-1}r}^r t^{n-1-n-2\gamma}\ dt\int_{\mbfC\cap \SSp^n} u(t, \omega)^2\ d\omega \\
                             & = \sum_{j\geq 1}\int_{K^{-1}r}^r t^{-1-2\gamma}\cdot (v_j^+(t)+v_j^-(t))^2\ dt,
        \end{align*}
        where 
        \begin{align*}
              & \ \int_{K^{-1}r}^r t^{-1-2\gamma}\cdot (v_j^+(t)+v_j^-(t))^2\ dt \\
            = & \ \begin{cases}
                \int_{K^{-1}r}^r (c_j^+ t^{\gamma_j^+ -\gamma} + c_j^- t^{\gamma_j^- - \gamma})^2 t^{-1}\ dt,   & \ \text{ if }\gamma_j^+ \neq \gamma_j^-;  \\
                \int_{K^{-1}r}^r (c_j^+ t^{-(n-2)/2 -\gamma} + c_j^- t^{-(n-2)/2 - \gamma}\log t)^2 t^{-1}\ dt, & \ \text{ if }\gamma_j^+ = \gamma_j^-.    
            \end{cases}
        \end{align*}
        Hence Lemma \ref{Lem_Ana on SMC_Growth Rate Monoton} is an immediate corollary of the following Lemma \ref{Lem_App_Growth rate mon alpha-beta} and \ref{Lem_App_Growth rate mon alpha-log}.
        \begin{Lem} \label{Lem_App_Growth rate mon alpha-beta}
            Suppose $\sigma>0$. Then there exists a real number $K_1 = K_1(\sigma)>2$ with the following property.  For every $\alpha\neq \beta\in \RR$ such that $|\alpha|, |\beta|, |\alpha+\beta| \geq \sigma$ and every $K\geq K_1$, the following \[
                \cI_K(r; c, c') := \int_r^{Kr}(cs^\alpha + c's^\beta)^2s^{-1}\ ds,   \]
            satisfies that for every $r>0$ and every $(0,0)\neq (c, c')\in \RR^2$ we have, \[
                \cI_K(K^2r; c, c') - 2\cI_K(Kr; c, c') + \cI_K(r; c, c') > 0.   \]
        \end{Lem}
        \begin{proof}
            \begin{align*}
                \cI_K(r; c, c') & = \int_r^{Kr} (c^2 s^{2\alpha-1} + 2cc' s^{\alpha+\beta -1} + c'^2 s^{2\beta -1})\ ds                                                                                        \\
                                & = c^2 r^{2\alpha}\cdot \frac{K^{2\alpha}-1}{2\alpha} + 2cc' r^{\alpha+\beta}\cdot\frac{K^{\alpha+\beta}-1}{\alpha+\beta} + c'^2 r^{2\beta}\cdot \frac{K^{2\beta}-1}{2\beta}.
            \end{align*}
            Thus,
            \begin{align*}
                  & \ \cI_K(K^2r; c, c') - 2\cI_K(Kr; c, c') + \cI_K(r; c, c')                                                                                                                                   \\ 
                = & \  c^2 r^{2\alpha}\cdot \frac{(K^{2\alpha}-1)^3}{2\alpha} + 2cc' r^{\alpha+\beta}\cdot\frac{(K^{\alpha+\beta}-1)^3}{\alpha+\beta} + c'^2 r^{2\beta}\cdot \frac{(K^{2\beta}-1)^3}{2\beta}.
            \end{align*}
            And hence to prove the lemma, it suffices to show that when $K\geq K_1(\sigma)$, we have 
            \begin{align*}
                \Delta(K; \alpha, \beta):= \Big[ \frac{(K^{\alpha + \beta}-1)^3}{\alpha + \beta} \Big]^2 - \frac{(K^{2\alpha} - 1)^3}{2\alpha}\cdot \frac{(K^{2\beta} - 1)^3}{2\beta} < 0.
            \end{align*}
            Since $\Delta(K; \alpha, \beta) = \Delta(K; \beta, \alpha) = K^{6\alpha+6\beta}\cdot \Delta(K; -\alpha,-\beta)$, we can assume WLOG that $\alpha>\beta$ and $\alpha+\beta\geq\sigma>0$. \\

            \noindent \underline{\textit{Case I.}} $0<\beta<\alpha$.  By a simple calculation, we have
            \begin{align*}
                \Delta(K; \alpha, \beta) = \frac{-(\alpha -\beta)^2}{4\alpha\beta(\alpha+\beta)^2} & \Big[ K^{6\alpha+6\beta} - K^{5\alpha+5\beta}(6+3\Lambda_1)                                                       \\
                                                                                                   & +  K^{4\alpha+4\beta}(15 + 3\Lambda_2) - K^{3\alpha+3\beta}( 20 + \Lambda_3+9\Lambda_1) \\
                                                                                                   & +  K^{2\alpha+2\beta}(15 + 3\Lambda_2) - K^{\alpha+\beta}(6+3\Lambda_1) + 1  \Big], 
            \end{align*}
            where for $i=1,2,3$, 
            \begin{align*}
                0 \leq \Lambda_i & := \frac{(\alpha+\beta)^2}{(\alpha - \beta)^2}\cdot (K^{i\alpha-i\beta}-2 + K^{i\beta -i\alpha}) \\
                & \leq \begin{cases} 
                  (3\beta)^2 K^{i\alpha-i\beta}\cdot (i\log K)^2,&\ \text{ if }\alpha\leq 2\beta; \\
                  3^2K^{i\alpha -i\beta},&\ \text{ if }\alpha>2\beta.
                \end{cases}\\
                & \leq 9(1+\beta^2) K^{i\alpha - i\beta}\cdot(i\log K)^2.             
            \end{align*}
            Hence by taking $K\geq K_1(\sigma)>>1$, $K^{6\alpha + 6\beta}$ dominates other terms in $\Delta(K; \alpha, \beta)$. In particular, we have $\Delta(K; \alpha, \beta)<0$.\\

            \noindent \underline{\textit{Case II.}} $\beta<0<\alpha$.  Then when $K^{2\sigma}>2$, we have $K^{2\alpha} -1 >K^{2\alpha} /2$, $1-K^{2\beta} > 1/2$, and hence,
            \begin{align}
                \Delta(K; \alpha, \beta) = \Big[ \frac{(K^{\alpha + \beta}-1)^3}{\alpha + \beta} \Big]^2 - \frac{(K^{2\alpha} - 1)^3}{2\alpha}\cdot \frac{(K^{2\beta} - 1)^3}{2\beta} \leq \frac{K^{6\alpha+6\beta}}{(\alpha +\beta)^2} - \frac{K^{6\alpha}}{256|\alpha\beta|}.  \label{Equ_App_beta<0<alpha, est Delta}
            \end{align}
            And since 
            \begin{align*}
                \frac{(\alpha+\beta)^2}{|\alpha\beta|} \geq  \begin{cases}
                    \frac{\sigma}{2|\beta|},     & \ \text{ if }\alpha+\beta>\alpha/2;     \\
                    \frac{\sigma^2}{2|\beta|^2}, & \ \text{ if }\alpha+\beta\leq \alpha/2.
                \end{cases}
            \end{align*}
            (\ref{Equ_App_beta<0<alpha, est Delta}) then implies $\Delta(K; \alpha, \beta)<0$ by taking $K\geq K_1(\sigma)>>1$.
        \end{proof}
        
        \begin{Lem} \label{Lem_App_Growth rate mon alpha-log}
            Suppose $\sigma>0$. Then there exists a real number $K_2 = K_2(\sigma)>2$ with the following property.  For every $\alpha'\in \RR$ such that $|\alpha'|\geq \sigma$ and every $K\geq K_2$, the following 
            \[
                \tilde{\cI}_K(r; c, c') := \int_r^{Kr}(cs^{\alpha'} + c's^{\alpha'} \log s)^2s^{-1}\ ds,
            \]
            satisfies that for every $r>0$ and every $(0,0)\neq (c, c')\in \RR^2$ we have, 
            \[
                \tilde{\cI}_K(K^2r; c, c') - 2\tilde{\cI}_K(Kr; c, c') + \tilde{\cI}_K(r; c, c') > 0.
            \]
        \end{Lem}

        \begin{proof}
            First note that the Lemma holds trivially if $c' = 0$. By the arbitrariness of $r>0$, we can assume WLOG that $(c, c') = (0, 1)$.  Also for simplicity of notations, let $\alpha:= 2\alpha'$. Then, 
            \begin{align*}
                \tilde{\cI}_K(r; 0, 1) & = \int_r^{Kr} s^{\alpha-1}(\log s)^2\ ds                                                                                                               \\
                                       & =\ r^{\alpha}(\log r)^2\cdot \frac{K^{\alpha}-1}{\alpha} + r^{\alpha}\log r\cdot (\frac{2K^{\alpha}\log K}{\alpha} - \frac{2(K^{\alpha}-1)}{\alpha^2}) \\
                                       & \ + r^{\alpha}\cdot \big(\frac{K^{\alpha}(\log K)^2}{\alpha} - \frac{2K^{\alpha}\log K}{\alpha^2} +\frac{2(K^{\alpha}-1)}{\alpha^3} \big). 
            \end{align*}
            A direct calculation provides,
            \begin{align*}
                    & \ \tilde{\cI}_K(K^2r; 0, 1) - 2\tilde{\cI}_K(Kr; 0, 1) + \tilde{\cI}_K(r; 0, 1)                                                                                                           \\
                =\  & \ r^\alpha(\log r)^2 \cdot \frac{(K^\alpha -1)^3}{\alpha}                                                                                                                                 \\
                +   & \ 2r^\alpha(\log r) \cdot \Big[ 3K^\alpha\log K\cdot \frac{(K^\alpha-1)^2}{\alpha} - \frac{(K^\alpha-1)^3}{\alpha^2} \Big]                                                                \\
                +   & \ r^\alpha\cdot \Big[  K^\alpha(\log K)^2 \cdot\frac{(K^\alpha-1)(9K^\alpha-3)}{\alpha} - K^\alpha\log K\cdot \frac{6(K^\alpha-1)^2}{\alpha^2} + \frac{2(K^\alpha-1)^3}{\alpha^3}  \Big].
            \end{align*}
            Hence to prove the Lemma, it suffices to show that 
            \begin{align*}
                \Delta(K; \alpha)  := & \ \Big[ 3K^\alpha\log K\cdot \frac{(K^\alpha-1)^2}{\alpha} - \frac{(K^\alpha-1)^3}{\alpha^2} \Big]^2 - \frac{(K^\alpha -1)^3}{\alpha}\cdot                                 \\
                                      & \ \Big[  K^\alpha(\log K)^2 \cdot\frac{(K^\alpha-1)(9K^\alpha-3)}{\alpha} - K^\alpha\log K\cdot \frac{6(K^\alpha-1)^2}{\alpha^2} + \frac{2(K^\alpha-1)^3}{\alpha^3}  \Big] \\
                =                     & \ \frac{(K^\alpha - 1)^4}{\alpha^2}\cdot\big[ 3 K^\alpha(\log K)^2 -\frac{(K^\alpha-1)^2}{\alpha^2} \big].
            \end{align*}
            is less than $0$ provided $K\geq K_2(\sigma)>>1$. This is clear since when $\alpha>0$ (and then by assumption, $\alpha\geq 2\sigma$), $-K^{2\alpha}/\alpha^2$ dominates the other terms, and when $\alpha<0$ (and again by assumption, $\alpha\leq -2\sigma$), $-1/\alpha^2$ dominates the other terms.
        \end{proof}

    \section{Geometry of Minimal Graphs} \label{Sec_App_Geom of Minimal Graph}
        Let $(M^{n+1}, g)$ be a Riemannian manifold without boundary (not necessarily complete) and $\Sigma\subset M$ be an embedded two-sided $C^3$ hypersurface (not necessarily proper) with a unit normal field $\nu$.
        Recall that for every $x\in \Sigma$, we have introduced the \textbf{regularity scale} $r_\cS (x):=r_\cS(x; M, g, \Sigma)$ of $\Sigma$ at $x$ in Definition \ref{Def_Reg Scale}. With aid of this, we defined the $C^k_*$-norms of functions over $\Sigma$, which is invariant under rescalings.

        For $f\in C^k(M)$ ($0\leq k\leq 3$ integer), $\beta\in \Sym(T^*M\otimes T^*M)$, and for every $x\in \Sigma$, we also define 
        \begin{align*}
        [f]_{x, g, C^k_*} := \sum_{j=0}^k r_\cS(x)^j\sup_{B^g_{r_\cS(x)}}|\nabla_g^j f|;&\ \ & 
        [\beta]_{x, g, C^k_*} := \sum_{j=0}^k r_\cS(x)^j\sup_{B^g_{r_\cS(x)}}|\nabla_g^j \beta|\,   
        \end{align*}
        This is also a norm invariant under rescaling, more precisely, for every $\lambda>0$, we have 
        \begin{align*}
          [f]_{x, \lambda^2 g, C^k_*} = [f]_{x, g, C^k_*};&\ & [\lambda^2\beta]_{x,\lambda^2 g, C^k_*} = [\beta]_{x, g, C^k_*}\,.       
        \end{align*}
        We shall omit the subscript $g$ if there's no ambiguity.
        
        Throughout this section, suppose $(M, g, \Sigma)\neq (\RR^{n+1}, g_{\eucl}, \RR^n)$, and thus $r_\cS <+\infty$ on $\Sigma$.     
        \begin{Thm} \label{Thm_Append_MSE for graphs}
            There exists $\delta=\delta(n)\in (0, 1)$ and $C=C(n)>>1$ with the following properties.
            \begin{enumerate} [(i)]
                \item If $u\in C^1(\Sigma)$ with $\|u\|_{C^1_*}\leq \delta$, then
                      \[
                          \Phi^u: \Sigma \to M,\ \ \ x\mapsto \exp^g_x(u(x)\cdot \nu(x)),
                      \]
                      is a $C^1$ embedding;
                \item There exists a $C^1$ area density function $F^f = F^f(x, z, \xi)$, where $x\in \Sigma$, $z\in \RR$ with $r_\cS^{-1}|z|<1$, $\xi\in T^*_x\Sigma$ with $|\xi|_g <1$, with pointwise estimate \[
                          |F^f(x, z, \xi) - 1| \leq C(n)(r_\cS(x)^{-1}|z| + |\xi| + [f]_{x, C^2_*});
                      \]
                      And such that for every $u\in C^2(\Sigma)$ with $\|u\|_{C^2_*}\leq \delta$, every $f\in C^2(M)$ with $[f]_{x, C^2_*}\leq \delta$, $\forall x\in \Sigma$, and every $\varphi\in C_c^0(M\setminus \Sing(\Sigma))$, we have 
                      \[
                          \int_M \varphi(x)\ d\|\graph_{\Sigma}(u)\|_{(1+f)g}(x) = \int_M \varphi\circ \Phi^u(x)\cdot F^f(x, u(x), du(x))\ d\|\Sigma\|_g.
                      \]
                \item Let $F^f$ be in (ii); let $\scM^f: C^2_*(\Sigma)\to C^0_{loc}(\Sigma)$ be the minimal surface operator ($\scM^f$ is only defined in the $\delta$-neighborhood of $\mathbf{0}$), in other words, for every $\varphi\in C_c^1(\Sigma)$, \[
                          \int_\Sigma \scM^f (u)\cdot \varphi\ d\|\Sigma\|_g := \frac{d}{dt}\Big|_{t=0} \int_\Sigma F^f\left(x, u+t\varphi, d(u+t\varphi)\right)\ d\|\Sigma\|_g.   \]
                      Then for every pair $f^\pm\in C^2(M)$ with $[f^\pm]_{x, C^2_*}\leq \delta$ along $x\in \Sigma$ and every pair $\|u^\pm\|_{C^2_*}\leq \delta$, we have
                      \begin{align*}
                          \scM^{f^+} (u^+) - \scM^{f^-} (u^-) = & -L_{\Sigma, g}(u^+ - u^-) + \frac{n}{2}\nu(f^+ - f^-) \\
                                                                & + div_{\Sigma, g}(\cE_1) + r_\cS^{-1}\cE_2,     
                      \end{align*}
                      where $\cE_1, \cE_2$ are functions on $\Sigma$ satisfying the pointwise bound along $x\in \Sigma$,
                      \begin{align*}
                          |\cE_1(x)| + |\cE_2(x)| & \leq C(n)\left(\sum_{i=\pm} [f^i]_{x, C^2_*} + r_\cS(x)^{-1}|u^i|(x) + |du^i(x)| \right)         \\
                                                  & \ \cdot \left([f^+ -f^-]_{x, C^2_*} + r_{\cS}(x)^{-1}|u^+ -u^-|(x) + |d(u^+ - u^-)|(x) \right). 
                      \end{align*}
                \item Let $u\in C^2(\Sigma)$ such that $\|u\|_{C^2_*}\leq \delta$, $\tilde{g}$ be a metric on $M$ such that $[\tilde{g}-g]_{x, C^3_*}\leq \delta$ for every $x\in \Sigma$. Denote for simplicity $\Sigma_u= \graph_{\Sigma, g}(u)$.
                Then we have pointwise estimate,
                \begin{align*}
                \left| |A_{\Sigma_u, \tilde{g}}|_{\tilde{g}}^2\circ \Phi^u - |A_\Sigma|_g^2 \right| \leq C(n)\left([\tilde{g}-g]_{x, C^3_*}+ \sum_{j=0}^2 r_\cS^{j-1}|\nabla^j_{\Sigma, g}u|\right)\cdot r_\cS^{-2}\,;
                \end{align*}
                And for every $\psi\in C^2_{loc}(\Sigma_u)$, we have \[
                (\Delta_{\Sigma_u}\psi)\circ\Phi^u - \Delta_\Sigma(\psi\circ \Phi^u) = \diverg_\Sigma(\vec{B}_0) + r_\cS^{-1}\cdot B_1\,,  \] 
                with pointwise estimate on error terms,
                \begin{align*}
                |\vec{B}_0| + |B_1| \leq C(n)\left([\tilde{g}-g]_{x, C^3_*} + \sum_{j=0}^2 r_\cS^{j-1}|\nabla_{\Sigma, g}^j u|\right)\cdot |d\psi|_g\,.
                \end{align*}
            \end{enumerate}
        \end{Thm}
        \begin{proof}
        First one can easily check that the Theorem is formulated invariant under rescalings. Hence (i) follows immediately by renormalizing the metric with $r_\cS^{-2}$ near each point. To prove (ii)-(iv), by taking the Fermi coordinates, it suffices to work under the following further assumptions.
        \begin{enumerate}[(a)]
        \item $M \supset M_0:= \BB_1^n \times (-1, 1)$ with global coordinates $\{x^1, x^2, \dots, x^n, x^{n+1}=z\}$ and metric $g|_{M_0} = g^z + dz^2$, where $g^z$ is a family of metric on $\BB_1^n$ parametrized by $z\in (0, 1)$ satisfying that,
              \begin{itemize}
              \item $\|g-g_{\eucl}\|_{C^3, M} < 1/10$;
              \item $g^0_{ij}=\delta_{ij}$, $\nabla^{g^0}_{\partial_i}\partial_j = 0$ both at $\mathbf{0}\in \BB_1^n$, $\forall 1\leq i, j\leq n$.
              \end{itemize}
        \item $\Sigma_0:= \Sigma \cap M_0 = \BB_1^n\times \{0\} \subset M_0$, and thus $\Phi^u(x) = (x, u(x))$ on $\Sigma_0$;
        \item $1\leq r_\cS \leq C_n$ on $\Sigma_0$;
        \item When estimating geometric quantities of $\graph_\Sigma(u)$, we only need to work at $x = \mathbf{0}$ and assuming $\|u\|_{C^2, \BB_1} + \|f\|_{C^2, M_0}\leq \delta_n$.
        \end{enumerate}
        First note that, the smooth family of metric $\{g^z\}_{z\in (0, 1)}$ satisfies the ODE
        \begin{align}
        \begin{cases}
         \partial_z g_{ij}^z = -2A^z_{ij};\\
         \partial_z A_{ij}^z = -A^z_{ik}A^z_{jl}g_z^{kl} + R^z_{ij},
        \end{cases} \label{Equ_App MinGraph_ODE of g^z}
        \end{align}
        where $[g_z^{ij}]$ denotes the inverse matrix of $[g^z_{ij}]$; $R^z_{ij}:= R_g(\partial_i, \partial_z, \partial_z, \partial_j)|_{(x, z)}$; $A^0_{ij} = A_\Sigma(\partial_i,\partial_j)$.

        To prove (ii), note that 
        \begin{align*}
         \partial_i \Phi^u(x) & = (\partial_i + u_i(x)\partial_z)|_{(x, u(x))}, \ \ \ \forall 1\leq i\leq n\,; \\
         \nu^{u,f}(x, u(x)) & = \frac{1}{\sqrt{(1+f)(1+g_u^{ij}u_iu_j)}}\cdot (\partial_z - g_u^ij u_i(x)\partial_j)\,.
        \end{align*}
        where $\nu^{u, f}$ is the upward unit normal field of $\Phi^u(\Sigma_0)$  under $(1+f)g$, $u_j:= \partial_ju$. Hence the area element of $\Phi^u(\Sigma_0)$ at $(x, u(x))$ under metric $(1+f)g$ is given by $F^f(x, u(x), du(x))$, where \[
          F^f(x, z, \xi) := \sqrt{(1+f(x,z))^n\cdot \det{}_{g^0}\left[g^z+\xi\otimes \xi \right](x)}.   \]
        Together with (\ref{Equ_App MinGraph_ODE of g^z}), this proves the estimate in (ii).
        
        To prove (iii), notice that 
        \begin{align}
        \scM^f(u)(x) = -\diverg \partial_\xi F^f(x, u(x), du(x)) + \partial_zF^f(x, u(x), du(x))\,,  \label{Equ_App MinGraph_Min Surf Oper}     
        \end{align}
        and 
        \begin{align*}
        \partial_\xi F^f(x, z, \xi) & = F^f(x, z, \xi)\cdot g^0(x)G(x,z,\xi)^{-1}\xi,\\
        \partial_z F^f(x, z, \xi) & = F^f(x, z, \xi)\cdot \left(\frac{n\partial_zf}{2(1+f)}\Big|_{(x,z)} + tr(G^{-1}\partial_z G)|_{(x,z,\xi)} \right)\,,
        \end{align*}
        where $G(x,z,\xi):= [g^z(x) + \xi\otimes \xi]$.  Thus again by (\ref{Equ_App MinGraph_ODE of g^z}) we have,
        \begin{align}
        \left|(\partial_\xi F^{f^+}|_{(x,z^-,\xi^-)}^{(x,z^+, \xi^+)}) - (\xi^+ - \xi^-)^\sharp \right|\lesssim_n \mbfO(x,z^\pm,\xi^\pm, f^\pm)\,; \label{Equ_App MinGraph_error in div}
        \end{align}
        where 
        \begin{align*}
          \mbfO(x, z^\pm, \xi^\pm, f^\pm) :=\ & \left(\sum_{i\in \{\pm\}} |z^i|+|\xi^i|+ \|f^i\|_{C^2, M_0}\right)\\
          \cdot & \left(|z^+-z^-|+|\xi^+-\xi^-|+ \|f^+-f^-\|_{C^2, M_0}\right). 
        \end{align*}
        Also, since
        \begin{align*}
        \left|tr(G^{-1}\partial_z G)|_{(x,z^-,\xi^-)}^{(x,z^+, \xi^+)} + \left(|A_\Sigma|^2_{g^0} + \Ric_g(\partial_z,\partial_z)\right)\cdot(z^+-z^-)\right| 
        \lesssim_n \mbfO(x, z^\pm, \xi^\pm)\,,
        \end{align*}
        We thus have,
        \begin{align}
        \begin{split}
        & \left|(\partial_z F^{f^+}|_{(x,z^-,\xi^-)}^{(x,z^+, \xi^+)}) - \left[\frac{n\partial_z(f^+-f^-)}{2}\Big|_{(x, 0)} - \left(|A_\Sigma|^2 + \Ric_g(\partial_z, \partial_z)\right)(z^+ - z^-) \right] \right| \\
        \lesssim_n \ & \mbfO(x,z^\pm,\xi^\pm, f^\pm)\,; 
        \end{split}  \label{Equ_App MinGraph_error in 0-order term}
        \end{align}
        (iii) is proved by combining (\ref{Equ_App MinGraph_Min Surf Oper}), (\ref{Equ_App MinGraph_error in div}) and (\ref{Equ_App MinGraph_error in 0-order term}).

        To prove (iv), the second fundamental form of $\Sigma_u$ under $\tilde{g}$ is given by \[
          A_{\Sigma_u, \tilde{g}}(\partial_i \Phi^u, \partial_j \Phi^u) = \langle \nabla^{\tilde{g}}_{\partial_i + u_i \partial_z}(\partial_j + u_j \partial_z), \tilde{\nu}\rangle\,,   \]
        where $\tilde{\nu}= \tilde{\nu}^j\cdot \partial_j + \tilde{\nu}^z\cdot \partial_z$ is the unit normal field of $\Sigma_u$ under $\tilde{g}$, satisfying the pointwise estimate 
        \begin{align}
          \sum_{j=1}^n|\tilde{\nu}^j| + |\tilde{\nu}^z - 1| \leq C_n([\tilde{g}-g]_{C^3, M_0} + |u|+|du|)\,;  \label{Equ_App MinGraph_est normal field under tilde(g)}  
        \end{align}
        Also by (\ref{Equ_App MinGraph_ODE of g^z}), the pull pack metric $\tilde{g}_u:= (\Phi^u)^*\tilde{g}$ satisfies the pointwise estimate 
        \begin{align}
        |\tilde{g}_u - g^0|_{g_0}+ |\nabla_{g^0}\tilde{g}_u|_{g^0} \leq C_n\left( \|\tilde{g}-g\|_{C^3, M_0} + |u|+|du|_{g^0}+ |\nabla^2 u|_{g^0} \right)\,; \label{Equ_App MinGraph_est |tilde(g)_u - g^0|} 
        \end{align}
        And the Christoffel symbols of $\tilde{g}$ and $g$ satisfies 
        \begin{align}
          |\tilde{\Gamma}_{IJ}^K(\mathbf{0}, z)- \Gamma_{IJ}^K(\mathbf{0}, 0)| \leq C_n([\tilde{g} -g]_{C^3, M_0} + |z|), \ \ \ 1\leq I,J,K\leq n+1 \,. \label{Equ_App MinGraph_est Gamma_IJ^K} 
        \end{align}
        Combining (\ref{Equ_App MinGraph_est normal field under tilde(g)}), (\ref{Equ_App MinGraph_est |tilde(g)_u - g^0|}) and (\ref{Equ_App MinGraph_est Gamma_IJ^K}) one proves the estimate on difference of second fundamental forms.
        
         
        To estimate Laplacian, note that $\Delta_{\tilde{g}_u}\psi -\Delta_{g^0}\psi = \diverg_{g^0}(\vec{B_0}) + B_1$, where
        \begin{align*}
        \vec{B_0}(\mathbf{0}) = (\tilde{g}_u^{ij}-g^{ij})\partial_j\psi \cdot\partial_i|_{\mathbf{0}}; &\ & B_1(\mathbf{0}) = \frac12\tilde{g}_u^{ij}\partial_i(\log \det[\tilde{g}_u])\partial_j\psi(\mathbf{0})\,,
        \end{align*}
        and hence by (\ref{Equ_App MinGraph_est |tilde(g)_u - g^0|}), they satisfies the estimate in (iv).
        \end{proof}

    \section{Singular Sheeting Theorem} \label{Sec_App_Sing Cap}
        In this section, we show a singular sheeting theorem for stable minimal hypersurfaces in dimension eight, extending the well-known smooth sheeting theorem \cite[Theorem~1]{schoen_regularity_1981}. This is equivalent to the affirmative answer to a question asked by Ilmanen in dimension eight.

        \begin{Que}[{\cite[Question~4]{ilmanen_strong_1996}}]
            Show that if a stable hypersurface $M$ lies within an $\varepsilon$-neighborhood of another stable hypersurface $N$, then on a smaller ball, $M$ decomposes into disjoint components, each of which is weakly close to $N$ as a Radon measure (that is, constitutes a single layer over $N$).
        \end{Que}

        Let's start with the following multiplicity one Lemma.
        \begin{Lem} \label{Lem_App_Stable MH has multi 1 infty tangent cone}
            Let $\Sigma\subset \RR^{n+1}$ be a connected stable minimal hypersurface with optimal regularity and a varifold tangent cone $m|\mbfC|$ near infinity, i.e. 
            \[
                \lim_{r\searrow 0}\mbfF(|\eta_{1/r}(\Sigma)|, m|\mbfC|) = 0\,.
            \]
            If $\mbfC \subset \RR^{n+1}$ is a regular cone, then $m=1$.
        \end{Lem}
        \begin{proof}
            By \cite{schoen_regularity_1981}, outside a large ball, $\Sigma$ is union of $m$ disjoint graphs over $\mbfC$. Then we can follow the same arguments in \cite{cao_structure_1997} to show that $\Sigma$ has only one end and thus, $m = 1$. For sake of completeness we include a sketch of the proof here.

            Indeed, we only need to deal with the case where $\Sing(\Sigma) \neq \emptyset$. Suppose for contradiction that $\Sigma$ has more than one end. In the construction of a non-constant bounded harmonic function with finite energy in \cite[Lemma~2]{cao_structure_1997}, we now consider $u_i$ solving the following Dirichlet problem on $D_i \setminus \mathbb{B}(\Sing(\Sigma), 1/R_i)$,
            \[
                \begin{cases}
                    \Delta u_i(x) &= 0\\
                    u_i|_{\partial E^{(i)}_1} &= 1\\
                    u_i|_{\partial E^{(i)}_j} &= 0\, (j \neq 1)\\
                    u_i|_{\partial \mathbb{B}(\Sing(\Sigma), 1/R_i)\cap \Sigma} &= 0\,.\\
                \end{cases}
            \]
            The same arguments therein lead to a nontrivial harmonic function $u(x) := \lim_{i \to \infty}u_i(x)$, with
            \begin{equation}
                0 \leq u(x) \leq 1,\quad E(u):=\int_M |\nabla u|^2 dx \leq C_1\,.
            \end{equation}
            Moreover, by standard elliptic estimates, near $p \in \Sing(\Sigma)$, we also have
            \begin{equation}\label{eqn:est_grad_u}
                |\nabla u(x)| \leq \frac{C_1}{\dist(p,x)}\,.
            \end{equation}

            Let $e(u) = |\nabla u|^2$ and $\phi(x)$ be a smooth function compactly supported in $\Sigma$. Then it follows from the proof of \cite[Lemma~2]{schoen_harmonicmapstopology_1976} that
            \[
                \int \phi^2 |\nabla \sqrt{e(u)}|^2 \leq C_2 \int e(u) |\nabla \phi|^2\,.
            \]
            If we take $\phi^2$ to be a cut-off function supported in $\mathbb{B}_{2R} \setminus \mathbb{B}(\Sing(\Sigma), 1/(2R))$ with $\phi^2 \equiv 1$ in $\mathbb{B}_{R} \setminus \mathbb{B}(\Sing(\Sigma), 1/R)$ and linear interpolating in between, then combining with (\ref{eqn:est_grad_u}), we obtain
            \[
                \int_{\mathbb{B}_{R}\cap \Sigma \setminus \mathbb{B}(\Sing(\Sigma), 1/R)} |\nabla \sqrt{e(u)}|^2 \leq \frac{C_2}{R^2}E(u) + C_3 R^4\vol\big(\mathbb{B}(\Sing(\Sigma), 1/R) \cap \Sigma\big)\,.
            \]
            Letting $R \to \infty$, we see that $e(u)$ is a constant on $\Sigma$. It follows from $E(u) < \infty$ that $e(u) \equiv 0$ and thus, $u$ is a constant, contradicting our construction. 
        \end{proof}
        
        \begin{Prop} \label{Prop_Induct_conn stable converg implies multi 1}
            Let $\set{\Sigma_j \subset (\BB_2^8, g_j)}_j$ be a family of two-sided stable minimal hypersurfaces, with 
            \[
                g_j\to g_{\eucl} \text{ in } C^4_{loc}, \quad |\Sigma_j|\ \mbfF\text{-converges to }m|\mbfC| \text{ in } \BB_2\,,
            \]
            where $\mbfC\subset \RR^8$ is a stable minimal hypercone, and $m\geq 1$ is an integer.  
            If all $\Sigma_j$'s are connected in $\BB_1$, then $m = 1$.
        \end{Prop}
        \begin{proof}
            Recall in (\ref{Equ_Pre_Density dicrete for SMC}), the densities of stable minimal hypercones in $\RR^8$ form a discrete set, denoted by \[
                \{\theta_\mbfC(\mathbf{0}): \mbfC \subset \RR^8 \text{ stable minimal hypercone}\} = \{1 = \theta_0 <\theta_1 <\theta_2 < \dots \nearrow +\infty\}.     \]
            We shall prove the lemma by induction on $\theta_l$. 

            Note that when $\theta_\mbfC(\mathbf{0}) = 1$, by volume monotonicity, $\mbfC$ is a hyperplane. Hence by \cite[Theorem~1]{schoen_regularity_1981}, $\Sigma_j$ is a multi-graph over $\mbfC$ in $\BB_1$, and the connectedness of $\Sigma_j\cap \BB_1$ implie that this should be a graph, so $m=1$.
            
            Now fix $l\geq 0$ and suppose the lemma holds for stable minimal hypercones with density $\leq \theta_l$. Consider $\mbfC$ with $\theta_\mbfC(\mathbf{0}) = \theta_{l+1}$ and suppose for contradiction that $m\geq 2$.  By \cite[Theorem~1]{schoen_regularity_1981} again, when $j>>1$,  $\Sigma_j\cap A_{g_j}(\mathbf{0}, 1/2, 1)$ is a disjoint union of $m$ graphs over $\mbfC$. 
            
            Fix an $\epsilon\in (0, 1)$ (To be determined later) and let 
            \[
                r_j := \inf\{s\in (0, 1/16): \exists p\in B^{g_j}(\mathbf{0},1/16) \text{ s.t. } \theta_{|\Sigma_j|}(p, 1) - \theta_{|\Sigma_j|}(p, s) \leq \epsilon \}.   
            \]
            And the multi-graphical property mentioned above implies that $r_j\to 0$ when $j\to \infty$. By Lemma \ref{Lem_Ana on SMC_Quanti Uniqueness of Tang Cone}, we can take $\epsilon = \epsilon(\theta_{l+1}, \theta_l, \mbfC)<<1$, such that there exists $p_j\in B^{g_j}(\mathbf{0},1/4)$ and $\Sigma_j \cap A^{g_j}(p_j, 2r_j, 1)$ is a disjoint union of $m$ graphs over translated cone $\mbfC+p_j$ with graphical functions of small $C^2_*$-norm. Note that $p_j \rightarrow \mathbf{0}$.
            
            Since $m > 1$, by Ilmanen's strong maximum principle \cite[Theorem A]{ilmanen_strong_1996}, we have $r_j >0$.  Note that the rescaled minimal hypersurface $\hat{\Sigma}_j := \eta_{p_j, 4r_j}(\Sigma_j) \hookrightarrow (\BB_{1/(2r_j)}, (4r_j)^{-2}(\eta_{p_j, 4r_j}^{-1})^* g_j)$ is still connected in $\BB_R$ for every $1\leq R \leq 1/2r_j$.  When $j\to \infty$, $|\hat{\Sigma}_j|$ $\mbfF$-subconverges to some $m'|\hat{\Sigma}_\infty|$, where $\hat{\Sigma}_\infty \subset (\RR^8, g_{\eucl})$ is a connected stable minimal hypersurface, and a multi-graph over $\mbfC$ outside $\BB_{1/2}$. By Lemma \ref{Lem_App_Stable MH has multi 1 infty tangent cone}, $\hat{\Sigma}_\infty$ is a single graph over $\mbfC$ outside $\BB_{1/2}$. Therefore, $m' = m \geq 2$ and by the choice of $r_j$, $\hat{\Sigma}_\infty$ is not a translation of a cone. 
            
            Now either 
            \begin{itemize}
                \item $\Sing(\hat{\Sigma}_\infty) = \emptyset$, or
                \item by volume monotonicity and the fact that $\mbfC$ is the tangent cone of $\hat\Sigma$ at infinity, every $\hat{p}\in \Sing(\hat{\Sigma}_\infty)$ has density strictly less than $\theta_{|\hat{\Sigma}_\infty|}(\infty) = \theta_{l+1}$. Hence, $\forall \hat{p}\in \Sing(\hat{\Sigma}_\infty), \ \theta_{|\hat{\Sigma}_\infty|}(\hat p) \leq \theta_l$.
            \end{itemize}
             Moreover, since $\hat{\Sigma}_\infty$ is a single graph over $\mbfC$ outside $\BB_{1/2}$, $\Sing(\hat{\Sigma}_\infty) \subset \BB_1$. We are going to obtain contradictions in either case.
            
            \begin{itemize}
                \item If $\Sing(\hat{\Sigma}_\infty) = \emptyset$, then $\hat\Sigma_j$ is a disjoint union of $m$ graphs over $\hat \Sigma_\infty$, which contradicts the connectedness of $\hat \Sigma_j \cap \BB_1$.
                \item If $\Sing(\hat{\Sigma}_\infty) \neq \emptyset$, then there should be only finitely many points, i.e.
                      \[
                          \Sing(\hat{\Sigma}_\infty) = \set{p_i}^L_{i = 1} \subset \BB_1\,.
                      \]
                      Hence, there exists a strictly increasing sequence $\set{j_k} \subset \mathbb{N}$ satisfying the following property.
                      
                      For every $k$, there exists a radius $r_k \in (0, 1/k)$, such that 
                      \begin{enumerate}
                          \item $\hat\Sigma_{j_k} \cap \BB_1$ is a disjoint union of $m$ graphs over $\hat\Sigma_{\infty} \cap \BB_1$ outside $\BB(\Sing(\hat{\Sigma}_\infty), {r_k})$;
                          \item For each $i \in \set{1, 2, \cdots, L}$, $\hat\Sigma_{j_k} \cap \partial \mathbb{B}(p_i, r_k)$ is a disjoint union of $m$ connected hypersurfaces;
                          \item For each $i \in \set{1, 2, \cdots, L}$, $\mbfF_{\BB_1}(|\eta_{p_i, r_k}(\hat\Sigma_{j_k})|,m|\eta_{p_i, r_k}(\hat\Sigma_\infty)|) < 1/k$;
                          \item There exists an $i_k \in \set{1, 2, \cdots, L}$, such that $\hat\Sigma_{j_k} \cap \BB(p_{i_k}, r_k)$ has a connected component $\hat\Gamma_{j_k}$ with
                                \[
                                    \mbfF_{\BB_1}(|\eta_{p_{i_k}, r_k}(\hat\Gamma_{j_k})|,m_k|\eta_{p_{i_k}, r_k}(\hat\Sigma_\infty)|) < 1/k\,,
                                \]
                                for some $m_k \in \set{2, \cdots, m}$. In particular, since $L$ and $m$ are finite, by taking a subsequence, we may replace $i_k$ and $m_k$ by some fixed $i_0$ and $m_0 \geq 2$ independent of $k$.
                      \end{enumerate}
                      
                      The last property follows from the connectedness of $\hat \Sigma_j \cap \BB_1$ and the other three properties. Moreover, $\eta_{p_{i_0}, r_k}(\hat\Sigma_\infty)$ would $\bfF$-converge to its tangent cone $\mbfC_{p_{i_0}}(\hat\Sigma_\infty)$, and thus,
                      \[
                          \mbfF_{\BB_1}(|\eta_{p_{i_0}, r_k}(\hat\Gamma_{j_k})|,m_0|\mbfC_{p_{i_0}}(\hat\Sigma_\infty)|) \to 0\,.
                      \]
                      This violates our inductive hypothesis as $\theta_{\mbfC_{p_{i_0}}(\hat\Sigma_\infty)}(\mathbf{0}) \leq \theta_l$.
            \end{itemize}
            
            In summary, we obtain that $m = 1$ for any stable hypercone in $\mathbb{R}^8$.
        \end{proof}
        
        \begin{Cor} \label{Cor_Conv of SMH w multi}
            Let $g_j$ be a family of $C^4$ metrics on a closed manifold $M^8$, $\Sigma_j \subset (M, g_j)$ be connected stable minimal hypersurfaces, $1\leq j\leq \infty$. Suppose that when $j\to \infty$, 
            \[ 
                g_j\to g_\infty\text{ in }C^4,\quad |\Sigma_j|\ \mbfF\text{-converges to } m|\Sigma_\infty|\,,
            \]
            for some integer $m\geq 1$.  Then 
            \begin{enumerate} [(i)]
                \item If $\Sigma_\infty$ is two-sided, then so are $\Sigma_j$ when $j>>1$, and $m = 1$.
                \item If $\Sigma_\infty$ is one-sided, then either $m = 1$ and $\Sigma_j$ are one-sided, or $m = 2$ and $\Sigma_j$ are two-sided, when $j>>1$.
            \end{enumerate}
        \end{Cor}
        \begin{proof}
            \item
            \paragraph{(i)} By \cite{schoen_regularity_1981}, for $j$ sufficiently large, away from $\Sing(\Sigma_\infty)$, every $\Sigma_j$ is a disjoint union of $m$ graphs over $\Sigma_\infty$. As the part in the proof of Proposition \ref{Prop_Induct_conn stable converg implies multi 1} dealing with $\hat \Sigma_\infty$, the multiplicty $m$ should be $1$. Moreover, since every LSMH in a contractible space is two-sided, $\Sigma_j \cap B^{g_\infty}(\Sing(\Sigma_\infty), r_0)$ is two-sided for small enough $r_0$. When $j>>1$, a unit normal field over $\Sigma_\infty$ induces a unit normal field on $\Sigma_j$ away from $\Sing(\Sigma_\infty)$, and by the previous statement, can extend globally to $\Sigma_j$. This proves that $\Sigma_j$ is two-sided for $j >> 1$.
            
            \paragraph{(ii)} Let $\pi: \hat{M} \to M$ be the double cover associated to $\Sigma_\infty$ such that $\pi^{-1}(\Sigma_\infty)$ is connected and two-sided in $\hat{M}$. Then $\pi^{-1}(\Sigma_j)$ are either connected or have two connected components, which corresponding to $\Sigma_j$ being one-sided or two-sided. Since $|\pi^{-1}(\Sigma_j)|$ $\mbfF$-converges to $m|\hat{\Sigma}_\infty|$, by (i), we have $m = 1$ or $2$. And since $\Sigma_j$ is connected and $\pi$ is a double cover, for sufficiently large $j$, we must have that $\Sigma_j$ is two-sided provided $m =2$, and one-sided provided $m = 1$.
        \end{proof}

        \begin{Thm}[Singular sheeting theorem]
            Given a closed Riemannian manifold $(M^8, g)$ and a connected stable two-sided minimal hypersurface $\Sigma$, for any $\Lambda>0$ and $\delta > 0$, there exists a constant $\varepsilon = \varepsilon(M, g, \Sigma, \Lambda, \delta) > 0$ with the following property. For any stable minimal hypersurface $\Sigma'$ in $(M, g)$ with
            \[
                \scH^7(\Sigma') \leq \Lambda,\quad \dist_H(\Sigma', \Sigma) \leq \varepsilon\,,
            \]
            then $\Sigma'$ can be decomposed into finitely many disjoint minimal hypersurfaces $\set{\Sigma'_i}$ with
            \[
                \mathbf{F}(|\Sigma|, |\Sigma'_i|) < \delta\,.
            \]
        \end{Thm}
        \begin{proof}
            We take $\set{\Sigma'_i}$ to be the disjoint union of connected components of $\Sigma'$, each of which is obviously a stable minimal hypersurface. Since $\scH^7(\Sigma') \leq \Lambda$, the collection is finite. 

            Let's suppose for contradiction that there exists a sequence of connected stable minimal hypersurfaces $\set{\Sigma''_j}^\infty_{j = 1}$ with 
            \[
                \Sigma''_j \subset B^g(\Sigma, 1/j), \quad \scH^7(\Sigma''_j) \leq \Lambda\,,
            \]
            but
            \[
                \mathbf{F}(|\Sigma|, |\Sigma''_j|) \geq \delta\,.
            \]
            By \cite{schoen_regularity_1981} again, we know that $|\Sigma''_j|$ $\mathbf{F}$-subconverges to some $m|\Sigma''_\infty|$ whose support $
            \overline\Sigma''_\infty$ is a stable minimal hypersurface. Moreover, it follows from $\Sigma''_j \subset B^g(\Sigma, 1/j)$ that 
            \[
                \overline\Sigma''_\infty \subset \overline\Sigma\,,
            \]
            which, by the constancy theorem, indicates that $\Sigma''_\infty = \Sigma$.

            By Corollary \ref{Cor_Conv of SMH w multi}(i), $m = 1$, which contradicts our assumption that $\mathbf{F}(|\Sigma|, |\Sigma''_j|) \geq \delta$.
        \end{proof}

        \begin{Lem} \label{Lem_App_Disjt MH near cone}
            Let $\mbfC\subset \RR^8$ be a stable minimal hypercone, and $\{g_j\}_{j\geq 1}$ be a sequence of Riemannian metrics on $\BB_2$ which $C^4_{loc}$-converges to $g_{\eucl}$. For each $j\geq 1$, let $\Sigma_j^{(1)}, \Sigma_j^{(2)}$ be two disjoint stable minimal hypersurfaces in $(\BB_2, g_j)$, and such that when $j\to \infty$, $|\Sigma_j^{(i)}|$ $\mbfF$-converge to $|\mbfC|$, $i=1,2$.
            
            Then for $j>>1$, at least one of the following holds. 
            \begin{enumerate} [(a)]
                \item At least one of the $\Sigma_j^{(i)}$'s has no singularity in $\BB_1$;
                \item For every $i=1, 2$ and every $p\in \Sing(\Sigma_j^{(i)})\cap \BB_1$, we have \[
                          \theta_{|\Sigma_j^{(i)}|}(p) < \theta_\mbfC(\mathbf{0}).   \]
            \end{enumerate}
        \end{Lem}
        \begin{proof}
            By passing to a subsequence, suppose WLOG that (b) doesn't hold for $j\geq 1$. By the upper-semi continuity of the density function at singular points under $\mbfF$-convergence, suppose WLOG that $p_j\in \Sing(\Sigma_j^{(1)})\cap \BB_1$ has $\theta_{|\Sigma_j^{(1)}|}(p_j) = \theta_\mbfC(\mathbf{0})$.  The goal is to show that under these assumptions, $\Sigma_j^{(2)}$ is smooth for $j>>1$. 
            
            First note that by the volume monotonicity, 
            \[
                \limsup_{j\to \infty} \left(\theta_{|\Sigma_j^{(1)}|}(p_j, 1) - \theta_{|\Sigma_j^{(1)}|}(p_j)\right) = \theta_\mbfC(\mathbf{0}, 1) - \theta_\mbfC(\mathbf{0}) = 0.   
            \]
            Hence by Lemma \ref{Lem_Ana on SMC_Quanti Uniqueness of Tang Cone}, for every sequence $\set{s_j\searrow 0}$, $|\eta_{p_j, s_j}(\Sigma_j^{(1)})|$ $\mbfF$-converges to $|\mbfC|$ as $j\to \infty$. 
            
            Let $r_j$ be the infimum among all $s\in (0, 1)$ such that \[
                \Sigma_j^{(2)}\cap A^{g_j}(p_j, s, 1) = \graph_{\Sigma_j^{(1)}}(u_j)\cap A^{g_j}(p_j, s, 1),   \]
            for some $u_j\in C^2(\Sigma_j^{(1)})$ satisfying $\|u_j\|_{C^2_*}\leq 1$. Since $\Sigma_j^{(i)}$ are disjoint, and so are their closures by Ilmanen's strong maximum principle \cite{ilmanen_strong_1996}, we know that $r_j>0$.  Also by Allard's regularity theorem \cite{allard_first_1972}, $r_j \to 0$ when $j\to \infty$. And by the definition of $r_j$, for $j>>1$.
            \begin{align}
                \Sing(\Sigma_j^{(2)})\cap A^{g_j}(p_j, 2r_j, 3/2) = \emptyset.  \label{Equ_App_Sing(Sigma_j^2) dsjt A(2r_j, 1)}
            \end{align}
            
            Consider the blow-up sequence $\hat{\Sigma}_j^{(i)}:= \eta_{p_j, r_j}(\Sigma_j^{(i)})$.  By the discussion above, when $j\to \infty$, $|\hat{\Sigma}_j^{(1)}|$ $\mbfF$-converges to $|\mbfC|$, and $\hat{\Sigma}_j^{(2)}$ $\mbfF$-subconverges to some stable minimal hypersurface $\hat{\Sigma}_\infty\subset (\RR^8, g_{\eucl})$, which is different from $\mbfC$ by the choice of $r_j$. Therefore, $\hat \Sigma_\infty$ is disjoint from $\mbfC$ again by the strong maximum principle. Finally, by \cite{hardtAreaMinimizingHypersurfaces1985, WangZH20_Deform}, $\Sing(\hat{\Sigma}_\infty) = \emptyset$, and thus, by Allard's regularity theorem \cite{allard_first_1972} again,
            \[
                \Sing(\Sigma_j^{(2)})\cap B^{g_j}(p_j, 3r_j) = \emptyset\,,
            \] for $j>>1$.  Together with (\ref{Equ_App_Sing(Sigma_j^2) dsjt A(2r_j, 1)}), we obtain the smoothness of $\Sigma^{(2)}_j$ for $j$ sufficiently large. 
        \end{proof}
        
        \begin{proof}[Proof of Theorem \ref{Thm_Pre_SCAP under converg w multip}.]
            For all cones $\mbfC\in \cC$ which are not hyperplanes, we shall define $\SCAP$ as follows,
            \begin{align}
                \SCAP(\mbfC):= 1 + \sup\{\limsup_{j\to \infty}\SCAP(\Sigma_j \cap \BB_1, \BB_1, g_j)\},  \label{Equ_Def_SCAP}
            \end{align}
            where the supremum is taken over all sequence of pairs $(\Sigma_j, g_j)$ satisfying, 
            \begin{enumerate} [(i)]
                \item $g_j$ is a Riemannian metric on $\BB_2$, $\Sigma_j \subset (\BB_2, g_j)$ is a stable minimal hypersurface;
                \item When $j\to \infty$, $g_j\to g_{\eucl}$ in $C^4_{loc}$, and $|\Sigma_j|$ $\mbfF$-converges to $|\mbfC|$;
                \item For every $p_j\in \Sing(\Sigma_j)\cap \BB_1$, $\theta_{|\Sigma_j|}(p_j) < \theta_{\mbfC}(\mathbf{0})$.
            \end{enumerate}
            And if there is no such a sequence of pairs for $\mbfC$, then we simply define $\SCAP(\mbfC) = 1$.
            
            First note that by (\ref{Equ_Pre_Density dicrete for SMC}), the densities of stable minimal hypercones form a discrete subset of $\RR$, denoted by $\{1=\theta_0<\theta_1<\theta_2<\dots \nearrow +\infty\}$. We denote \[
                \scC(l):= \{\mbfC\in \cC: \theta_\mbfC(\mathbf{0}) = \theta_l\}.   \]
            By the discussion in Section \ref{Sec_Analysis on SMC}, $\scC(l)$ is $\mbfF$-compact for every $l\geq 1$. We also denote $\scC(\leq l):= \bigcup_{1\leq j\leq l}\scC(l)$.
            
            Now, (\ref{Equ_Def_SCAP}) should be viewed as an inductive definition on $\theta_l$. For $\mbfC \in \scC(1)$, we can define $\SCAP(\mbfC) = 1$. Then for every $l\geq 1$, once $\SCAP$ is defined for cones $\mbfC\in \scC(\leq l)$, we can define $\SCAP(\Sigma, N, g)$ following (\ref{Equ_App_SCap for 1 and 2-sided}) for every triple $(\Sigma, N, g)\in \scM$ such that $\forall p\in \Sing(\Sigma)$, $\mbfC_p\Sigma \in \scC(\leq l)$.  And then we can define $\SCAP$ for cones $\mbfC\in \scC(l+1)$ following (\ref{Equ_Def_SCAP}) above.
            
            Under this inductive definition, we first verify inductively that 
            \[
                \SCAP(\mbfC)<+\infty
            \] for every $\mbfC\in \cC$. In fact, for every $l\geq 1$, suppose this has been established for cones in $\scC(\leq l)$. And by a diagonal argument, we can see that $\SCAP$ is upper semi-continuous on $\scC(l)$. Then it follows from the  $\mbfF$-compactness of $\scC(l)$ that
            \[
                \sup_{\mbfC\in \scC(\leq l)} \SCAP(\mbfC) < +\infty\,.
            \]
            Now consider a cone $\mbfC\in \scC(l+1)$, and let $(\Sigma_j, \BB_2, g_j)$ be a maximizing sequence of (\ref{Equ_Def_SCAP}) satisfying (i)-(iii) above. Then 
            \begin{align*}
                \SCAP(\mbfC) & \leq 1 + (\limsup_{j\to \infty}\# \Sing(\Sigma_j)\cap \BB_1 )\cdot \sup_{\mbfC'\in \scC(\leq l)} \SCAP(\mbfC') \\
                             & \leq 1 + \SCap(\mbfC)\cdot \sup_{\mbfC'\in \scC(\leq l)} \SCAP(\mbfC')                                         \\
                             & < +\infty,  
            \end{align*}
            where the second and the last inequality follow from Theorem \ref{Thm_Pre_SCap exist}.\\

            \noindent \textit{Proof of (iii) in Theorem \ref{Thm_Pre_SCAP under converg w multip}.} For any $p \in \Sigma_\infty$, there exists a radius $r_p > 0$, such that $\Sigma_\infty \setminus B^{g_\infty}(p, r_p)$ has index $I$. Since $\ind(\Sigma_j) = I$, by Sharp's compactness, for $j$ large enough, $\Sigma_j \setminus B^{g_\infty}(p, r_p)$ also has index $I$ and thus $\Sigma_j \cap B^{g_\infty}(p, r_p)$ is stable. 
            
            Therefore, by rescaling, it suffices to prove the inequality when $\Sigma_j$'s are stable and $(\Sigma_\infty, N, g_\infty) = (\mbfC, \BB_2, g_{\eucl})$ for some cone $\mbfC\in \cC$. By the definition (\ref{Equ_Def_SCAP}), one only needs to verify the case when there exists some $p_j\in \Sing(\Sigma_j)\cap \BB_1$ such that $\theta_{|\Sigma_j|}(p_j) = \theta_{\mbfC}(\mathbf{0})$. By the volume monotonicity, 
            \[
                \limsup_{j\to \infty} \theta_{|\Sigma_j|}(p_j, 1) - \theta_{|\Sigma_j|}(p_j) = \theta_{\mbfC}(\mathbf{0}, 1) - \theta_\mbfC(\mathbf{0}) = 0.   
            \]
            Hence by Lemma \ref{Lem_Ana on SMC_Quanti Uniqueness of Tang Cone} and Corollary \ref{Cor_Converg in all Scales}, $\Sing(\Sigma_j)\cap \BB_1 = \{p_j\}$, and $\mbfC_{p_j}(\Sigma_j)$ $\mbfF$-converges to $\mbfC$.  It follows from the upper semi-continuity of $\SCAP$ on $\scC(l)$ that
            \[
                \limsup_{j\to\infty}\SCAP(\Sigma_j, \BB_1, g_j) = \limsup_{j\to \infty} \SCAP(\mbfC_{p_j}(\Sigma_j)) \leq \SCAP(\mbfC)\,.  
            \]$\,$\\
            
            \noindent \textit{Proof of (iv) in Theorem \ref{Thm_Pre_SCAP under converg w multip}.}
            By the connectedness of $\Sigma_j$ and Corollary \ref{Cor_Conv of SMH w multi}, we know that $m = 2$ and $\Sigma_\infty$ is one-sided while $\Sigma_j$ are all two-sided for $j>>1$.  Let $\pi:\hat{N} \to N$ be the double cover such that $\hat{\Sigma}_\infty := \pi^{-1}(\Sigma_\infty)$ is connected and two-sided. Then for $j>>1$, $\pi^{-1}(\Sigma_j)$ has two connected components $\Sigma_j^\pm$, each isometric to $\Sigma_j$.  Moreover, $\hat{\Sigma}_\infty$ and $\Sigma_j^\pm$ are stable minimal hypersurfaces in $(\hat{N}, \hat{g}_j:= \pi^* g_j)$, and $|\Sigma^\pm_j|$ $\mbfF$-converges to $|\hat{\Sigma}_\infty|$ when $j\to \infty$. 
            Let $p\in \Sing(\hat{\Sigma}_\infty)$, and $W_p$ be a small neighborhood of $p$ such that $\Sing(\hat{\Sigma}_\infty)\cap \Clos(W_p) = \{p\}$. By definition (\ref{Equ_Def_SCAP}) and Lemma \ref{Lem_App_Disjt MH near cone}, we have
            \begin{equation}\label{Eqn:SCAP_drop}
                \limsup_{j\to \infty} \min\{\SCAP(\Sigma_j^\pm, W_p, \hat{g}_j)\}  \leq \SCAP(\hat{\Sigma}_\infty, W_p, \hat{g}_\infty) - 1\,.
            \end{equation}
            WLOG, let's assume that $\Sigma_j^+$ corresponds to the one with smaller $\SCAP$.  Therefore, by Theorem \ref{Thm_Pre_SCAP under converg w multip} (iii) and (\ref{Eqn:SCAP_drop}), we obtain
            \begin{align*}
                \SCAP(\Sigma_\infty, N, g_\infty) & = \SCAP(\hat{\Sigma}_\infty, \hat{N}, \hat{g}_\infty)                                                                        \\
                                                  & \geq \limsup_{j\to \infty} \SCAP(\Sigma_j^+, \hat{N}\setminus \Clos(W_p), \hat{g}_j) + \SCAP(\Sigma_j^+, W_p, \hat{g}_j) + 1 \\
                                                  & = \limsup_{j\to \infty} \SCAP(\Sigma_j, N, g_j) + 1.    
            \end{align*}
        \end{proof}
    
    \section{Regular Deformation Theorem}  \label{Sec_App_Reg Deform Thm}
        The goal for this section is to prove Theorem \ref{Thm_Reg Deform Thm_OneTwo sided}. We first briefly review the global analysis on singular LSMH introduced in \cite{WangZH20_Deform}.
        
        Let $k\geq 4$, $\alpha\in (0, 1)$ and $\Sigma$ be a two-sided LSMH in a closed manifold $(M^8, g)$ with a unit normal field $\nu$.  
        
        The function space for spectral theory of the Jacobi operator $L_{\Sigma, g}$ is $\scB(\Sigma)$, defined as follows. By \cite[Lemma 3.1]{WangZH20_Deform}, there exists a constant $C_{\Sigma, g}>0$ such that,
        \[
            \|\phi\|_{\scB(\Sigma)}^2:= Q_{\Sigma, g}(\phi, \phi) + C_{\Sigma, g}\|\phi\|_{L^2(\|\Sigma\|_g)}^2 \geq \|\phi\|_{L^2(\|\Sigma\|_g)}^2, \ \ \forall \phi\in C_c^1(\Sigma),   
        \]
        where $Q_{\Sigma, g}$ is the quadratic form associated to the second variation of area of $\Sigma$, defined in (\ref{Equ_Pre_Stablity Ineq for MH}).  Hence, \[
            \scB(\Sigma):= \overline{C_c^1(\Sigma)}^{\|\cdot\|_{\scB}},   \]
        is a well-defined Hilbert space which is naturally included in $L^2(\|\Sigma\|_g)$ and naturally contains $W^{1,2}(\Sigma)$.  By \cite[Lemma 3.2]{WangZH20_Deform}, every $\phi\in \scB(\Sigma)$ is locally $W^{1,2}$ along $\Sigma$; By \cite[Proposition 3.5 \& Lemma 3.9]{WangZH20_Deform}, $\scB(\Sigma)\hookrightarrow L^2(\|\Sigma\|_g)$ is a compact embedding. Hence we have the following spectrum decomposition of $L^2$.
        \begin{Lem} [{\cite[Proposition 3.5, Corollary 3.7]{WangZH20_Deform}}] \label{Lem_Spectrum Decomp of Jac operator}
            There exists a sequence
            \[
                \lambda_1<\lambda_2< \lambda_3< ... \nearrow +\infty
            \] and a corresponding sequence of finite dimensional pairwise $L^2$-orthogonal linear subspaces $\{E_j\}_{j\geq 1}$ of $\scB(\Sigma)\cap C^2(\Sigma)$ such that \[
                -L_{\Sigma, g}\phi = \lambda_j\phi, \ \ \ \forall \phi\in E_j,   \]
            and that 
            \begin{align*}
                L^2(\|\Sigma\|_g) = \overline{\bigoplus_{j\geq 1}E_j}^{L^2}; & \  & \scB(\Sigma) = \overline{\bigoplus_{j\geq 1}E_j}^{\scB}.
            \end{align*}
            Moreover, $\dim E_1 = 1$ and $E_1$ is generated by a positive function on $\Sigma$. And the index of $\Sigma$ defined in (\ref{Equ_Pre_Def of Ind(Sigma)}) coincides with the index of $-L_{\Sigma, g}$:  
            \[
                \ind(\Sigma) = \sum_{\lambda_j<0} \dim E_j\,.   
            \]
        \end{Lem}
        $\Sigma$ is called \textbf{non-degenerate} if the spectrum $\{\lambda_j\}_{j \geq 1}$ of $-L_{\Sigma, g}$ doesn't contain $0$.  This definition is compatible with the non-degeneracy for smooth minimal hypersurfaces.
        Also, by Lemma \ref{Lem_Pre_Asymp Rate takes value in Gamma(C) and int by part}, every slower growth Jacobi field falls in $W^{1,2}(\Sigma)\subset \scB(\Sigma)$.  Hence, $\Sigma$ being non-degenerate implies semi-nondegeneracy of $\Sigma$.
        
        \cite{WangZH20_Deform} proved the following results on induced Jacobi field for a sequence of converging LSMH.
        \begin{Thm} [{\cite[Theorem 4.2]{WangZH20_Deform}}] \label{Thm_App_Wang20 Induced Jac field}
            Suppose $\set{\Sigma_j\subset (M, g_j)}_j$ be a family of two-sided LSMH satisfying (a), (b) in Theorem \ref{Thm_Reg Deform Thm_OneTwo sided} with $f\in C^{k, \alpha}(M)$ such that $\nu(f)$ does not identically vanishes along $\Sigma$. 
            Then there exists a nontrivial induced Jacobi field $u\in C^2_{loc}(\Sigma)$ satisfying \[
                L_{\Sigma, g}u = c\nu(f),   \]
            for some $c\geq 0$. And for every $p\in \Sing(\Sigma)$, $\cA\cR_p(u)\geq \gamma_1^-(\mbfC_p\Sigma)$.
            
            Moreover, if for some $p\in \Sing(\Sigma)$, $\cA\cR_p(u)\in \{\gamma_1^\pm(\mbfC_p\Sigma)\}$, then for some sufficiently small neighborhood $U_p\subset M$ of $p$, we have $\Sing(\Sigma_j)\cap U_p = \emptyset$ for infinitely many $j>>1$.
        \end{Thm}
        
        To apply this theorem, in \cite{liwang2020generic}, it is shown that if $\Sigma$ is non-degenerate, then there is an open dense way of choosing $f$ in Theorem \ref{Thm_App_Wang20 Induced Jac field} such that the corresponding $u$ satisfies $\cA\cR_p(u)\in \{\gamma_1^\pm(\mbfC_p\Sigma)\}$ for at least one $p\in \Sing(\Sigma)$.  Here, we shall generalize this to semi-nondegenerate $\Sigma$, using the same argument as in \cite{WangZH20_Deform, liwang2020generic}.
        \begin{Lem}  \label{Lem_App_scF^k,alpha}
            Suppose $\Sigma$ is a two-sided LSMH. Then the following $\scF^{k,\alpha}$ is open and dense in $C^{k,\alpha}(M)$.
            When $\Sigma$ is non-degenerate, \[
                \scF^{k,\alpha} := \{f\in C^{k,\alpha}(M): L_{\Sigma, g}u = \nu(f) \text{ has no slower growth solution }u\};   \]
            When $\Sigma$ is degenerate, \[
                \scF^{k,\alpha}:= \{f\in C^{k,\alpha}(M): L_{\Sigma, g}u = \nu(f) \text{ has no solution }u\in \scB(\Sigma)\}.  \]
            
            Moreover, if $G < \text{Diff}(M)$ is a finite group which acts fixed-point-freely on $M$, such that $(\Sigma, M, g)$ is $G$-invariant in the sense of Definition \ref{Def_G-inv}, then $\scF^{k,\alpha}\cap C^{k,\alpha}(M)^G$ is also dense in $C^{k,\alpha}(M)^G$.
        \end{Lem}
        \begin{proof}
            When $\Sigma$ is degenerate, by integration by parts, 
            \[
                C^{k,\alpha}(M)\setminus \scF^{k,\alpha} = \{f\in C^{k,\alpha}(M): \nu(f)\in (\Ker L_{\Sigma, g})^\perp\},   
            \]
            which is a closed linear subspace of $C^{k,\alpha}(M)$ and does not contain any $C^{k,\alpha}(M)^G$. Hence $\scF^{k,\alpha}\cap C^{k,\alpha}(M)^G$ is open and dense in $C^{k,\alpha}(M)^G$.
            
            When $\Sigma$ is non-degenerate, openness of $\scF^{k,\alpha}$ follows from \cite[Lemma 3.21]{WangZH20_Deform}. And hence for a general finite subgroup $G < \text{Diff}(M)$, $\scF^{k,\alpha}\cap C^{k,\alpha}(M)^G$ is relatively open in $C^{k,\alpha}(M)^G$.
            
            To prove the denseness for general $G$, we also apply \cite[Lemma 3.21]{WangZH20_Deform} to see that for every $f_1\in \scF^{k,\alpha}\cap C^{k,\alpha}(M)^G$, every $c\neq 0$ and every $f_2\in C^{k,\alpha}(M)^G\setminus \scF^{k,\alpha}$, we have $cf_1+f_2\in \scF^{k,\alpha}\cap C^{k,\alpha}(M)^G$. Hence to prove denseness, it suffices to verify that $\scF^{k,\alpha}\cap C^{k,\alpha}(M)^G \neq \emptyset$. To see this, let $p\in \Sing(\Sigma)$ and $U_p\ni p$ be a small neighborhood in $M$ such that $\gamma(U_p)\cap U_p = \emptyset$, $\forall \gamma\in G\setminus \{\mathrm{Id}_M\}$. Take $u\in C_{loc}^\infty(U_p)\cap \scB(\Sigma)$ such that $u>0$ near $p$, $\spt(u)\subset \subset U_p$ and $L_{\Sigma, g} u$ vanishes near $\{p\}$. Then take $f\in C_c^\infty(U_p)$ such that $\nu(f) = L_{\Sigma, g} u$ in $U_p$.  Then extend $f$ and $u$ to a $G$-invariant smooth function on $M$ and $\Sigma$. By Lemma \ref{Lem_Spectrum Decomp of Jac operator}, $u\in \scB(\Sigma)$ is the unique solution of $L_{\Sigma, g} u = \nu(f)$. Then by Lemma \ref{Lem_Pre_Asymp Rate takes value in Gamma(C) and int by part} (ii), $f\in \scF^{k,\alpha}\cap C^{k,\alpha}(M)^G$. This finishes the proof of denseness.
        \end{proof}
        
        \begin{Rem}
            If $\Sigma$ is semi-nondegenerate, by adapting the proof of Lemma \ref{Lem_Cptness for Slower Growth Jac Fields}, one can also show that 
            \[
                \{f\in C^{k,\alpha}(M): L_{\Sigma, g}u = \nu(f) \text{ has no solution }u \text{ of slower growth} \},   
            \]
            is open and dense in $C^{k,\alpha}(M)$.
        \end{Rem}
        
        \begin{proof}[Proof of Theorem \ref{Thm_Reg Deform Thm_OneTwo sided}.]
            First assume that $\Sigma$ is two-sided. Let $\scF^{k,\alpha}$ be given by Lemma \ref{Lem_App_scF^k,alpha}. Also let $f\in \scF^{k,\alpha}$ and $\Sigma_j, g_j$ be given by Theorem \ref{Thm_Reg Deform Thm_OneTwo sided} (a), (b). Then by Theorem \ref{Thm_App_Wang20 Induced Jac field}, the induced Jacobi field $u \not\equiv 0$ satisfies $L_{\Sigma, g}u = c\nu(f)$ for some $c\geq 0$. Hence either $c\neq 0$, or $c = 0$ and thus $u\in \Ker^+L_{\Sigma, g}$. By definition of $\scF^{k,\alpha}$ and semi-nondegeneracy of $\Sigma$, we know that in both cases, $u$ is not a function of slower growth on $\Sigma$, in other words, there exists $p\in \Sing(\Sigma)$ such that $\cA\cR_p(u)<\gamma_2^+(\mbfC_p\Sigma)$.  Hence by Lemma \ref{Lem_Pre_Asymp Rate takes value in Gamma(C) and int by part}, $\cA\cR_p(u)\in \set{\gamma_1^\pm(\mbfC_p\Sigma)}$ and Theorem \ref{Thm_App_Wang20 Induced Jac field} applies to conclude Theorem \ref{Thm_Reg Deform Thm_OneTwo sided}.
            
            When $\Sigma$ is one-sided, let $\hat{\Sigma}\subset (\hat{M}, \hat{g})$ be its two-sided double cover. Then using Lemma \ref{Lem_App_scF^k,alpha} with $G = \ZZ_2 < \text{Diff}(\hat{M})$ to be the deck transformation, the same argument as above proves Theorem \ref{Thm_Reg Deform Thm_OneTwo sided}.
        \end{proof}
        
        \bibliography{reference}
\end{document}